\documentclass[11pt]{amsart}

\usepackage{enumerate,url,amssymb, mathrsfs, mathabx, esint, xcolor}
\usepackage{graphicx, comment}
\usepackage[hidelinks, bookmarksdepth=3]{hyperref}
\usepackage{tikz-cd}

\usetikzlibrary{decorations.markings}
\tikzset{degil/.style={
            decoration={markings,
            mark= at position 0.5 with {
                  \node[transform shape] (tempnode) {$\bigtimes$};
                  }
              },
              postaction={decorate}
}
}

\newtheorem{theorem}{Theorem}[section]
\newtheorem{lemma}[theorem]{Lemma}
\newtheorem*{lemma*}{Lemma}
\newtheorem{proposition}[theorem]{Proposition}
\newtheorem{corollary}[theorem]{Corollary}

\newcommand{\chiinf}{%
  \mathrel{\raisebox{1.1pt}{$\chi$}}^\infty%
}

\theoremstyle{definition}
\newtheorem{definition}[theorem]{Definition}
\newtheorem{example}[theorem]{Example}

\theoremstyle{remark}
\newtheorem{remark}[theorem]{Remark}

\numberwithin{equation}{section}

\newcommand{\C}{\mathbb{C}}
\newcommand{\D}{\partial}
\newcommand{\DD}{\mathbb{D}}

\newcommand{\R}{\mathbb{R}}

\DeclareMathOperator{\dist}{dist}

\DeclareMathOperator{\loc}{loc}

\def\XXint#1#2#3{{\setbox0=\hbox{$#1{#2#3}{\int}$}
\vcenter{\hbox{$#2#3$}}\kern-.5\wd0}}

\def\le{\leqslant}
\def\ge{\geqslant}

\renewcommand{\Re}{\operatorname{Re}}

\def \D{\textup{D}}
\def \mb{\mathbb}
\def \tp{\textup}
\def \p{\partial}

\newcommand{\N}{\mathbb{N}}
\newcommand{\M}{\mathbb{R}^{2 \times 2}}
\newcommand{\WW}{\mathrm{W}}
\newcommand{\LL}{\mathrm{L}}
\newcommand{\CC}{\mathrm{C}}
\newcommand{\Meas}{\mathscr{M}}
\newcommand{\dd}{\mathrm{d}}
\newcommand{\qc}{\mathrm{qc}}
\newcommand{\Id}{\mathrm{Id}}
\newcommand{\B}{\mathbf{B}}
\newcommand{\weak}{\rightharpoonup}

\newcommand{\eR}{\overline{\mathbb{R}}}

\definecolor{green}{RGB}{40,160,90}
\definecolor{orange}{RGB}{180,80,0}

\begin{document}
\baselineskip5mm
\title[Local Burkholder functional and quasiconvexity]{The Local Burkholder functional, quasiconvexity and \\
Geometric Function Theory} 

\author{K.~Astala}
\address{Department of Mathematics and Statistics, University of Helsinki, Finland}
\email{kari.astala@helsinki.fi}

\author{D. ~Faraco}
\address{Departamento de Matematicas, Universidad Autonoma de Madrid,  Spain and ICMAT CSIC, Madrid, Spain}
\email{daniel.faraco@uam.es}

\author{A.~Guerra}
\address{Institute for Theoretical Studies ETH-ITS, Z\"urich, Switzerland}
\email{andre.guerra@eth-its.ethz.ch}

\author{A.~Koski}
\address{Department of Mathematics and Systems Analysis, Aalto University, Espoo, Finland}
\email{aleksis.koski@aalto.fi}

\author{J.~Kristensen}
\address{Mathematical Institute, University of Oxford, United Kingdom}
\email{kristens@maths.ox.ac.uk}

\begin{abstract}  
We show that the local Burkholder functional ${\mathcal B}_K $ is quasiconvex. In the limit of $p$ going to $2$ we find a class
of non-polyconvex functionals which are quasiconvex on the set of matrices with positive determinant. 

In order  to prove the validity of lower semicontinuity arguments in this setting, we show that  the Burkholder functionals satisfy
a sharp extension of the classical function theoretic area formula. As a corollary, in addition to functionals in geometric function theory,
one finds new classes of non-polyconvex functionals, degenerating as the determinant vanishes, for which there is existence of minimizers. 
As a by-product we get optimal constants in by now classical estimates in geometric function theory.


\end{abstract}

\maketitle

\section{Introduction}


In 1952, C.~Morrey coined the notion of quasiconvexity \cite{Morrey} to investigate the scope of the
direct method of Calculus of Variations for vectorial variational problems of the type
\begin{equation}
\label{eq:variational}
{\mathcal E }[f] \,\equiv \int_\Omega { \bf E}(\D f(x)) \, \dd x
\end{equation}
where $f \colon \Omega \to \mathbb{R}^m$ is a suitably regular map and the functional ${ \bf E}\colon \R^{m \times n} \to \mathbb{R}$ is assumed to be at least Borel measurable.
Here, and throughout the paper, $\Omega\subset \R^n$ is a bounded domain and $m,n\geq 2$.
Under natural growth assumptions, quasiconvexity is equivalent  \cite{AcerbiFusco,Morrey} to the sequential weak lower semicontinuity of
\eqref{eq:variational}, and (strong) quasiconvexity is equivalent to the coercivity of \eqref{eq:variational}, see \cite{ChJK}. 

The interest  in quasiconvexity was  largely increased when J.M.\ Ball discovered its relevance to the theory of Nonlinear
Elasticity \cite{Ball1,Ball1981}, opening a new era in the field. In hyperelasticity one considers variational problems as in \eqref{eq:variational} with
$m=n$, but to exclude interpenetration of matter one additionally assumes $J_f(x) \geq 0$ a.e.~for the admissible maps, 
or even  the stronger condition
\begin{equation}
\label{eq:blowup}
{ \bf E}(A) \to +\infty \quad \tp{as } \det A \to 0^+,
\end{equation}
which penalizes the compressibility severely, 
see  for example \cite[\S 4.6]{Ciarlet} for a detailed discussion in the elasticity context.
Nowadays, condition \eqref{eq:blowup} has become classical in the mathematical treatment of variational problems arising in hyperelasticity see e.g \cite{Cartesian}.
In any case, both conditions  ask that ${ \bf E}$ be defined in $\R^{n\times n}_+\equiv \R^{n\times n}\cap \{\det >0\}$. 
 
Somewhat independently, the concept of quasiconvexity has attracted the attention of researchers in Geometric Function Theory,  for instance because
of its relations to the  sharp $\LL^p$ theory of singular integrals and quasiconformal 
mappings. Prime among the relevant objects here  is the celebrated Burkholder functional,  defined for $A\in \M$ by
\begin{equation}  \label{Burk}
 {\bf B}_p(A) \equiv  \Bigl( (\tfrac{p}{2} - 1)|A|^2 - \tfrac{p}{2}\det A \Bigr) |A|^{p-2}, \quad 2 \leq p < \infty,
\end{equation}
where $|\cdot|$ denotes the operator norm.
This functional was introduced by D.~Burkholder \cite{Burk1,Burk2} in the context of martingale inequalities;  see also \cite{BarnsteinMontgomerySmith,Sverak91}
and \cite{Tade} for its relation with quasiconvexity, geometric function theory and the Morrey problem discussed below. 

The Burkholder functional is both $p$-homogeneous and isotropic, and moreover ${\bf B}_2(A)=-\det(A)$.
Thus one can interpret ${\bf B}_p$ as an $\LL^p$-version of the determinant.  This analogy goes in fact much deeper: the determinant is a null-Lagrangian,
while the Burkholder functional has a large class of extremals, i.e.\ maps for which the integral bound \eqref{eq:quasiconvexity} below holds as an equality, see
e.g.\ \cite{AIPS12, AIPS15a, BarnsteinMontgomerySmith,Tade}.  Furthermore,  the Burkholder functional and its (potential) quasiconvexity  is
intimately related to the problem of determining the $\LL^p$-norms of the Beurling--Ahlfors transform, cf.\ \eqref{Beur1}.
We refer the reader to \cite{BJ,BJV,NV} and to the survey \cite{Banuelos} for further results and information in this direction.

The main objects of study in hyperelasticity are \textit{deformations}, i.e.\ regular maps with regular inverses \cite{Ball1981,Ciarlet,Cartesian}, so
that no distinction is made between the properties of the map and its inverse.  Precisely the same view holds also for the
quasiconformal maps and their applications \cite{IO}.  Thus, as in \cite{Ball1, Cartesian}, and motivated
by our interest in Nonlinear Elasticity and Geometric Function Theory, for functionals defined on $\R^{n\times n}_+$ it is natural to set:
\begin{definition}\label{quasiconvexity}
Let ${\bf E}\colon \R^{n\times n}_+\to \R$ be locally bounded and Borel measurable. Then ${\bf E}$ is said to be quasiconvex at $A\in \R^{n\times n}_+$ if 
\begin{equation}\label{eq:quasiconvexity}
{\bf E}(A)\leq \fint_{\Omega}  {\bf E}(\D f) \, \dd x
\end{equation}
whenever $\Omega \subset \R^n$ is a bounded domain and $f$ a $\CC^1$-diffeomorphism with $f=A$ in $\p \Omega$; we clarify that 
$f \in \CC^1(\overline \Omega)$ and  $f^{-1}\in \CC^1\Bigl(\overline{A(\Omega)} \Bigr)$ are  required.
\end{definition}

\begin{remark} From the general point of view, it is important to note that in this paper we study functionals ${\bf E}$ defined on $\M_+$ and their
quasiconvexity properties, so that the test functions are typically homeomorphisms. 
Some of these functionals satisfy \eqref{eq:blowup}, some not, and
thus the applications of our results goes beyond non-linear elasticity. 
\end{remark}
Another important remark is that for a given functional,  inequality \eqref{eq:quasiconvexity} typically holds for a much larger class
of Sobolev functions. For instance, for all explicit functionals studied in this work we will prove  \eqref{eq:quasiconvexity} for homeomorphisms such
that both $f \in \WW^{1,2}(\Omega, \Omega')$ and $f^{-1} \in \WW^{1,2}(\Omega', \Omega)$, where $\Omega' = A(\Omega)$. 

It is still a question of interest to study \eqref{eq:quasiconvexity} when testing the inequality for general Lipschitz maps  with  a.e.\ $\det(\D f) > 0$.
Then, however,  the  question becomes a problem of approximation
of Sobolev homeomorphisms, an active area of its own \cite{HP,IKO}. See Section~\ref{sec:shield} for a discussion on this and for some positive partial results.


Testing the quasiconvexity inequality with smooth approximations of a planar wave, i.e.\ a map which only takes two values, shows that that
quasiconvexity implies convexity along rank-one directions,  abbreviated as \textit{rank-one convexity}.  Whether conversely rank-one convexity
implies quasiconvexity is a famous problem, going back to Morrey's work \cite{Morrey,Morrey2}. The celebrated work of V.~\v{S}ver\'ak \cite{Sverak92}
gives a counterexample in dimensions $m\ge 3$. His example consists of a map which is a superposition of three planar waves (see also \cite{Grabovsky}
for a different example when $m\geq 8$).  When $m=2$  \v{S}ver\'{a}k's example does not work  \cite{PedregalSverak98} and indeed three waves cannot
provide a counterexample \cite{SebestyenSzekelyhidi17},  see also \cite{GTdC}.  In fact, in two dimensions there are partial
positive  results \cite{FaracoSzekelyhidi08,HKL18,KirchheimSzekelyhidi,MullerDiag} which suggest that rank-one convexity might imply quasiconvexity. 
 

The purpose of this paper is  to investigate whether for $n=m=2$ rank-one convexity might imply quasiconvexity, at least for functionals with
symmetries and additional structure, such as the Burkholder functionals.  

 In the context of Definition \ref{quasiconvexity} it is natural to set ${\bf E}(A) = +\infty$ when 
$\det(A) \leq 0$. In this case quasiconvexity as defined above is no longer a sufficient condition for weak lower semicontinuity,  cf.\ Section \ref{sec:prelims}.
Thus, in addition, one also needs an exploration towards the properties of the functionals on existence of minimizers. Here we realised that a
stronger quasiconvexity inequality, one which for the Burkholder functional ${\bf B}_p(A)$ can be viewed as a version of the classical area formula, is
then needed, see Theorem \ref{thm:Bpareaintro}. Moreover, versions of such inequalities in the limit $p \to 2$ lead to new lower semicontinuity and
existence theorems of interest in their own. 

\subsection{The Burkholder functional}
The Burkholder functional is known to be rank-one convex, since the original work   
\cite{Burk1,Burk2}. For other approaches on this see \cite{BarnsteinMontgomerySmith,Sverak91} or \cite{Tade}.

Our first theorem asserts that the Burkholder functional is quasiconvex when restricted to the set where it takes \textit{non-positive} values.
To interpret this setting,  note that 
\begin{equation}\label{Burkh46}
{\bf B}_p(A) \leq 0 \iff |A|^2 \leq \frac{p}{p-2} \det(A),
\end{equation}
thus such a map $A \in \M$ is  $K$-quasiconformal with $K = \frac{p}{p-2}$, equivalently
\begin{equation}\label{K46}
p = p_K \equiv \frac{2K}{K-1}. 
\end{equation}
In particular, if \eqref{Burkh46} holds then $A \in \M_+$, unless $A = 0$.
\begin{theorem} \label{thm:soft}
Let $p\geq 2$. For any $A \in \M$ and any  $f \in A+ \WW^{1,2}_0(\Omega , \R^2)$ such that ${\bf B}_p(\D f) \leq 0$ a.e.\ in $\Omega$, we have
\begin{equation}\label{Burkh}
{\bf B}_p(A) \;  \leq \;  \fint_{\Omega} \! {\bf B}_p\bigl(\D f(z)\bigr) \, \dd m(z).
\end{equation}
\end{theorem}
\smallskip

On the other hand,  it was shown in \cite{GK} that the positive part $\mathbf{B}_{p}^{+} \equiv \max \{ \mathbf{B}_{p},0 \}$ is
polyconvex. In particular, it follows that \eqref{Burkh} persists for maps $f \in A+ \WW^{1,2}_0(\Omega , \R^2)$ satisfying $\mathbf{B}_{p}(\D f) \geq 0$
almost everywhere in $\Omega$.
Thus  combined, these two results provide strong evidence towards the full quasiconvexity of ${\bf B}_p$.

Theorem \ref{thm:soft} was established in the special case  $A=\Id$ in \cite{AIPS12}. Also, as we will explain later, the result
entails a sharp integrability statement for quasiconformal maps.  At the moment let us emphasize that it comprises
delicate cancellation properties: for any $K$-quasiregular map $f$ and for $p = p_K$, we have ${\bf B}_{p_K}(\D f(z)) \in \LL^1_{\loc}$,
even if in general the map  $f \in \WW^{1,s}_{\loc}$ only for $s < p_{K}$ \cite{Astala1994}.

    



In order to establish the existence of minimizers for the induced Burkholder energy we actually need a stronger form of quasiconvexity. To this end,
it is convenient to introduce the \textit{local Burkholder functional}, defined by
\begin{equation}
\label{eq:locBurk}
{\mathcal B}_K (A) \equiv \begin{cases}
\B_{p_K}(A), & \tp{if } |A|^2 \leq K \det A,\\
+\infty & \tp{otherwise}.
\end{cases}
\end{equation}
Thus ${\mathcal B}_K$ equals the Burkholder functional $\B_p$ with the largest $p$ for which $\B_p(\D f) \leq 0$ for every $K$-quasiconformal map $f$,
and  it becomes defined in all of $\R^{2 \times 2}$ at the cost of admitting the value $+\infty$ outside the $K$-quasiconformal cone.

However, since ${\mathcal B}_K $ assumes the value $+\infty$, the notion of quasiconvexity needs to be strengthened to \textit{closed quasiconvexity}
\cite{JK,Pe1}: briefly, one requires  the Jensen inequality, that is \eqref{Burkh}, to hold not just for maps but also for gradient Young measures, cf.\ Section
\ref{sec:prelims} for the precise definition.
With this terminology,  we obtain the following  stronger version of Theorem \ref{thm:soft}.


%

\begin{theorem} \label{main}
Let $K\geq 1$ and  $p>\tfrac{2K}{K+1}$.
Then the local Burkholder functional   ${{\mathcal B}_{K}\colon \M \to \R \cup \{+\infty\}}$ 
is closed $\WW^{1,p}$-quasiconvex.
\end{theorem}





By combining Theorem \ref{main} with standard results from the theory of  Young measures and the Direct Method of the Calculus of Variations we obtain the  existence of minimizers:


\begin{corollary}\label{cor:Bpminims}
Let $K\geq 1$ and $2 \leq p < \frac{2K}{K-1}$.  Then for any $K$-quasiregular map $g\colon \C\to\C$,  the problem
$$
\inf\left\{\int_\Omega {\B}_p(\D f)\, \dd m(z):  f\in g+\WW^{1,p}_0(\Omega) \text{ is } K\text{-quasiregular}\right\}
$$
admits a minimizer $f\in g+\WW^{1,p}_0(\Omega)$.
\end{corollary}
\noindent Note that  here we do not need to 
 require the maps to be  homeomorphisms.  

\vspace{.03cm}

\subsection{The Burkholder area inequality}
We next turn to a further refinement  of Theorem \ref{thm:soft}, of independent interest, but also one of the key points 
in the study of the associated weak lower semicontinuity and minimization problems, see Section \ref{sec:swlsc}.

Namely, given a $\WW^{1,2}_\tp{loc}(\C)$-homeomorphism $f$, analytic outside $\mb D$, we say that $f$ is a \textit{principal map} if it  has  the Laurent expansion 
\begin{equation} \label{eq:principal}
f(z)=z + \frac{b_1}{z} + \sum_{j=2}^\infty \frac{b_j}{z^j},\qquad |z|>1.
\end{equation}

It follows from the classical Gr\"onwall--Bieberbach area formula 
 that in this expansion $|b_1| < 1$, see Section \ref{sec:principal}. This allows us to interpret
the first two terms of the above series,  i.e.\ the main asymptotics of $f$, in terms of the invertible  linear map
\begin{equation} \label{eq:asympt}
 A_f(z) = z+b_1 \bar z, \qquad {\rm equivalently} \quad A_f\equiv \fint_{\mb D} \D f(w) \, \dd m(w).
\end{equation}

As we will see, in many respects $A_f$ plays the role which the linear boundary values have in the standard definition of quasiconvexity.
For instance, with this notation the classical area formula,  cf.\ \eqref{areafmla},  asserts that
\begin{equation}
\label{eq:area}
\fint_{\mb D} \left[ - \det \D f + \det A_f\right]\dd m(z)=  \sum_{j=2}^\infty j |b_j|^2,
\end{equation}  
and one can think of this identity as a sharpening of the well-known null-Lagrangian property of the Jacobian determinant.

In the same spirit one can consider also the Burkholder functional, recalling that ${\B}_2(A) = - \det(A)$. In fact, with Theorem \ref{thm:soft}
we find an $\LL^p$ version of \eqref{eq:area}:

\begin{theorem}[Burkholder Area Inequality]\label{thm:Bpareaintro}
Let $f$ be a $K$-quasiconformal principal map as in  \eqref{eq:principal}.  Then, for any $2\leq p\leq p_K$, we have
\begin{equation*}
\fint_\mb D \left[ {\B}_p(\D f)- {\B}_p(A_f)\right] \dd m(z) \ge \gamma_p(A_f) \sum_{j=2}^\infty j |b_j|^2
\end{equation*}
where $\gamma_p(A_f) \equiv \frac{p}{2}   \frac{{\B}_p(A_f)}{{\B}_2(A_f)}>0$.  Note that $\gamma_2=1$.
\end{theorem}

%
%

Theorem \ref{thm:Bpareaintro}  sharpens  the main result in \cite{AIPS12}, since ${\B}_p(A_f) \ge {\B}_p(\Id )$
for any principal map $f$ as above (note, however, the different sign-convention for the Burkholder functional in  \cite{AIPS12}).

Since $A_f=\fint_{\mb D} \D f(z) \, \dd m(z)$,  a surprising feature of Theorem \ref{thm:Bpareaintro} is that it
establishes a Jensen inequality even without requiring that the map takes affine boundary values, as in the definition of quasiconvexity!
In fact, the result shows that, if 
$$
\int_{\mb D} {\B}_p(\D f) \, \dd m(z) = \int_{\mb D} {\B}_p(A_f) \, \dd m(z),
$$
then $f(z)=z+\frac{b_1}{z}$ for $|z|>1$; in particular, $f|_{\mb S^1}$ is linear.

\subsection{Functionals for Nonlinear Elasticity, as $p \to 2$.}

As observed in \cite{Tade} the theory of Burkholder functionals has very interesting consequences at the  limit when the index $p$ goes to $2$.
Namely, we have $\B_2(A)=-\det (A)$ and for the next order of approximation
\[ \lim_{p \to 2} \, \frac{p}{p-2} \left[ {\bf B}_p(A) - {\bf B}_2(A)\right] = {\mathscr  F}(A), \]
where the functional
 \begin{equation*}\label{F2}
 { \mathscr  F}(A) \equiv   \, |A|^2\, -  \,\left(1\, + \,\log |A|^2 \,\right )\, \det(A),  
 \end{equation*}
is rank-one convex in $\R^{2\times 2}$, but not polyconvex.  On the other hand,  the local quasiconvexity of the Burkholder functional, as formulated in
Theorem \ref{thm:soft},  allows us to show that  $\mathscr F$  satisfies the quasiconvexity inequality \eqref{eq:quasiconvexity} 
for all  $\WW^{1,2}$-homeomorphisms with linear boundary values $A \in \R^{2\times 2}_+$, see Corollary~\ref{cor:LlogL} for the precise statement.
Notice that this for instance implies a sharp and quantitative version of the celebrated $\LL \log \LL$-higher integrability properties for the derivatives
of $\WW^{1,2}$-homeomorphisms. 

Another consequence of these relations concerns the quasiconvexity of the functional 
\begin{equation}
\label{eq:Wfcn}
\mathscr W(A)\equiv  \frac{|A|^2}{\det A} - \log \left( \frac{|A|^2}{\det A}\right) + \log \det A, \quad A \in \M_+.
\end{equation}
This is an example of a rank-one convex but non-polyconvex functional, which diverges as the $\det(A) \to 0.$
It was introduced in \cite{AIPS12} and further studied in the recent works \cite{Voss1, Voss2}, where its quasiconvexity remained undecided. 

In fact, $\mathscr W$ arises from ${\mathscr  F}$ by applying the Shield transform \cite{Shield}. Therefore Theorem \ref{thm:soft} leads us to the following:
\begin{corollary} \label{cor:Wintro}
The functional $\mathscr{W} \colon \R^{2\times 2}_+\to \R$ 
is quasiconvex.
\end{corollary} 
In Section \ref{sec:swlsc} we investigate weak lower semicontinuity properties of $\mathscr W$ and establish that some
features similar to closed quasiconvexity hold also for $\mathscr{W}$;  for precise formulations see Proposition \ref{thm:closedqc.new}.

Originating from ${\bf B}_p$, the initial functional \eqref{eq:Wfcn} assumes all real values, and even tends  to $-\infty$ along suitable
directions when the determinant goes to zero. However, it allows easy modifications creating quasiconvex and non-polyconvex functionals that
satisfy \eqref{eq:blowup}. The following is perhaps the easiest example:
\begin{equation}\label{modiW}
\widetilde{\, \mathscr W \,} (A) \equiv  \frac{|A|^2}{\det A} - \log \left( \frac{|A|^2}{\det A}\right) + |\log \det A|,
\end{equation}
see Remark \ref{weehat}.


The interest in \cite{Voss1,Voss2} on $\mathscr W$ originates from the fact that it spans the only non-polyconvex extreme ray in a
class of  functionals satisfying the \textit{additive volumetric-isochoric split}, see  also \cite{Guerra19}  for further information on extremal functionals.
Thus, as a consequence of Corollary \ref{cor:Wintro}, we obtain a solution to Morrey's problem in a class of elastic functionals:

\begin{theorem}\label{thm:morreysplit}
Let $\, {\bf E}\colon \R^{2\times 2}_+\to \R$ be a functional of the form
$${\bf E}(A)=g(\det A)+h(K_A), \qquad K_A\equiv  \frac{|A|^2}{\det A},$$
where $h\colon[1,+\infty)\to \R$ is convex and $g\colon (0,+\infty)\to \R$. Then
$${\bf E} \text{ is rank-one convex} \iff {\bf E} \text{ is quasiconvex.}$$
\end{theorem}

The additive volumetric-isochoric split goes back at least to the work of Flory \cite{Flory} and since then it has been used extensively
to model slightly compressible materials, see for instance \cite{HartmannNeff,Ogden} and the references therein.
\smallskip

\subsection{Lower semicontinuity and existence of minimizers}
As discussed above the notion of quasiconvexity was introduced by Morrey to characterize sequential weak lower semicontinuity for integral functionals
in the vectorial calculus of variations. On the other hand, for functionals ${\bf E}\colon \R^{2\times 2}_+\to \R$ the condition \eqref{eq:blowup} expresses
the intuitive and natural requirement of hyperelasticity that \textit{an infinite amount of energy is required to compress a finite volume of material into
  zero volume} \cite{Ball1981, Ciarlet}. However, for such functionals with \eqref{eq:blowup} it is not clear if quasiconvexity suffices for lower semicontinuity
results. The stronger notion of polyconvexity does suffice, allowing  a wealth of interesting models for  hyperelastic materials \cite{Ball1}. 
 
In fact,  Theorem \ref{thm:morreysplit} already provides a natural class of functionals that have been considered before in the engineering literature \cite{Ciarlet},
but their weak lower semicontinuity and minimization properties had not been established. In addition to  Theorem~\ref{thm:morreysplit} or Corollary~\ref{cor:Wintro}
we shall also address this point here. In this connection, the extended Stoilow factorization due to Iwaniec and \v{S}ver\'{a}k \cite{IwaSve} suggested to us that
Jensen inequality with respect to principal maps might be sufficient for lower semicontinuity.  Indeed, this was our original indication that a theorem like the
Burkholder area inequality might be true. Applying this line of thought leads us to  the following Jensen inequality for principal maps: 

\begin{theorem} \label{thm:Wareaintro}
Let  $f \in \WW^{1,1}_{\loc}(\C)$ be a homeomorphism, and  a principal map with integrable distortion $K_f \in \LL^1(\mb D)$. Then
$$
\int_{\mb D} \bigl[ \mathscr W\bigl( \D f(z) \bigr) - \mathscr W\bigl( A_f \bigr) \bigr]\dd m(z) \geq 0.
$$ 
\end{theorem} 

\begin{remark}
Notice that the above inequality implies sharp  bounds on the integrability of $ \log(1/J_f) $ 
 in terms of those of $K_f$. 
\end{remark}

Here the assumption $K_f \in \LL^1$ is optimal mathematically. Moreover, notice that in the study of incompressible neo-Hookean materials, 
the first invariant of the isochoric part of the Cauchy-Green tensor of the deformation $f$ is $\widehat{I_1} = K_f + 1/K_f$, see \cite{Ogden}.
Since  it is unclear how to measure experimentally the response of materials as the determinant tends to zero \cite{Ciarlet}, the condition $K_f \in \LL^1$ might
in fact be  the right postulate within our current knowledge. In addition, recall that the norm $\| K_f\|_{\LL^1}$ equals the $\WW^{1,2}$-Sobolev norm of $f^{-1}$;
thus it is plausible that the condition $K_f \in \LL^1$ is the right regularity requirement  in order to have a flexible lower semicontinuity theory. 

It turns out that,  for general functionals with the blow-up \eqref{eq:blowup},  proving lower semicontinuity requires two properties: the Jensen inequality for principal maps,
allowing analysis via gradient Young measures, and secondly, control of concentration, typically via suitable equiintegrability. With these properties available we 
easily obtain the following lower semicontinuity result.  


\begin{theorem}
\label{thm:lscW}
Let $g\in \WW^{1,2}_\tp{loc}(\C)$ be a homeomorphism. Consider a sequence $(f_j )$ in $g + \WW^{1,2}_0(\Omega)$ such that 
$f_j\weak f$ in $\WW^{1,2}(\Omega)$ and for some $q>1$ we have $\|K_{f_j}\|_{\LL^{q}(\Omega)} \leq C < \infty$. Then 
$$
\liminf_{j\to \infty} \int_\Omega {\mathscr W}(\D f_j(z)) \, \dd m(z) \geq \int_\Omega {\mathscr W}(\D f(z))\, \dd m(z).
$$
\end{theorem}
Similar weak lower semicontinuity  holds for the rank-one convex  functionals from Theorem \ref{thm:morreysplit} with the appropriate volumetric-isochoric
split, see Corollary \ref{cor:lscsplit}. Notice that as in the Burkholder setting the endpoint result, which in this context is  $K_f \in \LL^1$, is missing.
This amounts to study possible concentration effects and we will investigate it in a future work. 


Building on $ {\mathscr W}$ or on $\widetilde  {\, \mathscr W\,}$ we obtain in Subsection \ref{sec:elasticity} a number of  quasiconvex non-polyconvex functionals
for which the direct method of the Calculus of Variation gives  existence of minimizers, see Corollary \ref{thm:miniselasticity} and Example \ref{nonpoly}. 
Morever since,  by the work of Koskela and Onninen \cite{KO},  the norm $\|K_f\|_{\LL^q}$   controls $\|\log J_f\|_{\LL^q}$,  many of these functionals
allow the blow-up condition \eqref{eq:blowup}. This, in particular,  sheds  light on the problem  of existence of minimizers in hyperelasticity, see \cite{Ball2002} and in particular Problem 1 there. 

On the other hand, this class also contains functionals which degenerate in various ways  when the determinant vanishes, like the Burkholder functional itself.
Therefore its interest is not restricted to the elasticity ecosystem but is relevant also  to the geometric function theory interpretation of our work.

\subsection*{Outline}

Finally, we conclude the introduction with a description of the organization of the paper.

 Section~\ref{sec:prelims} revisits the standard notions of the vectorial calculus of variations, with an emphasis on the special care needed to treat extended-real valued functionals.  

Section~\ref{sec:principal}  reviews relevant parts  of the basic quasiconformal theory,  with special focus on principal maps.

Section~\ref{Young} adapts the theory of quasiregular Young measures, initiated in  \cite{AstalaFaraco02} and applied for instance in \cite{FaracoSzekelyhidi08},
to the case of principal maps. It also extends this theory to maps of integrable distortion.  


Section~\ref{sec:Bf} provides the preliminary results needed in the proof of  the local quasiconvexity of the Burkholder functional,
in particular an extremality argument in the spirit of \cite{AIPS12}.  

Section~\ref{sec:Maintheorem} presents the proof of Theorems \ref{thm:soft} and \ref{main}.

 Section~\ref{sec:areaBp} contains the proof of Theorem \ref{thm:Bpareaintro}. 

Section~\ref{sec:LlogL} studies the functional $\mathscr F$, which is the derivative of $B_p$ at $p=2$, and is
closely related to the higher integrability of the Jacobian. 

Section~\ref{sec:shield} revisits the classical Shield transformation, which uses inverses to define new integral functionals and,
in particular, presents $\mathscr W$ as a transformation of $\mathscr F$, which leads to the proof of Corollary \ref{cor:Wintro}.
For the sake of completeness we also investigate for which class of test functions the quasiconvexity inequality for $\mathscr W$ can be verified. 

Section~\ref{sec:inequalities}  proves Theorem \ref{thm:Wareaintro}, which gives quasiconvexity of $\mathscr W$ in the class of principal maps.
Similar inequalities are established for $\mathscr F$. 

Section~\ref{sec:advolum} proves the quasiconvexity of functionals with volumetric isochoric split after that of $\mathscr W$,
streamlining the arguments in \cite{Voss1} and proving Theorem \ref{thm:morreysplit}.

In Section~\ref{sec:swlsc} we are then in position to apply a version of the direct method of calculus of variations to prove lower
semicontinuity theorems and existence of minimizers for a quite large family of functionals, see in particular Theorem~\ref{thm:miniselasticity}.
In this section  we also prove Corollary \ref{cor:Bpminims} and Theorem \ref{thm:lscW}.

\subsection*{Notation}
We denote by $\mb D\subset \C$ the unit disk and by $\mb A(r,R)$ the annulus $\{z:r<|z|<R\}$.
Unless explicitly stated otherwise, $\Omega\subset \R^n$ is a bounded domain such that $\mathscr{L}^n(\p \Omega)=0$; in most of the paper we will take ${n=2}$.
Given a map $g\colon \R^n\to \R^n$,  we  sometimes use the notation $\WW^{1,p}_g(\Omega)\equiv g+\WW^{1,p}_0(\Omega,\R^n)$; in particular, this space
is well-defined even if $\p\Omega$ is irregular.  We also use the standard notation
$$
  \fint_\Omega \varphi(x) \,\dd x \equiv \frac{1}{|\Omega|} \int_\Omega \varphi(x) \,\dd x.
$$
A matrix $A\in \R^{2\times 2}$ is naturally identified with a linear map $A\colon \C\to \C$.
It will also be useful to use conformal coordinates, whereby we identify $A\in \R^{2\times 2}$ with a pair $(a_+,a_-)\in \C^2$ according to the rule 
\begin{equation}
\label{eq:confcoords}
A(z) = a_+ z + a_- \bar z.
\end{equation}
Here, on the left-hand side we see $z$ as an element of $\R^2$, while on the right-hand side $z\in \C$. In these coordinates we have
\begin{equation}
\label{eq:detconformal}
\det A=|a_+|^2-|a_-|^2, \qquad |A|=|a_+|+|a_-|,
\end{equation}
where $|A|\equiv \max_{z\in \mb S^1} |A(z)|$ denotes the operator norm.

Given a map $f\in \WW^{1,1}_\tp{loc}(\Omega)$, we write $J_f\equiv  \det \D f$ for its Jacobian. If 
$J_f>0$ a.e.~ in $\Omega$ then there is a measurable function $K\colon \Omega\to [1,+\infty]$ such that $K<\infty$ a.e.~ and
\begin{equation*}\label{distortion2}
|\D f(z)|^2\leqslant K(z)J_f(z),\qquad {\rm a.e.\ in} \;\; \Omega.
\end{equation*}
The {\it distortion function} of $f$, which we denote by $K_f$, is the smallest such function $K$.

\subsection*{Acknowledgments}
D.F, K.A, A.K acknowledge the financial support of QUAMAP, the ERC Advanced Grant 834728, 
 and of  
the Severo Ochoa Programme 
CEX2019-000904-S.
A.G. was supported by Dr.\ Max R\"ossler, the Walter Haefner Foundation and the ETH Z\"urich Foundation.  
D.F and A.K were  partially supported by CM and UAM,
and A.K by Academy of Finland CoE Randomness and Structures, and Academy Fellowship Grant 355840. 
 D.F acknowledge financial support by PI2021-124-195NB-C32. 
 
 K.A,  D.F, A.G, A.K   acknowledge the hospitality and financial support of the Institute of Advanced studies during
various periods in 2021-2022 and the discussions there with  C. De Lellis, V.\v{S}ver\'ak  and L.Sz\'ekelyhidi Jr, on topics related to the paper. K.A, A.G, J.K acknowledge the hospitality of Universidad
Aut\'onoma de Madrid and ICMAT during the autumn of 2022. D.F also wants to acknowledge the hospitality of the Mathematical Institute of Oxford during the summer of 2023.

\section{Extended real-valued functionals and Young measures}
\label{sec:prelims}

\subsection{Quasiconvexity and rank-one convexity}

Put $\eR \equiv  \R \cup \{ \pm \infty \}$ and let $F \colon \R^d \to \eR$ be an extended real-valued function.
The epigraph of $F$ is the subset
$$
\mathrm{epi}(F) \equiv  \bigl\{ (x,y) \in \R^{d+1} : \, y \geq F(x) \bigr\}
$$
of $\R^{d+1}$. We say that $F$ is \textit{convex} if its epigraph is a convex subset of $\R^{d+1}$.

The following result is an easy consequence of the definitions and we leave the proof to the interested reader.
\begin{lemma}\label{lem1}
Assume $F \colon \R \to \eR$ is convex and that for some $x_0 \in \R$ we have $F(x_{0})=-\infty$. Then there exist
$\alpha$, $\beta \in \eR$ with $\alpha \leq x_{0} \leq \beta$ such that
$$
F = \left\{
\begin{array}{ll}
-\infty & \mbox{ in } ( \alpha , \beta ),\\
+\infty & \mbox{ on } \R \setminus [ \alpha , \beta ] .
\end{array}
\right.
$$
Furthermore, when $\alpha \in \R$ or $\beta \in \R$, the values of $F$ there are unrestricted.
\end{lemma}
Let ${\bf E} \colon \R^{m\times n} \to \eR$ be an extended real-valued function (henceforth we refer to functions defined on $\R^{m\times n}$ as \textit{functionals}).
The \textit{effective domain} of ${\bf E}$ is the subset
\begin{equation} \label{effective}
\mathrm{dom}({\bf E}) \equiv  \bigl\{ A \in \R^{m\times n}: \, {\bf E}(A) < +\infty \bigr\}
\end{equation}
of $\R^{m\times n}$, where we emphasize that the value $-\infty$ is allowed for ${\bf E}$ on $\mathrm{dom}({\bf E})$.

\begin{definition}\label{def:rc}
A functional ${\bf E} \colon \R^{m\times n} \to \eR$ is \textit{rank-one convex} if, for all $A$, $X \in \R^{m\times n}$ with $\mathrm{rank}(X)=1$, the
function $t \mapsto {\bf E}(A+tX)$ is convex. 
\end{definition}
We emphasize that in the considered generality this notion of rank-one convexity is
 rather weak, see already Example \ref{ex:4pts} below.

For the next result we need the notion of algebraic interior or \textit{core} of a subset $S \subseteq \R^{m\times n}$, namely
$A \in \mathrm{core}(S)$ provided for each $X \in \R^{m\times n}$ we can find $\delta > 0$ such that $A+tX \in S$ for all $t \in [0,\delta )$.

\begin{lemma}\label{lem2}
Assume ${\bf E} \colon \R^{m\times n} \to \eR$ is a rank-one convex functional and that for some $A_{0} \in \mathrm{core}\bigl( \mathrm{dom}({\bf E}) \bigr)$
we have ${\bf E}(A_{0}) \in \R$. Then ${\bf E} > -\infty$ on all lines through $A_0$ that are parallel to a rank one matrix.
\end{lemma}

\begin{proof}
If there exists $X \in \R^{m\times n}$ of rank one such that ${\bf E}(A_{0}-X) = -\infty$, then by convexity in the direction of $X$ we infer that ${\bf E}(A_{0}+tX) = +\infty$
for all $t>0$. But this is impossible when $A_{0} \in \mathrm{core}\bigl( \mathrm{dom}({\bf E}) \bigr)$, so ${\bf E}> -\infty$ must hold on all lines through $A_0$
that are parallel to a rank one matrix.
\end{proof}

\begin{corollary}\label{cor2a}
Assume ${\bf E} \colon \R^{m\times n} \to \eR$ is a rank-one convex functional and that for some $A_0$ in the topological interior of $\mathrm{dom}({\bf E})$ we
have ${\bf E}(A_{0}) \in \R$. Then ${\bf E} > -\infty$ on the connected component $\mathcal{U}$ of the interior of $\mathrm{dom}({\bf E})$ that contains $A_0$.
In fact, ${\bf E} > -\infty$ on any line through a point of $\mathcal{U}$ that is parallel to a rank one matrix.
\end{corollary}

\begin{proof}
It is well-known that any two points of an open connected set in $\R^{m\times n}$  can be connected by a piecewise linear curve with
each line segment parallel to a rank one matrix. It therefore follows using Lemma \ref{lem2} that ${\bf E} > -\infty$ on $\mathcal{U}$
and hence that ${\bf E}>-\infty$ on any line which is parallel to a rank one matrix and which intersects $\mathcal{U}$.
\end{proof}

In relation to rank-one convexity it is natural to consider quasiconvexity, which is an integral condition.
 In our key results we are considering  extended real-valued functionals ${\bf E}$, and thus given a Borel
measure $\nu$ in $\R^{m\times n}$ we write  for such functionals
\begin{equation}  \label{upperint}
 \int_{\R^{m\times n}}^* \! {\bf E} \, \dd\nu \equiv  
\begin{cases}
 \int_{\M} \! {\bf E} \, \dd\nu  \quad &\text{if} \;\; {\bf E} \in \LL^{1}(\nu), \\ 
 + \infty \quad &\text{if} \;\; {\bf E}^+ \notin \LL^{1}(\nu),\\ 
-\infty & \text{if} \;\; {\bf E}^+ \in \LL^{1}(\nu) \; {\rm but } \; {\bf E}^- \notin \LL^{1}(\nu).
\end{cases}
\end{equation}
One can interpret \eqref{upperint} as an \textit{upper $\nu$-integral} of ${\bf E}$.

In the setting of extended real-valued functionals, a natural counterpart to the notion \eqref{quasiconvexity} of quasiconvexity is to require integral bounds such as 
\begin{equation}
\label{eq:qc23}
{\bf E}(A)\leq \fint^*_\Omega {\bf E}(A+\D \varphi) \,\dd x
\end{equation}
to hold for all bounded domains $\Omega\subset \R^n$ and, say, for all  $\varphi \in \WW^{1,p}_0(\Omega,\R^m)$. 
However,  the above version of quasiconvexity is very weak. In fact, it does not even imply rank-one convexity \cite[Example 3.5]{BM1}:
\begin{example}[Two points]\label{ex:2pts}
Let $X\in \R^{m\times n}$ have $\tp{rank}(X)=1$. The functional
$${\bf E}(0)={\bf E}(X)=0,\qquad {\bf E}=+\infty \tp{ otherwise},$$
is  not rank-one convex, yet it satisfies \eqref{eq:qc23} for all $\varphi  \in \WW^{1,1}_0$ and $A \in \R^{m\times n}$.
\end{example} 
Thus, arguably, the correct notions to consider are those of closed rank-one convexity and closed quasiconvexity \cite{Pe1}, see also \cite{JK}.
In order to define these concepts we need to introduce the gradient Young measures, and this is the purpose of the next subsection.

\subsection{Gradient Young measures}
Let $\Omega\subset \R^n$  be a bounded domain with $\mathscr L^n(\p \Omega)=0$. Informally, a \textit{gradient Young measure} on $\Omega$ is
a parametrized family of probability measures $\nu=(\nu_x)_{x\in \Omega}$ on $\R^{m \times n}$, where at each point $x \in \Omega$, the measure
$\nu_x$ describes the oscillations of the derivatives $\D\varphi_j(y)$ for points $y$ near $x$, in the limit $j \to \infty$ for 
 a sequence $( \varphi_j )$ that converges weakly in $\WW^{1,p}(\Omega, \R^m)$. 

More precisely,  if $\Meas (\R^{m \times n})$ denotes the space of finite signed Borel measures on $\R^{m \times n}$, then to any $\varphi \in \WW^{1,p}(\Omega, \R^m)$
we can associate the function
$$
T_\varphi \colon \Omega \mapsto \Meas (\R^{m \times n}), \quad T_\varphi (x) = \delta_{\D\varphi(x)}.
$$
Here $T_\varphi$ is bounded and weakly$\mbox{}^\ast$ measurable, that is, $T_\varphi  \in \LL^{\infty}_{\omega}\bigl(\Omega, \Meas (\R^{m \times n})\bigr)$.
Furthermore, with the natural duality pairing, the Lebesgue--Pettis space $\LL^{\infty}_{\omega}\bigl(\Omega, \Meas (\R^{m \times n})\bigr)$ is the dual of
the Lebesgue--Bochner space $\LL^1\bigl(\Omega, \CC_0(\R^{m \times n})\bigr)$ and under this duality the $\WW^{1,p}$-gradient Young measures $\nu=(\nu_x)_{x\in \Omega}$
are precisely the weak$^\ast$-limits of the sequences $( T_{\varphi_j} )$, where $( \varphi_j )$ is a sequence converging weakly in $\WW^{1,p}(\Omega, \R^m)$.
In explicit terms, this means that for all ${\bf E}\in \CC_0( \R^{m\times n})$,
\begin{equation} \label{Young1}
  \;  {\bf E}( \D\varphi_j (x) ) \; \stackrel{w^*}{\longrightarrow} \; \; \langle \nu_x ,  {\bf E} \rangle \,  \equiv \int_{\R^{m \times n}} {\bf E}(A) \, \dd\nu_x(A)
  \qquad {\rm in } \; \;   \LL^\infty(\Omega).
\end{equation}
When  \eqref{Young1} holds, we say that the sequence $(D\varphi_j)$ of gradients \textit{generates} the Young measure $\nu$. 
Formalising this concept we have:
\begin{definition}\label{def:gYM}
Let $p\in [1,\infty]$. A parametrized family of probability measures $\nu=(\nu_x)_{x\in \Omega}$ on $\R^{m \times n}$ is a \textit{$\WW^{1,p}$-gradient Young measure} if  
there exists a weakly converging sequence $(\varphi_j)$ in $\WW^{1,p}(\Omega,\R^m)$ whose gradients generate $\nu$,  that is, \eqref{Young1} holds for all
${\bf E}\in \CC_0( \R^{m\times n})$.
\end{definition} 

\begin{remark} \label{spGYM}
We collect a few immediate consequences of the definition of $\WW^{1,p}$ gradient Young measure. First, as a limit of weak$^\ast$-measurable functions also the family
$( \nu_x )$ is weak$^*$-measurable. This means that for every ${\bf E}\in \CC_0( \R^{m\times n})$ the function
$$
\Omega \ni x \mapsto \int_{\R^{m\times n}} {\bf E}(A) \, \dd \nu_x(A)
$$
is Lebesgue measurable. This property ensures that standard constructions involving $( \nu_x )$ result in measurable functions.

Next, the fact that $( \nu_x )$ is generated by an $\LL^p$-bounded sequence implies that the measure satisfies the \textit{$p$-moment condition}: \, For  almost all $x \in \Omega$,
\begin{equation}\label{moments}
\int_{\Omega} \! \langle \nu_{x} , | \cdot |^{p} \rangle \, \dd x < \infty \quad \mbox{ when } p \in [1,\infty ),
\end{equation}
whereas when $p=\infty$ there exists $R>0$ such that the supports $\mathrm{supp}( \nu_{x}) \subset B_{R}(0)$ for $\mathscr{L}^n$ almost all $x \in \Omega$.
The moment condition in particular means that for $\mathscr{L}^n$ almost all $x \in \Omega$ the probability measure $\nu_x$ has a centre
of mass, and it is not difficult to see from \eqref{Young1} and Remark \ref{Young31} that
\begin{equation} \label{limit} 
\langle \nu_x, \Id \rangle \equiv \int_{\R^{m \times n}} \! A \, \dd\nu_x(A) = \D \phi (x) \qquad {\rm for} \; a. e. \; x \in \Omega ,
\end{equation}
where $\phi$ is the weak limit of the generating sequence $( \phi_{j})$.

Finally we note that since for any bounded open set $\Omega$ with $\mathscr{L}^{n}( \partial \Omega )=0$ the inclusion
$\WW^{1,s}(\Omega,\R^m) \subset  \WW^{1,p}(\Omega,\R^m)$ holds if $p \leq s$, every $\WW^{1,s}$-gradient Young measure is a $\WW^{1,p}$-gradient Young
measure whenever $p \leq s$.
\end{remark}
Any norm-bounded sequence of gradients admits a subsequence that generates a gradient Young measure \cite[Theorem 3.1]{Muller}: 

\begin{theorem}\label{thm:fundYM}
Let $p\in [1,\infty]$ and let $( \varphi_j  )$ be a bounded sequence in $\WW^{1,p}(\Omega,\R^m)$.
Then there is a subsequence, which we do not relabel, such that $(\D\varphi_j)$ generates a $\WW^{1,p}$-gradient Young measure $\nu=(\nu_x)_{x\in \Omega}$. 
\end{theorem}



However, even if the original sequence $(\varphi_j)$ converges weakly in $\WW^{1,p}(\Omega,\R^m)$, the gradient Young measures its
subsequences generate need not be unique. On the other hand, since the center of mass or barycenter is the weak limit of \eqref{limit}, this is of course
the same for all Young measures generated by the subsequences of a given weakly converging sequence.
\smallskip

In view of \eqref{Young1} the gradient Young measures provide flexible methods to compute weak limits of nonlinear quantities.
Hence the notion is very useful also in our setting. For later purposes we  list here some of their basic and well-known properties.
For instance, given a Young measure one can often modify or improve the generating sequence.

\begin{theorem} \label{generates} \cite[Theorem 8.15 and Lemma 6.3]{Pe} Suppose $\nu=(\nu_x)_{x\in \Omega}$ is a $\WW^{1,p}$-gradient Young measure in $\Omega$
  and suppose $\langle \nu_{x},\mathrm{Id} \rangle = \D g(x)$ a.e.~in $\Omega$, where $g \in \WW^{1,p}( \R^n , \R^m )$.

(1)  Then there is a bounded sequence $(\varphi_j)$ in $g+\WW^{1,p}_{0}(\Omega,\R^m)$ whose gradients generate $\nu$ and for which 
  \begin{equation} \label{equi}
( | \D \varphi_j |^p ) \quad {\rm is \; equiintegrable \; on} \; \Omega.  
\end{equation}

(2)  If $(\varphi_j)$ is a bounded sequence in $\WW^{1,p}(\Omega,\R^m)$ that generates $\nu$, and $(\psi_j)$ is another sequence for which $\D \psi_j-\D\varphi_j\to 0$
in $\LL^p(\Omega)$, then also $(\D\psi_j)$ generates $\nu$.
\end{theorem} 
\medskip

The fact that the generating sequence  can be chosen so that $(|\D\varphi_j|^p)$ is equiintegrable will be important for us and is
related to the so-called Decomposition Lemma \cite{FMP,JK-1,JK0}.

\begin{remark} \label{Young31}
On the other hand,  the limit \eqref{Young1} exists even for a general continuous  ${\bf E}\in \CC(\R^{m\times n})$ which does
not vanish at $\infty$, but where, for the generating sequence, $( {\bf E}(\D \varphi_j) )$ is equiintegrable, see \cite[Theorem 6.2]{Pe}.
\end{remark} 

The literature on Gradient Young measures is by now quite extensive and for further properties 
we refer to the monographs \cite{KrRa, Muller, Pe, Rindler}.

\begin{remark} \label{rmk:moreintegrands}  It is not difficult to extend the convergence \eqref{Young1} also 
to  lower semicontinuous functionals ${\bf E}\colon \R^{m\times n}\to\R\cup\{+\infty\}$ which are continuous on $\tp{dom}({\bf E})$,
provided that $\left({\bf E}(\D \varphi_j)\right)$ is equiintegrable.
This can be proved by approximating ${\bf E}$ with the sequence of truncations ${\bf E}_k\equiv \min\{{\bf E},k\}$ and applying the standard
lower semicontinuity result
\begin{equation}
\label{eq:lscYM}
\liminf_{j\to \infty} \int_\Omega \eta(x) {\bf E}_k(\D \varphi_j)\, \dd x\geq \int_\Omega \eta(x) \int_{\R^{m\times n}} {\bf E}_k(A) \,\dd \nu_x(A) \, \dd x,
\end{equation}
where $\eta\in \LL^\infty(\Omega)$ is arbitrary, cf.\ \cite[Corollary 3.3]{Muller}; the opposite direction simply follows from the equi-integrability assumption. 
We note that \eqref{eq:lscYM} can also be proved by approximating lower semicontinuous functionals with continuous ones, using the Scorza--Dragoni theorem.
\end{remark}
In this paper, only  in Section \ref{sec:swlsc} we really use  gradient Young measures in the full generality of Definition \ref{def:gYM}. Instead, most of the time we will
be content to work with homogeneous gradient Young measures:

\begin{definition}
A $\WW^{1,p}$-gradient Young measure $(\nu_x)_{x\in \Omega}$ is \textit{homogeneous} if there is a probability measure $\nu$ on $\R^{m\times n}$
such that $\nu_x=\nu$ for a.e.\ $x\in \Omega$. In this case we naturally identify  $(\nu_x)_{x\in \Omega}$ with $\nu$.

We often denote by $\Meas^p_\tp{qc}$ the class of homogeneous $\WW^{1,p}$-gradient Young measures;  the notation is motivated
by Theorem \ref{thm:KP} below. For a subset $\mathcal{U}\subset\R^{m\times n}$ we write $\Meas^{p}_{\qc}( \mathcal{U})$ for the set of measures in $\Meas^{p}_{\qc}$
whose support is contained in $\mathcal{U}$. 
\end{definition}

\begin{remark}\label{graddist}
There are many natural ways to construct homogeneous  $\WW^{1,p}$-gradient Young measure. For instance, given $A \in \R^{m \times n}$ and $\phi \in \WW^{1,p}_{0}( \Omega , \R^m )$ we associate to them the measure  $\nu_{A+\D\phi}$ defined by the rule
\begin{equation}\label{exhgym}
\nu_{A+\D\phi}(\mathcal{S}) \equiv \frac{\mathscr{L}^{n}\bigl( \{ x \in \Omega : \, A+\D \phi (x) \in \mathcal{S} \} \bigr)}{\mathscr{L}^{n}( \Omega )}, 
\end{equation}
where $\mathcal{S} \subset \R^{m \times n}$ is a Borel set.
By inspection, $\nu_{A+\D\phi}$ is a Borel probability measure with a finite $p$-th moment $\langle \nu_{A+\D\phi} , | \cdot |^{p} \rangle < +\infty$ and centre of mass
$\langle \nu_{A+\D\phi},\mathrm{Id} \rangle =A$. In particular, the measure
$\nu_{A+\D\phi}$ describes the distribution of values of $A+\D \phi (x)$ in $\R^{m \times n}$ when $x$ varies over $\Omega$ and we use the normalized volume $\mathscr{L}^n$ as weight.

To represent $\nu_{A+\D\phi}$  as a homogeneous Young measure, since $\Omega$ is a bounded domain with $\mathscr{L}^{n}( \partial \Omega )=0$, we can realize this distribution on any other open bounded subset of $\R^n$, and for later purposes we choose  a realization 
on the open unit cube, $\mathbf{Q} \equiv \bigl( -\tfrac{1}{2},\tfrac{1}{2} \bigr)^n$.
Indeed, a standard exhaustion argument allows us to write $\mathbf{Q}$ as a disjoint union of scaled and translated copies of $\Omega$:
$$
\mathbf{Q} = N \cup \bigcup_{s \in \N} \bigl( x_{s}+r_{s}\Omega \bigr) \quad \bigl(\mbox{disjoint union!} \bigr)
$$
where $\mathscr{L}^{n}(N)=0$. Next we import $\phi$ on $\mathbf{Q}$ by the definition
$$
\varphi (x) \equiv \left\{
\begin{array}{ll}
  r_{s}\phi \bigl( \frac{x-x_{s}}{r_{s}} \bigr) & \mbox{ if } x \in x_{s}+r_{s}\Omega , \, s \in \N\\
  0 & \mbox{ if } x \in N.
\end{array}
\right.
$$
It is routine to check that hereby $\varphi \in \WW^{1,p}_{0}( \mathbf{Q}, \R^m )$ and that $\nu_{A+\D\varphi} = \nu_{A+\D\phi}$, if $\nu_{A+\D\varphi}$ is defined
as in \eqref{exhgym} with the obvious modifications. It is now easy to check that $\nu_{A+\D\phi}$ is a homogeneous $\WW^{1,p}$-gradient Young measure with centre of
mass at $A$: Namely, extend $\varphi$ to $\R^n$ by $\mathbf{Q}$-periodicity and define $u_{j}(x) \equiv Ax+\varphi (jx)/j$, $x \in \Omega$.
Using the Riemann-Lebesgue lemma it follows that $u_{j} \weak A$ in $\WW^{1,p}( \Omega , \R^m )$ and that $( \D u_j )$ generates the Young
measure $( \nu_x )_{x \in \Omega}$, where $\nu_x = \nu_{A+\D\phi}$ for all $x \in \Omega$. 
\end{remark}

The examples provided in the next Subsection \ref{closedqc} show that $\Meas^{p}_{\qc}$, the set of homogeneous measures, contains probability measures that cannot be represented
as $\nu_{A+\D\phi}$ for any $A$, $\phi$. On the other hand, using a variant of the above construction and a diagonalization argument it is not too difficult to see that $\Meas^{p}_{\qc}$
can be defined as a suitable closure of the set $\bigl\{ \nu_{A+\D\phi} : \, \phi \in \WW^{1,p}_{0}( \Omega , \R^m ), \, A \in \R^{m \times n} \bigr\}$. However, we will not need this in the sequel.


The usefulness of homogeneous gradient Young measures stems from the fact that a \textit{general} gradient Young measure is essentially
a collection of homogeneous gradient measures:

\begin{proposition}\textrm{(The Localization Principle \cite{KP1})}\label{prop:homogenization}
Given a $\WW^{1,p}$-gradient Young measure $\nu=(\nu_x)_{x\in \Omega}$, we have $\nu_x\in \mathscr M^p_\tp{qc}$ for a.e.\ $x\in \Omega$.
\end{proposition}

\begin{remark}\label{rmk:homogenization}
In fact,  if $\nu$ is generated by a sequence $(\varphi_j)$ then, for a.e.\ $x_0$, the measure $\nu_{x_0}$ is generated by a diagonal subsequence
of the sequence $\psi_{j,\lambda}(x) \equiv \lambda^{-1} (\varphi_j(x_0+\lambda x) - \varphi(x_0))$,  as $j\to \infty$ and $\lambda\to 0$. 
Note that
$$
\D \psi_{j, \lambda}(x) = \D\varphi_j(x_0+\lambda x).
$$
Since $\lambda\to 0$, we can assume that the maps $\psi_{j,\lambda}$ are defined on any bounded open set $\Omega$ for which $\mathscr L^n(\p \Omega)=0$.  

\end{remark}
Also, the basic invariance properties of $\mathscr M^p_\tp{qc}$ follow quickly: For  $t>0$, let
$$
\langle \nu_t,{\bf E}\rangle \equiv  \langle \nu, {\bf E}(t\cdot)\rangle = \int_{\R^{n\times n}} \! {\bf E}(t A) \, \dd \nu(A),
$$
and similarly, if $m=n$ and $Q,R\in \tp{SO}(n)$, we define
$$
\langle \nu_{Q,R},{\bf E}\rangle \equiv \int_{\R^{n\times n}} \! {\bf E}(Q A R) \, \dd \nu(A).
$$
\begin{lemma}\label{lemma:invYM}
Fix $p\in [1,\infty]$.
\begin{enumerate}
\item \label{it:YMhomogeneous} For any $t>0$, the map $\nu\mapsto \nu_t$, maps $\mathscr M^p_\tp{qc}$ bijectively onto itself.
\item\label{it:YMisotropic} Similarly, if $m=n$ and $Q,R\in \tp{SO}(n)$, the map 
$\nu\mapsto \nu_{Q,R},$ 
maps $\mathscr M^p_\tp{qc}$ bijectively onto itself.
\end{enumerate}
\end{lemma} 
\begin{proof}
Let $(\varphi_j)$ be a bounded sequence in $\WW^{1,p}( \Omega ,\R^m )$ such that $( \D\varphi_j )$ generates $\nu$.
For part (\ref{it:YMhomogeneous}) we consider a new sequence $\psi_j\equiv  \varphi_j(t \cdot)$ and for part (\ref{it:YMisotropic}) we take $\eta_j \equiv Q \varphi_j(R\cdot)$, so 
$$
\D \psi_j = t \D \varphi_j(t \cdot), \qquad \D \eta_j = Q \D \varphi_j(R\cdot) R.
$$
It is easy to verify that the sequences 
$(\D \psi_j)$ and $(\D \eta_j)$ generate the measures $\nu_t$ and $\nu_{Q,R}$, respectively.  Since $(\nu_t)_{t^{-1}}= \nu$ and $(\nu_{Q,R})_{Q^{-1},R^{-1}}= \nu$, the claim follows.
\end{proof}  
Homogeneous gradient Young measures admit another characterization which is often more convenient to work with \cite{KP,KP1}.

\begin{theorem}
\label{thm:KP} Fix $p\in [1,\infty)$ and let $\nu$ be a Borel probability measure on $\R^{m\times n}$. Here and throughout we denote its center of mass by
$$
\left\langle \nu, \mathrm{Id} \right\rangle \equiv \int_{\R^{m\times n}} A \,\dd \nu(A),
$$
so that in particular $\langle \nu, \mathrm{Id} \rangle \in \, \R^{m\times n}$. 

Then we have $\nu\in \Meas^p_\tp{qc}$ if and only if $\nu$ satisfies the following two conditions:
\begin{enumerate}
\item\label{it:jensen} For all quasiconvex ${\bf E} \colon \R^{m\times n} \to \R$ with 
$\sup_{A\in \R^{m\times n}} \frac{|{\bf E}(A)|}{1+|A|^p} < \infty$, (no condition if $p=\infty$), the Jensen inequality
$$
  {\bf E}( \langle \nu, \mathrm{Id} \rangle ) \, \leq \; \int_{\R^{m\times n}} \! {\bf E}(A) \, \dd \nu(A) 
$$
holds;  
\item \label{it:pmoment} $\nu$ has finite $p$-th moment, in the sense \eqref{moments}. 
\end{enumerate}
\end{theorem}

\noindent In this connection we also recall that
\begin{equation*}\label{momentcondition}
  \Meas_{\qc}^{p} = \Meas^{1}_{\qc} \cap \bigl\{ \nu : \, \nu \mbox{ has a finite $p$-th moment } \bigr\}
\end{equation*}  
for $p \in (1,\infty ]$, see \cite[Corollary 1.8]{JK0} for $p<\infty$ and \cite{KZ} for $p=\infty$. 
\medskip

\begin{remark} \label{affine}
Note also that if $p\in (1,\infty)$, $\nu \in \Meas^p_\tp{qc}(\R^{m\times n})$ and  $\varphi$ is the weak limit of a sequence  $(\varphi_j)$ generating $\nu$,
then by \eqref{limit} 
$$
\D\varphi(x) \equiv \langle \nu, \mathrm{Id} \rangle
$$
and therefore $\varphi$ is affine.
\end{remark}
As a last remark, Theorem \ref{thm:KP} also motivates the definition of the class of measures which satisfy  Jensen's inequality with respect
to rank-one convex instead of quasiconvex functionals.

\begin{definition}
For $p\in [1,\infty]$, $\Meas^p_\tp{rc}$ is the set of those Borel probability measures $\nu$ in $\R^{m\times n}$ such that conditions (\ref{it:jensen}) and
(\ref{it:pmoment}) of Theorem \ref{thm:KP} hold, with the word \textit{quasiconvex} replaced by \textit{rank-one convex}.
\end{definition}
The set $\Meas^p_\tp{rc}$ agrees with the set of \textit{$p$-laminates}, see e.g.\ \cite[Definition 5.3]{Faraco2004}, although we will not use laminates in the sequel.

\subsection{Closed quasiconvexity and closed rank-one convexity}\label{closedqc}
We are now ready to introduce what we believe are the correct notions of rank-one convexity and quasiconvexity for extended-real valued functionals.

\begin{definition}\label{def:closedqc}
A functional ${{\bf E}}\colon \R^{m \times n} \to \overline \R$ is \textit{closed $\WW^{1,p}$-quasiconvex}
(respectively \textit{closed $p$-rank-one convex}) if for all $A\in \R^{m\times n}$ and all $\nu \in \Meas^p_\tp{qc}$ (respectively all $\nu \in \Meas^p_\tp{rc}$)
with $\langle \nu, \mathrm{Id} \rangle =A$ we have 
\begin{equation}\label{closed}
{\bf E}(A) = {\bf E}\left( \int_{\R^{m\times n}} \lambda \,\dd \nu(\lambda) \right) \leq \int_{\R^{m\times n}}^* {\bf E}(\lambda) \,\dd \nu(\lambda).
\end{equation}
\end{definition} 
In particular we note that if ${{\bf E}}\colon \R^{m \times n} \to \overline \R$ is closed $\WW^{1,p}$-quasiconvex, then it 
is also closed $\WW^{1,s}$-quasiconvex for every $s > p$.
Furthermore, closed $\WW^{1,p}$-quasiconvexity implies the standard $\WW^{1,p}$-quasiconvexity. In order to verify this, let $\phi \in \WW^{1,p}_{0}( \Omega , \R^m )$ and
$A \in \R^{m \times n}$ and recall Remark \ref{graddist}, where the probability measure $\nu_{A+\D \phi}$ was defined and shown to be a homogeneous $\WW^{1,p}$-gradient Young
measure with centre of mass $A$. Hence if $\mathbf{E}$ is a closed $\WW^{1,p}$-quasiconvex functional, then
$$
\fint_{\Omega}^* \! \mathbf{E}(A + \D \phi (x)) \, \dd x = \int_{\R^{m \times n}}^* \! \mathbf{E} \, \dd \nu_{A+\D\phi} \geq \mathbf{E}(A).
$$
For other approaches see for instance \cite{Muller, Pe, Rindler}.



 However, the converse is not true, and it is instructive to exemplify this point here. 
Similarly to Example \ref{ex:2pts}, the standard counterexamples are based on the following functional: Given a subset $\mathcal{U}\subset\R^{m\times n}$ we define
its characteristic function, in the sense of convex analysis, by
\begin{equation}
\label{eq:indicator}
\chiinf_{\,\mathcal{U}}(A) \equiv  \left\{
\begin{array}{ll}
  0 & \mbox{ if } A \in \mathcal{U},\\
  +\infty & \mbox{ if } A \in \M \setminus \mathcal{U}.
\end{array}
\right.
\end{equation}
We then choose the set $\mathcal U$ to be a special finite set of matrices called a \textit{$T_N$-configuration}, see \cite[Definition 2.1]{Forster2018}.
In particular such sets contain no rank-one connections, meaning that if $\mathcal T_N=\{A_1,\dots, A_N\} \subset \R^{2\times 2}$ then
\begin{equation}
\label{eq:noconnection}
\tp{rank}(A_i-A_j)>1\quad \tp{ for all } i\neq j.
\end{equation}
Kirchheim and Preiss \cite{K1,K2} provide the following optimal example concerning $T_N$-configurations with $N \geq 5$, see also \cite{Forster2018}.

\begin{example}[five points]\label{ex:5pts}
It is possible to construct a set $\mathcal T_5=\{A_1,\dots, A_5\}$ of five matrices with the following property:
\begin{equation}
\label{eq:trivsols}
\tp{There is } A\not \in \mathcal{T}_5 \tp{ and }\varphi \in W^{1,\infty}_0(\Omega,\R^2) \tp{ with } A + \D \varphi \in \mathcal{T}_5 \tp{ a.e.}
\end{equation}
Thus the functional $\chiinf_{\,\mathcal T_5}$ is rank-one convex but not quasiconvex, it does not satisfy \eqref{eq:qc23}. Note that this
in particular provides us with an example of a rank-one convex functional on symmetric $2$-by-$2$ matrices which is not quasiconvex. It is an extended
real-valued and lower semicontinuous functional and it is not clear if one could construct such an example which is continuous, let alone real-valued.
\end{example}

The property \eqref{eq:trivsols} fails for all configurations of four matrices as proved by Chlebik and Kirchheim in \cite{Chlebik}. This gives us the
following example:

\begin{example}[four points]\label{ex:4pts}
Let $\mathcal T_4=\{A_1,\dots, A_4\}$ be any set in $\R^{2 \times 2}$ satisfying \eqref{eq:noconnection}. Then \eqref{eq:trivsols} fails for the set. Thus the functional
$\chiinf_{\,\mathcal T_4}$ is both quasiconvex and rank-one convex. Moreover, as observed in \cite[Example 1.3]{JK}  the corresponding variational integral
$\int_{\Omega} \! \chiinf_{\, \mathcal{T}_{4}} \! (Du) \, \dd x$ is sequentially weakly lower semicontinuous on $\WW^{1,1}( \Omega , \R^{2})$. 
Nonetheless, if we assume the set $\mathcal{T}_4$ is a $T_4$ configuration, then there exists nontrivial $\nu \in \Meas^\infty_\tp{rc}(\mathcal T_4)\subseteq \Meas^\infty_\tp{qc}(\mathcal T_4)$
(see for instance \cite{Muller}), so $\chiinf_{\,\mathcal T_4}$ is neither closed $\WW^{1,p}$-quasiconvex nor closed $p$-rank-one convex for any $p \in [1,\infty ]$.
\end{example}

To summarize,  the above examples with Example \ref{ex:2pts} show for the extended real-valued functionals in $\R^{m\times n}$, with $m,n\geq 2$, the following relations
$$
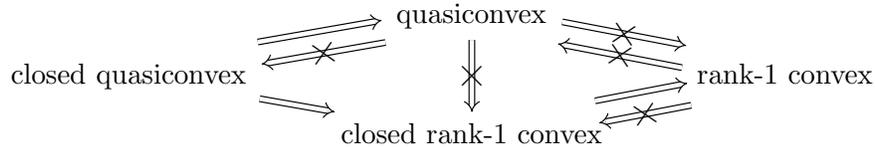
\begin{tikzcd}[row sep=tiny,arrows=Rightarrow]
& \tp{quasiconvex} \arrow[dr, shift left=0.9ex,degil]\arrow[dd,degil]\arrow[dl,shift left=0.9ex,degil]& \\
\tp{closed quasiconvex}\arrow[ur,shift left=0.9ex] \arrow[dr] & & \tp{rank-1 convex}  \arrow[shift left=0.9ex, ul,degil] \arrow[shift left=0.9ex,dl,degil]\\
& \tp{closed rank-1 convex} \arrow[ur,shift left=0.9ex]	& 
\end{tikzcd}
$$
and we emphasize that whether or not closed rank-one convexity implies either quasiconvexity or closed quasiconvexity is unknown if $m=2$, $n \geq 2$.
If $m>2$ then closed rank-one convexity is a strictly weaker notion \cite{Sverak92}.

As observed already by Morrey \cite{Morrey}, quasiconvexity of a given functional ${\bf E}\colon \R^{m\times n}\to \R \cup\{+\infty\}$ is intimately tied
to the sequential weak lower semicontinuity of the associated integral (compare also Remark \ref{rmk:moreintegrands}, Proposition \ref{prop:homogenization} and Theorem \ref{thm:KP}). 
 In particular, see e.g. \cite{Dacorogna2007}, for functionals with standard growth  $ |{\bf E}| \leq C(1+|\cdot|^p)$, the $\WW^{1,p}$-quasiconvexity is equivalent to sequential weak lower semicontinuity in  $\WW^{1,p}$. Here to be precise, for signed functionals  such as the determinant which take also negative values, one must restrict to sequences with fixed boundary values, as otherwise concentration effects
at the boundary can destroy the lower semicontinuity (see \cite{Dacorogna2007} for an example and \cite{ChJK} for some positive results).
 
 It is not known what are the weakest possible growth properties for functionals  still guaranteing this equivalence. In any case,  all growth conditions imply that $\mathbf{E}$ is real valued. Without going into too much detail, we remark that for {\it non-negative} functionals ${\bf E} \geq 0$ taking the value $+\infty$ and for $1\leq p<\infty$, we have, 
$$
\begin{tikzcd}[row sep=tiny,arrows=Rightarrow]
\tp{closed $\WW^{1,p}$-quasiconvex}\arrow[r,shift left=0.9ex] 				
& \arrow[l,shift left=0.9ex,degil]\tp{$\WW^{1,p}$-seq. wlsc}\arrow[r,shift left=0.9ex] 	 &		\arrow[l,shift left=0.9ex,degil]\tp{$\WW^{1,p}$-quasiconvex}			\\								
\end{tikzcd}
$$
To see that the implications above cannot be reversed, see Example \ref{ex:4pts}  for the first and Example \ref{ex:2pts} for the second. Indeed, by \cite[Example 3.5]{BM1} the functional in  Example \ref{ex:2pts} is quasiconvex but with planar waves,  described in the introduction, one sees that the functional is not sequentially weak lower semicontinuous.

The case of signed functionals that are
allowed to assume also the value $+\infty$, such as the local Burkholder functionals arising from our work, has not been considered before. Consequently, this requires a separate study of their lower semicontinuity properties, covered in the last Section 12.


Finally, note that our
approach here with \textit{pointwise} definitions of the associated variational integrals, such as \eqref{eq:variational}, is not the only possibility.
Following an old tradition going back to H.~Lebesgue, J.~Serrin and introduced in the current context by P.~Marcellini \cite{Marcellini}, a definition by
\textit{relaxation from smooth maps} is often more natural and desirable as it leads to variational integrals with better properties.
We intend to return to this elsewhere.

The characteristic function in \eqref{eq:indicator} also leads us to a natural notion of quasiconvexity for sets:
\begin{definition}\label{qcset}
A subset $\mathcal{U}\subset\R^{m\times n}$ is said to be \textit{$\WW^{1,p}$-quasiconvex} if
$$
\nu \in \Meas^p_{\qc}(\mathcal U) \quad \implies \quad  \langle \nu, \mathrm{Id} \rangle  \in \mathcal U.
$$
\end{definition}
Thus a closed set $\mathcal{U}$ is $\WW^{1,p}$-quasiconvex if and only if $\chiinf_{\,\mathcal U}$ 
is a closed $\WW^{1,p}$-quasiconvex functional.

We have been discussing $\WW^{1,p}$-quasiconvex functionals, but we are yet to mention the role played by $p$.
For later reference we record the following example, which displays the dependence of quasiconvexity on $p$.

\begin{example}\label{ex:qccone}
Let $K\geq 1$. The $K$-quasiconformal cone 
$$
Q_{2}(K)\equiv \bigl\{ A\in \R^{2\times2}:|A|^2\leq K \det A \bigr\}
$$
is a $\WW^{1,p}$-quasiconvex set if and only if $p \geq \tfrac{2K}{K+1}$.
\end{example}

The statement for $p>\frac{2K}{K+1}$ was shown in \cite{AstalaFaraco02} and follows easily from Theorem \ref{thm:YMqc} below, while the reader can find a
proof of the case $p<\frac{2K}{K+1}$ in \cite{Faraco2004}. We also refer the reader to \cite{YZ} where higher dimensional versions of Example \ref{ex:qccone} are discussed in detail. 
The borderline case $p=\frac{2K}{K+1}$ is more subtle and follows from a variant of Theorem \ref{thm:soft}, as will be shown elsewhere.


\subsection{Rank-one convexity and radial maps}

To conclude this section we relate rank-one convexity to radial maps. For simplicity and because it is our main focus we will only consider planar maps.
\begin{definition}
A \textit{radial map} is a map $\phi \in \WW^{1,\infty}(\mb D)$ of the form 
$$\phi(z)=\rho(r)\frac{z}{r}, \qquad r\equiv |z|,$$ where $\rho\colon[0,1]\to \R$ is a Lipschitz function such that $\rho(0)=0$.
If $\rho\colon [0,1]\to [0,+\infty)$ we say that $\phi$ is a \textit{radial stretching}.
\end{definition}
 It is easy to verify that for a radial map $\phi$ we have a.e.\ in $\mb D$ the identities
\begin{gather}
\label{eq:derivativephi}
\p_z \phi = \frac 1 2 \Big(\dot \rho(r) +\frac{\rho(r)}{r}\Big),\qquad
\p_{\bar z} \phi = \frac 1 2 \Big(\dot \rho(r) -\frac{\rho(r)}{r}\Big) \frac{\, z \, }{\bar z}.
\end{gather}

Rank-one convex functionals are quasiconvex along radial maps, see e.g.\ \cite{Sivaloganathan1988} or \cite{Ball1990}.
Here we state a slightly more general version of this result in order to account for extended real-valued functionals.

\begin{lemma}\label{lemma:rcradialstretchings} 
Let ${\bf E}\colon \R^{2\times 2}\to \overline \R$ be a rank-one convex functional such that $\{t \,\Id: t>0\}$ is in the interior of
$\tp{dom}({\bf E})$. Then, for a radial map $\phi\colon \mb D\to \C$ such that $\D \phi \in \tp{int}(\tp{dom}({\bf E}))$ a.e., we have
$$
{\bf E}(\Id)\leq \fint_{\mb D}^* {\bf E}(\D \phi(z)) \,\dd m(z). 
$$
\end{lemma}
\begin{proof}
The essential range of $\D \phi$ is the smallest closed set $\mathcal U$ such that $\D \phi \in \mathcal U$  a.e.\ in $\mb D$; since $\phi$ is Lipschitz,
clearly this set is compact. Thus, since $\{t\, \Id: t>0\}$ is contained in the interior of $\tp{dom}({\bf E})$, after mollifying ${\bf E}$ 
we can without loss of generality assume that ${\bf E}$ is smooth in the interior of its domain.

Identifying $\D \phi=(\p_z \phi, \p_{\bar z} \phi)$, by \eqref{eq:derivativephi} we may write
$$
\D \phi = \left(\frac{\rho(r)}{r},0\right)+\frac 1 2\left(\dot \rho(r)-\frac{\rho(r)}{r}\right) \left(1,\frac{z}{\bar z}\right)
$$
and notice that the second term corresponds to a matrix with rank one, since it is singular, cf. \eqref{eq:detconformal}.
Thus, by rank-one convexity,  we obtain
\begin{equation*}
\label{eq:rcinequality}
      {\bf E}(\D \phi)\geq {\bf E}\left(\frac{\rho(r)}{r} ,0\right) + \frac 1 2 \Big\langle {\bf E}'\left(\frac{\rho(r)}{r},0\right), \left(\dot \rho(r)
      -\frac{\rho(r)}{r}\right)\left(1,\frac{z}{\bar z}\right)\Big\rangle.
\end{equation*}
Integrating over $\mb D$, and using $\int_{\mathbb S^{1}} (1,\frac{z}{\bar z}) \,\dd \theta=\int_{\mb S^1}(1,0) \,\dd \theta$,  we find 
\begin{align*}
& \int_{\mb D} {\bf E}(\D \phi) \,\dd m(z)
 \\ & \geq \int_{\mb S^{1}} \int_0^1
 \left[{\bf E}\left(\frac{\rho(r)}{r} ,0\right) + \frac 1 2 \Big\langle {\bf E}'\left(\frac{\rho(r)}{r},0\right),
   \left(\dot \rho(r)-\frac{\rho(r)}{r}\right)\left(1,0\right)\Big\rangle\right]r\, \dd r \,\dd \theta\\
&  = \frac 1 2 \int_{\mb S^{1}} \int_0^1 \frac{\dd }{\dd r}\left(r^2 {\bf E}\bigg(\frac{\rho(r)}{r} \Id\bigg)\right) \dd r\, \dd \theta \\
& = \pi {\bf E}(\Id).
\end{align*}	
Note that, since $\rho$ is Lipschitz continuous and ${\bf E}$ is locally Lipschitz continuous in a neighborhood of $\{t\, \Id:t>0\}$, see e.g.\ \cite{Dacorogna2007},
the application of the fundamental theorem of calculus in the last step is justified.
\end{proof}

\section{Principal mappings} \label{sec:principal} 

Throughout this paper we will use principal quasiconformal maps to test quasiconvexity
and to generate the required gradient Young measures. 

One can of course always use  the flexibility  allowed by Theorem \ref{generates} in  choosing the generating sequences.
However, in proving quasiconvexity for the local Burkholder functional this is a delicate issue, since here the natural deformations
of maps typically deform  also the boundary values, c.f. Subsection \ref{deformations} for a discussion.
We solve this issue by using the principal maps, which allow deformations yet carry enough information to control the center of mass of
the limiting gradient Young measures; for typical cases see e.g. the next section. 

In the end, it turns out that quasiconvexity with respect to the principal maps determines the lower semicontinuity properties
for a large class of functionals, similarly as the classical quasiconvexity does for functionals with standard growth.

In this section we hence briefly recall the basic properties of such maps. 
First, for $K\geq 1$, a planar ${\WW}^{1,2}_{\loc}$ map $f\colon \Omega \to \C$ is \textit{$K$-quasiregular} if
$$
|\D f(z)|^2 \leq K J_f(z)  \qquad \tp{for a.e.\ } z \tp{ in } \Omega;
$$
if in addition $f$ is a homeomorphism, we say that $f$ is \textit{$K$-quasiconformal}.
Any such mapping satisfies the Beltrami equation
\begin{equation}\label{Belt1}
\, f_{\bar{z}} = \mu(z) f_z
\end{equation}
where  $\mu$ is a measurable function with
$\| \mu\|_\infty  \leq  k \equiv  \frac{K-1}{K+1} <1$. We are especially interested in homeomorphic solutions of \eqref{Belt1} with a certain normalization:
\smallskip

\begin{definition}\label{prmap}
A map $f\colon \C \to \C$ is said to be a \textit{principal map} if:
\begin{enumerate}
\item $f$ is a $\WW^{1,1}_\tp{loc}(\C)$-homeomorphism and 
\item $f$ is conformal outside  $\DD$, with Laurent series 
\begin{equation}\label{princip1}
f(z) = z +  \frac{b_1}{z} + \sum_{j=2}^\infty \frac{b_j}{z^j}, \qquad |z| > 1.
\end{equation}
\end{enumerate}
\end{definition}


The classical area formula, a quick consequence of  Green's theorem \cite[Theorem 2.10.1]{AIM} gives for any $\WW^{1,2}_\tp{loc}(\C)$-principal map 
the identity
\begin{equation}\label{areafmla} 
\int_{\mb D} \! J_f(z) \, \dd m(z) = \pi \left( 1 - \sum_{j=1}^\infty j |b_j|^2 \right),
\end{equation}
which controls the size  of the coefficients in \eqref{princip1}.

For instance, the Jacobian of a Sobolev  homeomorphism does not change sign,  thus $J_f(z) \geq 0$ for any principal map and hence
$|b_1| \leq1$ in \eqref{princip1}. Indeed, even if the principal map has only $\WW^{1,1}_\tp{loc}$-regularity we still have the bound  
$$
\sum_{j=1}^\infty j |b_j|^2 \leq 1,
$$
and in particular the condition $|b_1| \leq 1$ holds. But if in either case $|b_1| = 1$, then the area formula forces all other
coefficients to vanish, and that would force $f$ affine and non-injective on the unit circle. Thus $|b_1| < 1$ for every principal homeomorphism as in Definition \ref{prmap}.

    
We hence find that for any principal map as in \eqref{princip1}, the associated linear operator
$$
A_f(z) \equiv z+b_1 \bar z
$$
is a homeomorphism with
\begin{equation}\label{asympto} 
\det( A_f ) > 0.
\end{equation}
Principal maps are therefore, in a sense,  close to having affine boundary values on the unit circle, but are yet flexible enough to allow deformations of maps.


As another aspect of this view,  the area formula \eqref{areafmla}  implies \cite[Corollary 2.10.3]{AIM} that 
\begin{equation}\label{Kareafmla} 
f\; \mbox{ is $K$-quasiconformal } \quad  \Longrightarrow  \quad A_f \; \mbox{ is $K$-quasiconformal}.
\end{equation}
This is not immediate since the set of $K$-quasiconformal linear maps is not convex.

The notion of a principal map is very natural also since for each coefficient $\mu$ supported in the unit disk with $\| \mu\|_\infty < 1$,
there is a unique  $\WW^{1,2}_\tp{loc}(\C)$ principal solution $f = f_\mu$ to \eqref{Belt1}, cf.\ \cite[Theorem 5.3.2]{AIM}.
This extends even to suitable  degenerate  Beltrami equations, see \cite{IwaSve}. 

Indeed, a simple way to find principal solutions for the given coefficient $\mu$ is via  the Cauchy transform
$$
\mathbf C \varphi (z) =  \frac{1}{\pi} \int_\mathbb C \frac{\varphi(\xi)}{ z - \xi }\, \textnormal d \xi.
$$ 
One now looks for a solution in the form 
 \begin{equation} \label{Cauchy}
f(z) = z + (\mathbf C \omega) (z),
\;\;\;\textnormal{with} \;\; \omega 
\, \in \,  \mathscr \LL^2(\mb D),
\end{equation}
and the derivative $f_{\bar{z}} \equiv  \omega $ is then found by a Neumann-series argument.
Namely, if $\mathbf S$ is the \textit{Beurling--Ahlfors transform}, i.e.\ a Calder\'on--Zygmund singular integral operator bounded in $\LL^s(\C)$ for all
$1 < s < \infty$ and defined by
\begin{equation}\label{Beur1}
\mathbf S \varphi (z) \; \equiv \;- \frac{1}{\pi}  \int_\mathbb C \frac{\varphi(\xi)}{ (z - \xi)^2 }\dd \xi,
\end{equation}
then \cite[(5.8)]{AIM} shows that  
\begin{equation}\label{Beur2}
f_{\bar{z}}  = (I - \mu  \mathbf S)^{-1} \mu, \qquad f_z = 1 + (I - \mu  \mathbf S)^{-1}  \mathbf S \mu.
\end{equation}
Here $\| \mathbf S\|_{L^2(\C)} = 1$ while by  \cite[Theorem 14.0.4]{AIM}, the operator $I - \mu  \mathbf S$ is invertible on $\LL^{s}(\C)$ whenever
\begin{equation}
\label{eq:invertBelt}
1 +  \| \mu\|_\infty < s < 1+1/ \| \mu\|_\infty,
\end{equation}
with the operator-norm of the inverse $\| (I - \mu  \mathbf S)^{-1} \|_{\LL^s(\C)}$
bounded by a constant that depends  only on $s$ and $ \| \mu\|_\infty$. 

Finally, to show that the mapping defined by \eqref{Beur2} is a homeomorphism requires more work; for details see \cite{AIM}. 
\medskip

From \eqref{Beur2} we obtain global higher integrability bounds for the derivatives of principal solutions
to  the Beltrami equation \eqref{Belt1}, in particular
\begin{equation}\label{Belt2}
\| f_{\bar{z}}  \|_{\LL^{s}(\C)} + \| f_z -1 \|_{\LL^{s}(\C)} \leq C_s(K) < \infty,
\end{equation} 
whenever
\begin{equation}\label{Holder}
\frac{2K}{K+1} < s < \frac{2K}{K-1},
\end{equation}
where this last condition comes from \eqref{eq:invertBelt}, recalling that $\|\mu\|_\infty\leq  \frac{K-1}{K+1}$.

Note also that for exponents $s > 2$ and $\alpha < 1 -2/s$, the Cauchy operator $\mathbf{C} \colon \LL^{s}(\mb D) \to \CC^\alpha(\C)$ is compact.
Therefore  from \eqref{Cauchy} and \eqref{Beur2} we see that  for each $k <1$ the family
\begin{equation}\label{Belt3}
{\mathscr F}_k \equiv  \biggl\{ f \mbox{ is a principal solution to \,} \eqref{Belt1} \, \mbox{ with } \| \mu \|_{\infty} \leq k  \biggr\}
\end{equation}
is normal, i.e.\ every sequence of ${\mathscr F}_k $ contains a subsequence converging uniformly on $\C$. The limit, too, belongs to 
${\mathscr F}_k $ since any (non-constant) limit of a  uniformly converging sequence of $K$-quasiconformal maps is $K$-quasiconformal. 
\smallskip

%
%

\section{Quasiregular gradient Young measures}\label{Young}

In this section we consider homogeneous gradient Young measures which are generated by sequences with suitable bounds on their distortion;
our goal is to show that such measures can be generated also by principal maps with the same distortion bounds. The case where the sequence has
uniformly bounded distortion was studied by the first two authors in \cite{AstalaFaraco02}, but here we are also interested in the more general
case where only the integrals of the distortion functions are uniformly bounded.

To start with, recall that the effective domain \eqref{effective} of the local Burkholder functional ${\mathcal B}_K (A)$ in \eqref{eq:locBurk} is
the set of all $K$-quasiconformal matrices 
$$
Q_2(K) \equiv \{ A \in \R^{2 \times 2} :  |A|^2 \leq K \det A \}.
$$

\noindent
This takes us to study the gradient Young measures supported on $Q_2(K)$. 

We  begin with a modification of results from \cite{AstalaFaraco02}: 

\begin{theorem} \label{thm:YMqc}
Let $\frac{2K}{K+1}<  s < \frac{2K}{K-1}$ and let $\nu\in \mathscr{M}^s_\tp{qc}(Q_2(K))$ be such that
\begin{equation}  \label{CM2}
\langle \nu, \mathrm{Id} \rangle = A, \qquad Az = z + a \bar z.\end{equation}
Then there is a sequence of $K$-quasiconformal principal maps which generates in $\DD$ the Young measure $\nu$.
\end{theorem}


%


\noindent
Here recall that any matrix in $ Q_2(K)$ is a scalar multiple of $Az = z + a \bar z$ where $|a| \leq \frac{K-1}{K+1}$.
Thus the choice  \eqref{CM2} is merely a normalisation.
Moreover, recall  that if $f \in \WW^{1,2}_{\loc}(\C)$ is the weak limit of the generating sequence given by Theorem \ref{thm:YMqc}, then
$f$ is $K$-quasiconformal and principal; this follows from the compactness of the family ${\mathscr F}_k$ in \eqref{Belt3}. In addition,  \eqref{limit} shows that 
$$
\D f(x) = \langle \nu, \Id \rangle \qquad {\rm a.e.} \; x \in \mb D.
$$  
Thus also the center of mass $\langle \nu, \Id \rangle$ is a $K$-quasiconformal matrix and for some $|a| \leq \frac{K-1}{K+1}$,
$$
f(z) = \left\{
\begin{array}{ll}
  z+a\bar{z} & \mbox{ if } |z| \leq 1,\\
  z+\tfrac{a}{z} & \mbox{ if } |z| > 1.
\end{array}
\right.
$$
For the proof of Theorem \ref{thm:YMqc} we  borrow an auxiliary result from  
\cite[Lemma 4.1]{AstalaFaraco02}:

\begin{lemma} \label{gen.seq}
Suppose $\frac{2K}{K+1}< s < \frac{2K}{K-1}$ and  $\nu\in \mathscr{M}^s_\tp{qc}(Q_2(K))$ is generated by a sequence
$(\phi_j)$ such that $( |\D\phi_j|^s)$ is equiintegrable in $\mb D$.

Then there are measurable functions $\mu_j\colon\mb D\to \C$, with $\| \mu_j\|_\infty  \leq \frac{K-1}{K+1}$, such that 
$$
\lim_{j\to \infty} \| {\partial_{\bar z}} \phi_j - \mu_j \partial_z \phi_j \|_{\LL^s(\mb D)} = 0.
$$
\end{lemma}



\begin{proof}[Proof of Theorem \ref{thm:YMqc}]
Let $(\phi_j) \subset \WW^{1,s}(\mb D, \C)$ be a sequence generating $\nu$,  given by Theorem \ref{generates}, for which $(|\D \phi_j|^s)$ is equiintegrable. We then set
$$
\eta_j \equiv \bigl( \partial_{\bar z}\phi_j - \mu_j \partial_z \phi_{j}\bigr)  \chi_{\mb D} \in \LL^s(\C),
$$
where the $\mu_j$ are given by Lemma \ref{gen.seq}.
As in \eqref{Cauchy}, via the Cauchy transform we find global solutions  $\omega_j$ to
$$ \partial_{\bar z} \omega_j -  \mu_j \partial_z \omega_j =  \eta_j$$ with $|\D\omega_j| \in \LL^s(\C)$, simply by  letting
$$
\p_{\bar z}  \omega_j = (I - \mu_j  \mathbf S)^{-1} \eta_j, \quad  \omega_j = \mathbf C( \partial_{\bar z}  \omega_j).
$$
In particular, $\| \D\omega_j\|_{\LL^s(\C)} \leq C_s(K)  \| \eta_j \|_{\LL^s(\C)} \to 0$ as $j \to \infty$.

By Theorem \ref{generates}, also the sequence of maps
$$
\psi_j \equiv  \phi_j - \omega_j \in \WW^{1,s}(\mb D)
$$
 generates the given Young measure $\nu$, since $\| \D\psi_j - \D\phi_j \|_{\LL^s (\mb D)} \to 0$. 
Further, the maps $\psi_j$ are all $K$-quasiregular, $\partial_{\bar z} \psi_j -  \mu_j \partial_z \psi_j = 0$ in $\mb D$, but they need not be  homeomorphisms.

On the other hand, in Section \ref{sec:principal} we saw that there  is a  principal solution $f_j \in \WW^{1,p}_\textup{loc}(\C, \C)$ to the Beltrami equation 
$$
\partial_{\bar z} f_j -   \chi_{\mb D} \mu_j \partial_{z} f_j = 0.
$$
Then by Stoilow's factorization  \cite[Theorem 5.5.1]{AIM} we have 
$$
\psi_j = h_j \circ f_j,
$$
where the maps $h_j$ are holomorphic in $f_j(\mb D)$.  
 
We can next estimate as in \cite[(4.11)]{AstalaFaraco02}. Namely, as $J_{f_j}(z) \leq  |\D f_j(z)|^2$, by the change of variables formula and \eqref{Belt2} we have
\begin{align} \label{holom.factor}
\int_{f_j(\mb D)} |h_j'(w)|  & \, \dd m(w)  = \int_\mb D  |h_j'(f_j (z))| J_{f_j}(z) \, \dd m(z)\\
& \leq \left( \int_\mb D  |h_j'(f_j (z))|^s |\D f_j(z)|^s \right)^{1/s} \left( \int_\mb D |\D f_j(z)|^t\right)^{1/t} \nonumber \\
& \leq C_t(K) \left( \int_\mb D  |\D \psi_j (z)|^s \right)^{1/s}, \nonumber
\end{align}
where $t$, the H\"older conjugate of $s$, also satisfies  \eqref{Holder}.
Thus the norms $\| h_j'  \|_{\LL^{1}(f_j(\mb D))}$ of the derivatives of the holomorphic factors are uniformly bounded.
Choosing a point $x_0 \in \mb D$ and adding a constant to elements of the generating sequence $\phi_j$, we can assume that $\psi_j(x_0) = x_0$, i.e.\ that 
$h_j$ takes $f_j(x_0)$ to $x_0$. Choosing then a subsequence such that $f_j(x_0)$ converge, we see that the holomorphic functions
$h_j \colon f_j(\mb D) \to \C$ form a normal family.

All in all, taking subsequences we can assume that 
$f_j \to f \in {\mathscr F}_k$  uniformly on $\C$ and 
$h_j \to h$  locally uniformly on $f(\mb D)$,
where $h$ is analytic on $f(\mb D)$. Further, we have the weak convergence
\begin{equation}  \label{CM3}
\D \psi_j = h_j'(f_j(z)) \D f_j \rightharpoonup  \D(h \circ f) \quad \mbox{ in } \; \LL^{s}(\mb D).
\end{equation}

Since $(\psi_j)$ generate the given Young measure $\nu$,  we see from \eqref{CM2} and Remark \ref{affine} 
 that $\D(h \circ f)(z)=A$ for a.e.\ $z \in \mb D$. 
 In particular, this means  that $f$ has the complex dilatation
$\mu = a\chi_{\DD}$. 
But given any compactly supported dilation with $\| \mu \|_\infty < 1$, this time $\mu = a \chi_{\DD}$, there is a unique $W^{1,2}_{\rm loc}$-principal mapping with this dilatation \cite[Theorem 5.1.2]{AIM}. So  $f$  must be equal to 
$$ 
f(z)=\begin{cases}
A(z)= z + a \bar z, & \tp{if } |z|\leq 1, \\ z+ \frac a z, & \tp{if } |z|\geq 1.
\end{cases}
$$ 
It follows that $\D h = \Id$ in $f(\mb D)$ and so we see that $h_j'(f_j(z)) \to 1$ locally uniformly on $\mb D$ as $j \to \infty$.
Thus with \eqref{Belt2}, for any $r<1$, 
$$
\|\D \psi_j - \D f_j\|_{\LL^{s}(\mb D(0,r))} = \| (h_{j}' \circ f_j - 1) \D f_j\|_{\LL^{s}(\mb D(0,r))} \to 0.
$$
Therefore, again by Theorem \ref{generates}, 
$( f_j)$ and $(\psi_j)$ generate the same (homogeneous) Young measure $\nu$ in every disc compactly contained in $\mb D$.
Since the sequence $(|\D f_j|^s)$ is equiintegrable over $\mb D$, and as $\mathscr L^2(\mb S^1)=0$, this finally shows that the sequence of principal maps $( f_j )$ generates  $\nu$ in $\mb D$. 
\end{proof}

Combining Theorem \ref{thm:YMqc} with \eqref{Belt2} we obtain the following consequence,  cf.\  \cite[Corollary 1.6]{AstalaFaraco02}: 

\begin{corollary}
\label{cor:higherintegrability}
If $\nu\in\Meas^s_\tp{qc}(Q_2(K))$ for some  $\frac{2K}{K+1}< s$, then $\nu \in \Meas^p_\tp{qc}(Q_2(K))$ for all $p<\frac{2K}{K-1}$.
\end{corollary}

We conclude  this section with  a version of Theorem \ref{thm:YMqc} for gradient Young measures which are generated by sequences with  integrable distortion. 

\begin{theorem}\label{thm:intdistort}
Let $\nu \in \mathscr M^{2}_\tp{qc}(\R^{2\times 2}_+)$ be such that, for some $|a|<1$, 
\begin{equation}
\label{centermA}
\langle \nu, \Id \rangle = A,  \qquad A(z)=z + a \bar z.
\end{equation}
Assume that $\nu$ is generated by a bounded sequence $\{\psi_j\}\subset \WW^{1,{2}}(\mb D)$ of homeomorphisms such that 
for some $q > 1$,
$$\|K_{\psi_j}\|_{\LL^{q}(\mb D)}\leq C. 
$$
Then there is a sequence of  maps $f_j\colon \C\to \C$ such that:
\begin{enumerate}
\item $f_j$ are principal maps;
\item For each $r<1$, the sequence $(f_j |_{\mb D(0,r)}) \subset \WW^{1,2}(\mb D(0,r))$
is bounded  and generates $\nu$; 
\item $\psi_j= h_j \circ f_j$ for some conformal maps $h_j\colon f_j(\mb D)\to \psi_j(\mb D)$.
\end{enumerate}
\end{theorem}

The proof strategy is similar to that of Theorem \ref{thm:YMqc}.
However, since we are in a setting where the Beltrami equations are degenerate elliptic the argument is more subtle. For instance, one does not know if 
the sequence $(f_j )$ is bounded in  $\WW^{1,2}(\mb D)$, which causes technical problems.
On the other hand, for the applications we have in mind it suffices to consider generating sequences which consist of homeomorphisms.

The argument  relies crucially on the Stoilow factorization for maps with integrable distortion, due to Iwaniec and \v{S}ver\'{a}k  \cite{IwaSve}. 

\begin{proof}
The maps $\psi_j  \in \WW^{1,2}(\mb D)$ give us homeomorphic  solutions to the Beltrami equations 
\begin{equation}
\label{integrable.dist}
\p_{\bar z}  \psi_j -\mu_j \p_z \psi_j=0 \quad \text{in } \mb D,
\end{equation}
where $K_{\psi_j} \equiv \frac{1+|\mu_j|}{1-|\mu_j|}\in \LL^q(\mb D)$. 

By the Iwaniec--\v{S}ver\'{a}k theorem, given any non-constant solution $F \in \WW^{1,2}_\tp{loc}(\mb D)$ to \eqref{integrable.dist}, 
there is a principal solution
$f_j\in \WW^{1,1}_\tp{loc}(\C)$ to the Beltrami equation
$$
\p_{\bar z} f_j -  \chi_{\mb D} \mu_j\p_z f_j=0.
$$
The general Iwaniec--\v{S}ver\'{a}k argument requires merely that $K_{\psi_j} \in \LL^1(\mb D)$, see   \cite[Theorem 20.2.1]{AIM}, then however one has uniform ($j$-independent) bounds only in $ \WW^{1,1}_{\loc}(\C)$, see \cite[pp. 539-540]{AIM}. 
 
 For better uniform bounds in our setting, needed  for the local uniform convergence,  we first note that for any principal map $f(z)$ as in \eqref{princip1}, all  scalings $\, f(Rz)/R$, with $R > 1,$  are again principal maps. This combined with Koebe's theorem, see e.g.  \cite[Theorem 2.10.4]{AIM}, shows that $f_j(R \,\DD) \subset 2 R \, \DD$,  for every $R \geq 1$. In particular,
$$
  \int_{\DD(0,R)} J_{f_j} dm(z) \le 4 \pi R^2, \qquad R \geq 1.
$$
Moreover, in  our situation $K_{f_j} \in \LL^q(\DD)$, with uniformly bounded norms. These allow the simple estimate

\begin{align*}
 \int_{\DD(0,2)} |\D f_j|^{\frac{2q}{q+1}} dm(z) & \leq \int_{\DD} K_{f_j}^{\frac{q}{q+1}}J_{f_j}^{\frac{q}{q+1}} dm(z) + 16 \pi \leq 16 \pi (1+C),
\end{align*}
where $\frac{2q}{q+1}>1$.

On the the other hand, outside the unit disc the area formula gives   uniform bounds for the coefficients $b_{j,n}$ in the expansion $f_j=z+\sum b_{j,n} z^{-n}$.  In brief, we see that $ f_j -z \in \WW^{1,\frac{2q}{q+1} }_{\textrm{loc}}(\C)$, with finite and $j$-independent bounds for  the $L^{\frac{2q}{q+1}}(\C)$-norms of their derivatives.  Notice that latter bound follows
 as well from the boundedness of the Beurling Ahlfors transform.

Coming back to the generating sequence of homeomorphisms $\psi_j$, by  the Iwaniec-\v{S}ver\'{a}k  theorem \cite[Theorem 20.2.1]{AIM} these  
admit the Stoilow's factorization 
$$
\psi_j = h_j\circ f_j		
$$
where $h_j$ are holomorphic, in fact conformal homeomorpisms, in $f_j(\mb D)$.  

To control the  factors $h_j$, we may assume that $\psi_j(0) = 0$. Next note that 
since $(\psi_j)$ generates the measure $\nu$, with center of mass $A$ as in \eqref{centermA},  and since by \cite[Theorem 20.1.1]{AIM}  the sequence is  locally uniformly equicontinuous,  Rellich's theorem with  \eqref{limit} then shows that 
\begin{equation}  \label{weakLim}
\psi_j \to A \quad {\rm locally \;  uniformly \;  on} \quad \mb D.
\end{equation}
In particular,  for any $r <1$, $\psi_j(r \, \mb D)$ is a compact subset of $\psi_j(\mb D)$, bounded away from the boundary of $\psi_j(\mb D)$ by a uniform constant $\delta = \delta(r) > 0$.
 
Further, we combine these with   bounds given by the classical Koebe 1/4-theorem,  see e.g.
 \cite[Theorem 2.10.6]{AIM}. 
 Indeed, since $h_j^{-1}\colon \psi_j(\mb D) \to f_j(\mb D)$ is conformal, the theorem gives
 $$ \left| \left(h_j^{-1}\right)'(w) \right| \leq 4 \frac{\dist \left(h^{-1}(w), f_j(\partial \mb D)\right)}{\dist\left(w, \psi_j( \partial \mb D) \right)} \leq \frac{8}{\delta(r)} , \quad w  \in \psi_j(r \mb D).
 $$
 The last estimate again uses  Koebe's result  $f_j(\mb D) \subset 2 \mb D$, valid for every principal map.


Consequently, via the decomposition $\psi_j = h_j \circ f_j$, we then have  for any fixed $r < 1$ and for all $j \geq j_0(r)$ large enough,
$$
\int_{\DD(0,r)} |\D f_j|^2 \, \dd m(z) = \int_{\DD(0,r)} | (h_j^{-1})'\bigl(\psi_j(z)\bigr)|^2  |\D \psi_j|^2 \, \dd m(z) \leq \frac{8}{\delta(r)}C ,
$$
where by our assumptions
$$ C \equiv \sup_j  \int_{\DD} |\D \psi_j|^2 \, \dd m(z) < \infty.
$$

 It only remains to show that the $(f_j|_{\mb D(0,r)})$ generate $\nu$ for every $r<1$. For this we first 
 use again \cite[Theorem 20.1.1]{AIM}  to get local uniform convergence on $\DD$, this time for the sequence of principal maps $f_j$. 
Moreover,  \cite[Theorem 20.2.1]{AIM} gives the inverse maps $g_j = (f_j)^{-1}$  a uniform modulus of continuity, so that the local uniform limit $f(z) = \lim f_j(z)$ is a homeomorphism on $\DD$. 

Similarly, as principal maps the $f_j$ define on $ \C \setminus \overline{\DD}$ a family of normalised conformal maps. This gives us a limiting normalised conformal map $f(z) = \lim f_j(z)$ on $ \C \setminus \overline{\DD}$, with convergence is locally uniform in the exterior disc. 
Finally, by the above $\WW^{1,\frac{2q}{q+1} } _{\textrm{loc}}(\mathbb C)$-bounds  and  Rellich's theorem,  we see that  
$f_j(z) $ converges to $f(z) $ in $\LL^{2q} _{\textrm{loc}}(\C)$, with derivatives $D f_j(z) - D f(z) \rightharpoonup  0$
 in $\LL^{\frac{2q}{q+1}}(\C)$. 

 
 We next argue as  in the proof of Theorem \ref{thm:YMqc}, but this time use  the pointwise bounds  
$$J_{f_j}(z) \leq  |\D f_j(z)| J_{f_j}(z)^{1/2}.$$  With a change of variables these give
\begin{align} \label{holom.factor}
 \int_{f_j(\mb D)} |h_j'(w)|   \, \dd m(w) 
&\leq \int_\mb D  |h_j'(f_j (z))|  |\D f_j(z)| J_{f_j}(z)^{1/2} \, \dd m(z) \\ 
 \leq \left( \int_\mb D  |h_j'(f_j (z))|^2  |\D f_j(z)|^2 \right)^{1/2} & \left( \int_\mb D J_{f_j}(z) \right)^{1/2} 
 \leq \pi \left( \int_\mb D  |\D \psi_j (z)|^2 \right)^{1/2}, \nonumber
\end{align}
since by the  area formula \eqref{areafmla}, for principal maps the area $|f_j(\mb D)| \leq \pi$.
 Therefore  $h_j \to h$  locally uniformly on $f(\mb D)$,
where $h$ is  conformal. 

Fixing $r<1$, we can thus take advantage of the uniform convergence of $h_j'(f_j)$ to deduce
that 
\begin{equation}  \label{CM3}
\D \psi_j = h_j'(f_j(z)) \D f_j \rightharpoonup  \D(h \circ f) \quad \mbox{ in } \; \LL^2(r \mb D).
\end{equation}
Via \eqref{weakLim} this shows that  $h \circ f(z) = A(z)$ for   $z \in \DD$. Here $h$ is a 
conformal map from $f(\DD)$ to the ellipse $$ {\mathcal E} \equiv \{ z+ a \overline{z}: |z| < 1 \}. $$

 In  the exterior disc $\C \setminus {\overline{\DD}}$ the limit map $f(z)$ is itself conformal, with the Laurent series of a principal mapping. 
In particular, $f$ has the complex dilatation $\mu_f=\mu_A  \chi_{\mb D} = a \chi_{\mb D}$. Thus our last task, in analogy with the proof of Theorem \ref{thm:YMqc}, is  to identify the factor  $h(z)$.

However, globally $f$ has only the  $\WW^{1,\frac{2q}{q+1} } _{\textrm{loc}}$-reqularity. 
In general the unit circle is not removable for such 
 maps of finite distortion,  for a simple example see e.g. \cite[(5.38)]{AIM}. 
 
 Hence we need to use  the specifics of $f$, but this comes easily: the rational  transformation
 $ R(z) = z + \frac{a}{z} 
 $
 defines a conformal map $R: \overline{\C} \setminus {\overline{\DD}} \to \overline{\C} \setminus {\overline{\mathcal E}}$.
 We can thus consider the map
 $$ 
\Phi(z)=\begin{cases}
h^{-1}(z), & \tp{if }  z \in \mathcal E \\ f \circ R^{-1}, & \tp{if } z \in \overline{\C} \setminus {\overline{\mathcal E}}.
\end{cases}
$$ 
Here  make use the auxiliary mapping  defined by $G(z) = R(z)$ for $|z| \geq 1,$ and $G(z) =  A(z) $ for $|z| \leq 1$. 
 This is  bi-Lipschitz, with $\Phi(z) = f \circ G^{-1}(z)$ for a.e. $z \in \C$. As $f \in \WW^{1,\frac{2q}{q+1} } _{\textrm{loc}}(\mathbb C)$, we obtain the same regularity for $\Phi(z)$ as well. But  then Weyl's lemma shows 
 $\Phi$ to be analytic in all of $ \overline{\C}$.
 
 We have now shown that  $\Phi(z)$ is a Mobius transformation fixing $\infty$,  thus a compostion of scaling and translation. Since both $f(z)$ and $R(z)$ are principal maps, in fact $\Phi(z) = z$.

Consequently, $h(z) = z$,  so that 
$h'_j\circ f_j \to 1$ locally uniformly,
and exactly as in the proof of Theorem \ref{thm:YMqc} we have
  $$
\|\D \psi_j - \D f_j\|_{\LL^{2}(\mb D(0,r))} = \| (h_{j}' \circ f_j - 1) \D f_j\|_{\LL^{2}(\mb D(0,r))} \to 0.
$$
 This  shows that $(\psi_j)|_{\mb D(0,r)}$ and $(f_j)|_{\mb D(0,r)}$ generate the same gradient Young measure.
\end{proof}
\smallskip

\section{The Burkholder functionals} \label{sec:Bf}
In the literature there are a few slightly varying versions of the Burkholder functional.
For instance, Iwaniec  studies in \cite{Tade}  the functionals 
$$
{\mathscr B}^{\pm}_p(A) = \left| 1 - \frac{n}{p}\right| \, |A|^p \, \pm |A|^{p-n} \det A, \quad A\in \R^{n \times n}.
$$
Here, see \cite[Theorem 5]{Tade}, both functionals ${\mathscr B}^{+}_p$ and ${\mathscr B}^{-}_p$ are \textit{rank-one convex},
and in two dimensions 
${\mathscr B}^{+}_p$ is convex in the directions of matrices with
nonnegative determinant, while ${\mathscr B}^{-}_p$ is convex in the directions of matrices with nonpositive determinant, see
\cite[Proposition 12.1]{Tade}. In fact, precomposing with a reflection one can interchange the functionals ${\mathscr B}^{\pm}_p$.
In any case,   our choice \eqref{Burk} amounts to
$${\bf B}_p(A) = (p/2){\mathscr B}^{-}_p(A),$$ with $n=2 \leq p < \infty$, which is a rank-one convex functional.

In conformal coordinates, cf.\ \eqref{eq:confcoords}, the Burkholder functional can be written as  
${\bf B}_p(A) = |A|^{p-1} \bigl( (p-1)|a_-| - |a_+| \bigr)$ and, in particular,  
\begin{equation} \label{Bconf}
{\bf B}_p(\D f) = |\D f|^{p-1} \left( (p-1)|f_{\overline{z}}| - |f_z| \right).
\end{equation}
  One easily checks that  ${\bf B}_p( \Id ) = -1$ and that, for $p= \frac{2K}{K-1}$,
\begin{equation*}
\label{eq:levelsetBp}
\bigl\{ {\bf B}_p \leq 0 \bigr\} = \bigl\{ A \in \R^{2 \times 2} : \, |A|^2 \leq K \det(A) \bigr\} = Q_2(K),
\end{equation*}
the cone of $K$-quasiconformal matrices introduced in Example \ref{ex:qccone}.
\medskip

\subsection{Weighted integral estimates and the Burkholder functionals} \label{deformations}

The proof of closed quasiconvexity of the local Burkholder functional \eqref{eq:locBurk},
with domain the  quasiconformal cone $Q_2(K)$, is based on two fundamentally different methods, both essential for the argument.
The first,  developed in \cite{AIPS12}, establishes  optimal weighted integral bounds for $K$-quasiconformal principal mappings,
via holomorphic motions and special complex interpolation. On the other hand, the second method, developed in this work,
analyses the interaction of the Burkholder functional with the gradient Young measures generated by principal mappings. 

For  the optimal integral estimates, 
recall first that 
principal solutions $f = f_\mu$ to the Beltrami equation  \eqref{Belt1} allow holomorphic deformations \cite[Section 12]{AIM}.
That is, given the coefficient $\mu(z)$ with $\| \mu \|_\infty < 1$, one can construct families of coefficients $\mu_\lambda(z)$ depending
holomorphically on the parameter $\lambda \in \mb D$, such that $\| \mu_\lambda \|_\infty < 1$ for all $\lambda \in \mb D$, and $\mu_{\lambda_0} = \mu$
for a suitable $\lambda_0 \in \mb D$. Typically one also requires that  $\mu_0 \equiv 0$.

In this setting, for each $\lambda \in \mb D$ there is a unique principal solution to \eqref{Belt1}  with coefficient $\mu_\lambda(z)$.
In fact, this family of solutions defines a {\it holomorphic motion} 
\begin{equation} \label{holomotion} \Phi(\lambda,z) \equiv f_{\mu_{\lambda}}(z),
\end{equation}
 that is 
$ \lambda \mapsto \Phi(\lambda,z) $ is holomorphic, $z \mapsto \Phi(\lambda,z) $ is injective and $\Phi(0,z) = z$, c.f. \cite[Section 12]{AIM}.

Furthermore,  the gradients of the above holomorphic deformations \eqref{holomotion} and their $\LL^p$-norms allow optimal  interpolation bounds,
similar to those in the classical  Riesz-Thorin complex interpolation. For details see  \cite[Lemmas 1.4 and 1.6]{AIPS12}.

To combine all this with the Burkholder functional, note that for any given solution $f$ to the Beltrami equation 
\eqref{Belt1}  the Burkholder functional may be written equivalently as 
\vspace{.01cm}
\begin{equation}\label{Beltqc}
 {\bf B}_p\bigl(Df(z)\bigr) =  \Bigl(\frac{p |\mu(z)}{1+|\mu(z)|} - 1\Bigr) \bigl(|f_z(z)| + |f_{\overline{z}}(z)| \bigr)^{p}.
\end{equation}

This suggests that for quasiconformal principal mappings  one should approach the Burkholder integrals via $\LL^p$-estimates in
the appropriate weighted spaces. Indeed, using the deformations \eqref{holomotion} and combining \eqref{Beltqc} with the  optimal
weighted $\LL^p$-bounds from  the above complex interpolation  
leads to the following result.

\begin{theorem} \cite[Theorem 3.5]{AIPS12}  \label{closedQC2}
Suppose $f \colon \C \to \C$ is a principal solution to the Beltrami equation
\begin{equation}\label{Belt4}
 f_{\overline{z}} = \mu f_z, \quad | \mu(z)| \leq k \chi_{\DD}(z), \quad 0 \leq k < 1,
\end{equation}
Then for all exponents $2 \leq p \leq 1 + 1/k$, we have
\begin{equation}\label{Belt55}
 \fint_{\DD} \Bigl( 1-\frac{p |\mu(z)|}{1+|\mu(z)|} \Bigr) \left(|f_z(z)| + |f_{\overline{z}}(z)| \right)^{p} \dd m(z) \leq 1. 
\end{equation}
\end{theorem} 
\medskip

Finally, comparing now \eqref{Beltqc} and \eqref{Belt55}, we may rewrite Theorem \ref{closedQC2}  as follows

\begin{theorem} \cite[Theorem 1.3]{AIPS12}   \label{closedQC}
Suppose $f \colon \C \to \C$ is a principal solution to the Beltrami equation \eqref{Belt4}. Then for all $p \in [2,1+1/k]$
we have
\begin{equation} \label{Burk123}
{\bf B}_p(\Id)\leq \fint_\mb D {\bf B}_p(\D f(z)) \,\dd m(z). 
\end{equation}
\end{theorem} 
\smallskip

In particular, if a quasiconformal map $f$ on $\mb D$ has identity boundary values, 
it extends trivially to a principal map of $\C$. Thus the above result shows that the Burkholder functional ${\bf B}_p$ is
quasiconvex at identity when tested with $K$-quasiconformal  maps, under the condition that $2\leq p \leq \frac{2K}{K-1}$.
Furthermore, the equality in \eqref{Burk123} occurs for a large class of radial mappings, see for instance Subsection \ref{Bradial} below.


\begin{remark}
Note that holomorphic deformations of solutions $f$ to  \eqref{Belt4}, such as $\lambda \mapsto f_{\mu_{\lambda}}(z)$ above, in general change
the boundary values of the mapping, even if the original map $f$ has identity boundary values  on $\partial \mb D$. 
For this reason, in particular,  the principal mappings and their integral bounds are indispensable  for the quasiconvexity estimates   \eqref{Burk123}.
\end{remark}
\medskip

On the other hand,  to prove quasiconvexity bounds such as \eqref{Burk123} 
for quasiconformal maps on $\mb D$ with  linear boundary values $A \neq \Id$, 
different methods appear necessary. Here we will make extensive use of the gradient Young measures discussed in Sections \ref{sec:prelims} and \ref{Young}.
However, even for these Theorem \ref{closedQC} is required as their basis and  the starting point for the estimates they provide, see e.g. Proposition \ref{Id}.


\begin{remark}
Note that, in general, at the endpoint exponent $p=1 + 1/k$ 
we have $\D f \notin \LL^{p}_{\loc}$ for the  solutions \eqref{Belt4}: simple examples are obtained, for instance, by considering radial stretchings. Nonetheless, we always have $\D f \in \tp{weak-}\LL^p_\tp{loc}$ \cite[Theorem 13.2.1]{AIM}.
A surprising feature of Theorem \ref{closedQC} is that, even if ${\bf B}_p$ has $p$-growth, we are able to test the quasiconvexity inequality \eqref{Burk123} with maps which are \textit{not} in $\WW^{1,p}_{\loc}$, and in particular ${\bf B}_p$ is \textit{integrable} along such maps. One can regarded this as an extension to the planar quasiconformal setting of the results of \cite{IS}, where it is shown that $\det \D f\in L^1_\tp{loc}$ if $f\in \tp{weak-}W^{1,n}_\tp{loc}$ is orientation-preserving.
\end{remark}

\vspace{.1cm}

\subsection{Burkholder functional and radial mappings}\label{Bradial}

In Lemma \ref{lemma:rcradialstretchings} we saw that rank-one convex functionals are quasiconvex along radial maps. The Burkholder functional is special, as it is \textit{quasiaffine} on a large class of radial stretchings:

\begin{lemma}\label{lemma:Bpradialstretchings}
Let $\phi(z)=\rho(r) \frac{z}{r}$ be a radial stretching satisfying the condition
\begin{equation*}
\label{eq:nonexpanding}
|\dot \rho(r)|\leq \frac{\rho(r)}{r}.
\end{equation*}
The functional ${\bf B}_p$ is quasiaffine along such radial stretchings: 
$$-1={\bf B}_p(\Id)= \fint_{\mb D} {\bf B}_p(\D \phi) \,\dd m(z).$$
\end{lemma}

\begin{proof}
We refer the reader to \cite[Theorem 8.1]{AIPS15a} or \cite{Guerra19} for a proof.
\end{proof}

We also have the following refinement of  \cite[Theorem 8.1]{AIPS15a} to the $\overline \R$-valued setting, showing that 
${\bf B}_p$ is an extreme point in a natural class of functionals.

\begin{theorem}
Let ${\bf E}\colon \R^{2\times 2} \to \overline \R$ be an functional such that:
\begin{enumerate}
\item\label{it:domain} for some $K\geq 1$, $Q_2(K)\subset\tp{int}\bigl(\tp{dom}({\bf E})\bigr)$;
\item ${\bf E}$ is rank-one convex and ${\bf E}(\Id)=-1$;
\item ${\bf E}$ is positively $p$-homogeneous, for some $p\geq 2$;
\item ${\bf E}$ is isotropic, that is ${\bf E}(QAR)={\bf E}(A)$ for all $A\in \R^{2\times 2}$ and all $Q,R\in \tp{SO}(2)$.
\end{enumerate} 
Then ${\bf E}(A)\geq {\bf B}_p(A)$ for all $A\in Q_2(K)$.
\label{thm:extremality}
\end{theorem}

\begin{proof}
Let $\phi(z)=\rho(r) \frac{z}{r}$ be a $K$-quasiconformal radial stretching satisfying \eqref{eq:nonexpanding}. Combining Lemmas \ref{lemma:rcradialstretchings} and \ref{lemma:Bpradialstretchings}, we estimate
$$\fint_{\mb D}^* {\bf E}(\D \phi) \,\dd m(z) \geq {\bf E}(\Id) = -1 ={\bf B}_p(\Id)=\fint_{\mb D} {\bf B}_p(\D \phi) \,\dd m(z);$$
here note that condition \eqref{it:domain} ensures that Lemma \ref{lemma:rcradialstretchings} is applicable.
We now take, for some $\alpha\in [-1,1]$,
$$\rho(r)\equiv \begin{cases}
\frac{r}{2^{\alpha-1}} & \text{if } r\leq \frac 1 2,\\
r^\alpha &\text{if } r\geq \frac 1 2.
\end{cases}$$  In particular, $\rho$ satisfies \eqref{eq:nonexpanding} whenever $|\alpha|\leq 1$. With this 
choice the map $\phi$ is $\frac{1}{|\alpha|}$-quasiconformal, 
since we have a.e.\ the identities
\begin{equation*}
|\D \phi(x)|^2= r^{2(\alpha-1)}, 
\qquad \det \D \phi(x)=\alpha r^{2(\alpha-1)},
\end{equation*}
 cf.\ \eqref{eq:derivativephi}. 
Since ${\bf E}(\Id)={\bf B}_p(\Id)=-1$ and both ${\bf E}$ and ${\bf B}_p$ are positively $p$-homogeneous, 
$$\int_{\frac 1 2\mb D} {\bf E}(\D \phi)\,\dd m(z) = \int_{\frac 1 2 \mb D} {\bf B}_p(\D \phi)\, \dd m(z).$$
For $\frac{1}{2} \leq |z| \leq 1$ we have $\rho(r)=r^\alpha$. Thus
\begin{equation}\label{annulus}
\int_{\mb A(\frac 1 2, 1)} {\bf E}(\D \phi) \,\dd m(z) \geq \int_{\mb A(\frac 1 2, 1)} {\bf B}_p(\D \phi)\, \dd m(z).\end{equation}
 Moreover, it follows from  \eqref{eq:derivativephi}  that for ${\bf E}$ as in the statement of the lemma
(hence for ${\bf B}_p$ as well),
\begin{equation}
{\bf E}(\D\varphi(x))=r^{p(\alpha-1)}{\bf E}(\alpha,1),
\end{equation}
where we have identified  $(x,y)\equiv  \tp{diag}(x,y)$.
Thus 
\eqref{annulus} gives
$${\bf E}(\alpha,1)\int_{\frac 1 2}^1  r^{p(\alpha-1)+1}\dd r \geq {\bf B}_p(\alpha,1)\int_{\frac 1 2}^1  r^{p(\alpha-1)+1}\dd r 
 \implies  {\bf E}(\alpha,1)\geq {\bf B}_p(\alpha,1). $$
Finally, varying $\alpha\in [1/K,1]$, we have ${\bf E}\geq {\bf B}_p$ on the rank-one segment $[\Id, \tp{diag}(1/K,1)]\subset Q_2(K)$.  Since  ${\bf E}$ and ${\bf B}_p$ are positively $p$-homogeneous and isotropic, it follows that ${\bf E}\geq {\bf B}_p$ in $Q_2(K)$.
\end{proof}
\smallskip

\section{Proof of Theorem \ref{main}} \label{sec:Maintheorem}

The purpose of this section is to prove Theorem \ref{main}, the closed quasiconvexity of the local Burkholder functional.  In fact, it will be convenient to consider  slightly more general versions of the local Burkholder functional $ \mathcal B_{K}$ introduced in \eqref{eq:locBurk}, so let us define
\begin{equation}\label{locB}
{\bf B}_{K,p}(A) \equiv 
\begin{cases}
{\bf B}_p(A),\quad &\text{if} \;\; |A|^2 \leq K \det(A), \\ 
+\infty, &\text{otherwise}. 
\end{cases}
\end{equation}
Thus these functionals all have the same domain $\,\mathrm{dom} \left({\bf B}_{K,p} \right)$ = $Q_2(K)$,
while $ \mathcal B_{K} = {\bf B}_{K,p_{K}}$  when $\,p\,$ equals the limiting exponent $p_K = 2K/(K-1)$. Also, for this exponent
$$
Q_2(K) = \{ A\in \M : {\bf B}_{p_K}(A) \leq 0 \}.
$$

We first consider the closed $\WW^{1,p}$-quasiconvexity   at the identity,  for these modified functionals  and for 
  $2 < p < p_K$.
\smallskip

\begin{proposition} \label{Id}
Let $K > 1$ and fix $p \in (2,\tfrac{2K}{K-1})$. Then the 
functional ${\bf B}_{K,p}\colon \M \to \R \cup \{+\infty\}$
is closed $\WW^{1,p}$-quasiconvex at $ \Id$.
\end{proposition}
\begin{proof}
Let $\nu \in \Meas_{\qc}^{p}$ be a homogeneous gradient Young measure with center of mass $\langle \nu, \Id \rangle = \Id$; our goal is to establish the Jensen inequality
$$ {\bf B}_{K,p}(\Id)\leq \langle \nu, {\bf B}_{K,p}\rangle = \int_{\R^{2 \times 2}}  {\bf B}_{K,p}(A) d\nu(A).$$
There is nothing to prove when $\langle \nu , {\bf B}_{K,p} \rangle = +\infty$ so without loss of generality we assume that
$\langle \nu , {\bf B}_{K,p} \rangle < +\infty$. 

Under this assumption, from the definition \eqref{locB} 
we see that $\nu$ is supported in the $K$-quasiconformal cone $Q_{2}(K)$.  Therefore
Theorem \ref{thm:YMqc} gives us a sequence $( f_j )$ of $K$-quasiconformal principal maps generating the measure $\nu$. In particular,  
each map  $f_j$ is  conformal outside $\DD$, and  by \eqref{Belt2} we have $(\D f_j )$ bounded in $\LL^s(\DD)$, for each $s < p_K$. 

Furthermore,  for this sequence ${\bf B}_{K,p}\bigl(\D f_j(z)\bigr) = {\bf B}_{p}\bigl(\D f_j(z)\bigr)$, so that  with Theorem \ref{closedQC} and \eqref{Burk123} we have
$$
{\bf B}_{K,p}\bigl( \langle \nu, \Id \rangle  \bigr)  =  {\bf B}_{p}( \Id )   \leq \,  \liminf_{j \to \infty} \fint_{\DD} \! {\bf B}_{K,p}\bigl(\D f_j(z)\bigr) \, \dd m(z).
$$

On the other hand, as  pointed out above, $(\D f_j )$ bounded in $\LL^s(\DD)$ for $p < s < p_K$ and 
 since  ${\bf B}_{p}$ is $p$-homogeneous, it follows 
that the sequence  $\bigl( {\bf B}_{K,p}(\D f_j )\bigr) = \bigl( {\bf B}_{p}(\D f_j )\bigr) $ is equiintegrable in $\mb D$. We can thus apply 
Remark \ref{rmk:moreintegrands}, which implies
$$  
\lim_{j\to \infty} \fint_{\DD} \! {\bf B}_{K,p}\bigl(\D f_j(z)\bigr) \, \dd m(z) = \int_{\M} \! {\bf B}_{K,p}(A) \, \dd\nu (A).
$$
This completes the proof.
\end{proof} 

To complete the proof of Theorem \ref{main}, our task is then to extend the closed  quasiconvexity 
from $A = \Id$ to a general matrix $A \in \R^{2 \times 2}$. We start with:  

\begin{proposition}\label{kqc}
Let $K>1$. Then  for each $\, 2 < p < \tfrac{2K}{K-1}$,  the local Burkholder functional  ${\bf B}_{K,p}  \colon \M \to \R \cup \{ +\infty \}$ is
closed $\WW^{1,p}$-quasiconvex.
\end{proposition}
\smallskip

 \begin{proof}
Define the relaxation $\mathscr{R}\colon \M \to \eR$ of the Burkholder functional ${\bf B}_{K,p}$ by
$$
\mathscr{R}(A) \equiv  \inf \left\{ \int_{\M}^{\ast} \! {\bf B}_{K,p} \; \dd \nu : \, \nu \in \Meas^{p}_{\qc} \, \mbox{ and } \,  \langle \nu, \Id \rangle = A \right\}
\, , \quad A \in \M .
$$
Since $p > \tfrac{2K}{K+1}$,  the $K$-quasiconformal cone $Q_{2}(K)$ is $\WW^{1,p}$-quasiconvex,  cf.\ Example \ref{ex:qccone} and the discussion before Lemma  \ref{gen.seq}.
Thus we have for the effective domain 
 $\mathrm{dom} (\mathscr{R}) = Q_{2}(K)$.  Obviously, $$\mathscr{R} \leq {\bf B}_{K,p} \quad {\rm on} \quad  \M, $$ and according to Proposition \ref{Id},
 we have the equality at the identity matrix:
\begin{equation}\label{finiteid}
\mathscr{R}( \Id ) = {\bf B}_{K,p}( \Id ) = -1.
\end{equation}

We then claim that $\mathscr{R}$ is rank-one convex. To this end, let us fix $A_0$, $A_1 \in \M$ with $\mathrm{rank}(A_1-A_0) = 1$,
$\lambda \in (0,1)$ and let $A_{\lambda} \equiv  (1-\lambda )A_{0}+\lambda A_1$: we will show that 
\begin{equation}\label{rkconvex}
\mathscr R(A_\lambda)\leq \lambda \mathscr R(A_0)+(1-\lambda) \mathscr R(A_1).
\end{equation}
There is nothing to prove if one of the matrices 
$A_i \notin Q_{2}(K)$, so we can assume that $A_0$, $A_1 \in Q_{2}(K)$. 

Under this assumption, take $t_i \in \R$ such
that $t_i > \mathscr{R}(A_i )$. By the definition of $\mathscr R$, we may then find $\nu_i \in \Meas^{p}_{\qc}$ with 
$ \langle \nu_i, \Id \rangle = A_i$ such that  $t_i > \langle \nu_{i},{\bf B}_{K,p} \rangle$.
In particular, it follows that both measure $\nu_0$ and $\nu_1$ are supported on $Q_{2}(K)$.

Next, define $\nu_{\lambda} \equiv  (1-\lambda )\nu_{0}+\lambda \nu_1$.
Hereby $\nu_\lambda$ is a probability measure on $\M$, supported on $Q_{2}(K)$, with $ \langle \nu_\lambda, \Id \rangle=A_\lambda$. It clearly
also has a finite $p$-th moment and, if ${\bf E}$ is a quasiconvex functional
 of at most $p$-th growth in the sense of Theorem \ref{thm:KP}.(1), then
\begin{eqnarray*}
  \bigl\langle \nu_{\lambda},{\bf E} \bigr\rangle &=& (1-\lambda ) \bigl\langle \nu_{0},{\bf E} \bigr\rangle + \lambda \bigl\langle \nu_{1},{\bf E} \bigr\rangle\\
  &\geq& (1-\lambda ){\bf E}(A_{0})+\lambda {\bf E}(A_{1}) \geq {\bf E}(A_{\lambda}),
\end{eqnarray*}
where the first inequality follows from Theorem \ref{thm:KP}.(1), while second  holds since as a real-valued quasiconvex functional, ${\bf E}$ is rank-one convex.  

Thus all conditions required by Theorem \ref{thm:KP} are fulfilled. Consequently, $\nu_\lambda \in \Meas^{p}_{\qc}$
and  we get
$$
(1-\lambda )t_0 + \lambda t_1 > \bigl\langle \nu_{\lambda},{\bf B}_{K,p} \bigr\rangle \geq \mathscr{R}(A_{\lambda}).
$$
Since $t_i>\mathscr R(A_i)$ was arbitrary, the desired inequality \eqref{rkconvex} follows and proves that $\mathscr{R}$ is rank-one convex.  

By inspection we also deduce that $\mathscr{R}$ is positively $p$-homogeneous and isotropic. Indeed, this is a simple consequence of
Lemma \ref{lemma:invYM} and the fact that ${\bf B}_{K,p}$ has the same properties.  Let us just check positive homogeneity, as isotropy is similar.
Fix $t>0$ and, for a measure $\nu$, let $\langle \nu_t,{\bf E}\rangle \equiv  \langle \nu, {\bf E}(t \cdot)\rangle$ for any continuous ${\bf E}$ with
$p$-growth.  Since the map $\nu\mapsto \nu_t$ is a bijection of $\mathscr M^p_\tp{qc}$ onto itself, we have
\begin{align*}
t^p \mathscr R(A)  & = \inf\left\{\int^\ast {\bf B}_{K,p}(t\cdot) \, \dd \nu : \nu \in \mathscr M^p_\tp{qc},  \,  \langle \nu, \Id \rangle  = A\right\}\\ 
& = \inf\left\{\int^\ast {\bf B}_{K,p} \, \dd \mu : \mu \in \mathscr M^p_\tp{qc},  \,  \langle \mu, \Id \rangle  = tA\right\} = \mathscr R( t A),
\end{align*}
where we used positive homogeneity of ${\bf B}_{K,p}$ in the first equality.

We have so far shown that $\mathscr{R} \colon \R^{2\times 2} \to \overline\R$ is positively $p$-homogeneous, isotropic, rank-one convex
and $-\infty<\mathscr{R}\leq {\bf B}_{p}$ in $Q_2(K)$, with equality at the identity matrix. It follows from Theorem \ref{thm:extremality}
that in fact $\mathscr{R} = {\bf B}_{p}$ on $Q_{2}(K)$ and, consequently, that $\mathscr{R} = {\bf B}_{K,p}$
on $\M$. But this is exactly what we set out to prove: 
$$
{\bf B}_{K,p}(A) = \mathscr{R}(A) \leq \int_{\M} \! {\bf B}_{K,p} \, \dd \nu,
$$
where the inequality holds, by definition of $\mathscr R$, for all $A \in \M$ and all $\nu \in \Meas_{\qc}^{p}$ with $  \langle \nu, \Id \rangle =A$.
In brief, we have shown that for any $2 <  p < \tfrac{2K}{K-1}$ the functional ${\bf B}_{K,p}$ is
closed $\WW^{1,p}$-quasiconvex.  
\end{proof}

\begin{remark}  In fact, combined with Corollary \eqref{cor:higherintegrability}, the above argument shows that ${\bf B}_{K,p}$ is closed 
$\WW^{1,s}$-quasiconvex for every $s > \frac{2K}{K+1}$.
\end{remark}

We have done most of the hard work, so that the proof of Theorem \ref{main} now follows easily.

\begin{proof}[Proof of Theorem \ref{main}]
Let $p > \tfrac{2K}{K+1}$ and $\nu \in \Meas^{p}_{\qc}$. If $\int_{\M}^{\ast} \! \mathcal B_K \, \dd \nu = +\infty$, then there is nothing to prove,
so we can assume that $\int_{\M}^{\ast} \! \mathcal B_{K} \, \dd \nu < +\infty$. This assumption implies that $\nu$ is supported in $Q_{2}(K)$ and therefore
by Corollary \ref{cor:higherintegrability} we have  that
$\nu \in \Meas_{\qc}^{s}\bigl( Q_{2}(K) \bigr)$ for every $s < p_{K} = \tfrac{2K}{K-1}$. 

Next, we can use Theorem \ref{kqc} to get
$$
{\bf B}_{s}(  \langle \nu, \Id \rangle ) \leq \int_{\M} \! {\bf B}_{s} \, \dd \nu
$$
for each $2< s < p_{K}$. Also, observe that ${\bf B}_{s} \leq 0$ on $Q_{2}(K)$ and that ${\bf B}_{s} \to {\bf B}_{p_{K}}$ pointwise as $s \nearrow p_K$.
Consequently, Fatou's lemma yields
$$
{\bf B}_{p_K}\bigl(   \langle \nu, \Id \rangle \bigr) = \lim_{s\nearrow p_K} {\bf B}_{s}\bigl(   \langle \nu, \Id \rangle \bigr) \leq \limsup_{s \nearrow p_K} \int_{\M} \! {\bf B}_{s} \, \dd \nu \leq \int_{\M} \! {\bf B}_{p_K} \, \dd \nu. 
$$
This completes the proof, since $\mathcal B_K = {\bf B}_{p_K}$ in the support of $\nu$.
\end{proof}

In particular, as closed $\WW^{1,p}$-quasiconvexity implies $\WW^{1,p}$-quasiconvexity, see for instance the discussion after  Definition \ref{def:closedqc},
Theorem \ref{main} implies  Theorem \ref{thm:soft}.
\smallskip

\section{The Burkholder area inequality}
\label{sec:areaBp}

The purpose of this section is to prove Theorem \ref{thm:Bpareaintro}. We restate result here in a slightly different (but equivalent) form: 
\smallskip

\begin{theorem}\label{thm:Bparea}
Let $f$ be a $K$-quasiconformal principal map, conformal  outside $\DD$ with expansion
\begin{equation}
\label{eq:expansionarea}
f(z) = z + \frac {b_1}{z} + \sum_{j=2}^\infty \frac{b_j}{z^j} \equiv z+ \frac{b_1}{z} + \phi(z), \quad |z| > 1.
\end{equation}  
 Then with the linear asymptotics $A_f(z) \equiv z+b_1 \bar z$, we have
\begin{equation}
\label{eq:Bparea}
\int_\mb D \Big({\bf B}_p(\D f)-{\bf B}_p(A_f)\Big)\, \dd m(z) \ge -  \frac{p}{2} \frac{{\bf B}_p(A_f)}{ \det(A_f)}\int_{\C\setminus \DD}  |\phi'(z)|^2 \, \dd m(z),
\end{equation}
provided that $2\leq p \leq \frac{2K}{K-1}$.
\end{theorem}


We begin with a rather general lemma which improves the asymptotics  of the map in a controlled manner,
 while keeping the map unchanged in the disk.

\begin{lemma} \label{lemma:extension}
Suppose $f  \in \WW^{1,1}_\tp{loc}(\C)$
 is a 
 principal map with expansion \eqref{eq:expansionarea}. 
Then the  map $\tilde f \colon \C \to \C$ defined by
$$\tilde f \equiv \begin{cases} f & \tp{in } \mb D,\\ h \circ A_f & \tp{in } \C \backslash \mb D,\end{cases}$$ 
is a $\WW^{1,1}_\tp{loc}$-homeomorphism, where $h\colon A_f(\C \setminus \DD) \to f(\C \setminus \DD)  $ is a conformal map 
defined by $$h \equiv  f \circ R^{-1}, \qquad R(z)\equiv z+ b_1/z.$$
Moreover,
$$h(z) = z + {\mathcal O}\left(z^{-2}\right) \text{ as } |z| \to \infty.$$
Finally, if $f$ is $K$-quasiconformal for some $K \geq 1$, then so is $\tilde f$.
\end{lemma}

\begin{proof} We saw in Section \ref{sec:principal} that $|b_1| < 1$. Hence the rational map $R$ is injective in $\C \setminus \DD$ and equals the linear map $A_f$ on the unit circle. In particular, $R^{-1} \circ A_f$ is a homeomorphism of the exterior disc  $\C \setminus \DD$ and equals the identity on $\mb S^1$. Hence the map $\tilde f$ is a $\WW^{1,1}_\tp{loc}$-homeomorphism of $\C$.

It is clear that $h\equiv f \circ R^{-1} $ is conformal in $\C \setminus A_f(\DD)$ and by \eqref{Kareafmla} $\tilde f$ is $K$-quasiconformal when $f$ is, so we only need to prove the decay of $h$.
But this is easy, since 
$$ R^{-1}(z) = \frac{1}{2} \left( z + \sqrt{z^2 - 4b_1} \right) = z - \frac{b_1}{z} +  {\mathcal O}\left(\frac{1}{z^3} \right)\qquad \tp{as }  |z| \to \infty,
$$
which shows that
\begin{align*}
h(z)  =  f \circ R^{-1}(z)  & =  R^{-1}(z) +  \frac{b_1}{R^{-1}(z) } +  {\mathcal O}\left(\frac{1}{z^2} \right)\\
&  =  z - \frac{b_1}{z} + \frac{b_1}{ z - \frac{b_1}{z}} +  {\mathcal O}\left(\frac{1}{z^2} \right) = z +  {\mathcal O}\left(\frac{1}{z^2} \right).
\end{align*}
 completing the proof.
\end{proof}

The next lemma establishes a global quasiconvexity inequality	for the modified map $\tilde f$; here the fast decay of $h$, provided by Lemma \ref{lemma:extension}, is crucial.

\begin{lemma}
\label{lemma:globalqcBp}
Let $f$ be as in Theorem \ref{thm:Bparea} and define $\tilde f$ as in Lemma \ref{lemma:extension}. Then
$$\int_\C \Big({\bf B}_p(\D \tilde f) - {\bf B}_p(A_f)\Big)\, \dd m(z) \geq 0.$$
\end{lemma}

\begin{proof}
Choose first  a sequence of exponents ${p_n} \nearrow p \leq \frac{2K}{K-1}$, and write  $K_{p_n} \equiv \frac{{p_n}}{{p_n}-2}$ so that $K < K_{p_n}$. Then define an auxiliary function $\psi(z)$, $|z|>1$,  via
$$\tilde f(z)=h(A_f(z))=A_f(z) + \mathcal O(A_f(z)^{-2})\equiv A_f(z) + \psi(z).$$

Finally choose,  for each large integer $j\geq 2$,  a smooth radially symmetric cutoff $0\leq \eta_j\leq 1$  such that
$\eta_j(z)= 1$ if $|z|\leq j$,  $\eta_j(z)=0$ if $|z|\geq j+1$ and $|\nabla \eta_j|\leq 2$. We claim that for any given index $n$, for all $j$  large enough  the map $A_f+ \eta_j \psi $ is $K_{p_n}-$quasiconformal 
 in $\C \backslash \mb D$. 
Indeed, note first that via \eqref{Kareafmla} the linear map $A_f$ is $K$-quasiconformal. Thus by construction, also $\tilde f(z)$ is   $K$-quasiconformal in all of $\C$. 

We then set
\begin{equation} \label{tildej}
\tilde f_j \equiv \begin{cases} f & \tp{in } \mb D,\\ A_f+\eta_j \psi & \tp{in } \C \backslash \mb D,\end{cases}
\end{equation}
so that the function equals $\tilde f$ whenever $|z| < j$.
In addition, since we chose  $K_{p_n}  >  K$, for any $n$
 we can take small $ \varepsilon_n > 0$, converging to $0$ as $n \to \infty$, 
 so that $X\in Q_2(K_{p_n})$ for every matrix with $|A_f-X|<\varepsilon_n$.


On the other hand, for all $j$ sufficiently large we can estimate
\begin{equation} \label{decay13}
|\D (\eta_j \psi)(z)| \leq  |\eta_j \D \psi|(z)+ |\psi \otimes \nabla \eta_j|(z) \leq C |z|^{-2}
\end{equation}
where the constant $C$ is independent of $j$. Thus, taking $j = j_n$ so  large that $C j_n^{-2}\leq \frac{\varepsilon_n}{ 2}$, 
one obtains $|\D (\eta_{j_n} \psi)(z)| \leq \frac{\varepsilon_n }{2}$ for $|z|\geq j_n$.  
 This means  that by \eqref{Burkh46} and by what we have just shown,  the map in \eqref{tildej} satisfies ${\bf B}_{p_n}(\D \tilde f_{j_n}) \leq 0$. 

We are now in a position to apply Theorem \ref{thm:soft}:
Since $\tilde f_{j_n}(z)=A_f$ for $|z|\geq {j_n}+1$, 
we conclude that 
\begin{equation} \label{preliBp}
\int_{\C} \Big({\bf B}_{p_n}(\D \tilde f_{j_n}) - {\bf B}_{p_n}(A_f)\Big)\, \dd m(z)\ge 0, \qquad n \geq 1.
\end{equation}

To complete the argument we decompose the integral in \eqref{preliBp} as follows. Since
 $\tilde f_{j_n} = \tilde f$ in $\{ |z|\leq {j_n}\}$, 
 we choose some fixed $\rho > 1$ and from \eqref{preliBp} obtain for all  ${j_n} > \rho$ the lower bound
 \begin{align}  \label{preliBp7} 
& \int_{|z|\leq \rho}  \left({\bf B}_{p_n}(\D \tilde f) - {\bf B}_{p_n}(A_f)\right) \dd m +  \int_{\rho \leq |z|\leq j_n} \left({\bf B}_{p_n}(\D \tilde f) - 
{\bf B}_{p_n}(A_f)\right) \dd m \\ 
&\;  \qquad  + \int_{j_n\leq |z|\leq j_n+1} \left({\bf B}_{p_n}(\D \tilde f_{j_n}) - {\bf B}_{p_n}(A_f)\right) \dd m\ge 0. \nonumber 
\end{align}


For the first term above,  Lemma \ref{lemma:extension}  tells that $\tilde f$ is $K$-quasiconformal and hence ${\bf B}_{p_n}(\D\tilde  f) \leq {\bf B}_p(\D \tilde f) \leq 0$. We can thus use  Fatou's lemma to see that  
$$ \int_{|z|\leq \rho} \Big({\bf B}_p(\D  \tilde f) - {\bf B}_p(A_f)\Big)\dd m\geq
\limsup_{p_n \nearrow p} \int_{|z|\leq \rho} \Big({\bf B}_{p_n}(\D \tilde f) - {\bf B}_{p_n}(A_f)\Big)\dd m.$$

For the second term in \eqref{preliBp7} recall that  in the annuli $1 \leq |z| \leq j_n$ we have $\tilde f(z) = h \circ A_f(z)$, where for the conformal factor
$h'$ is bounded on the set $A_f(\{ |z| \geq \rho \})$. Therefore in these annuli the functions
$${\bf B}_{p_n}(\D \tilde f) - {\bf B}_{p_n}(A_f) = {\bf B}_{p_n}(A_f)\left( |h'\circ A_f|^{p_n}-1 \right)= {\bf B}_{p_n}(A_f) \mathcal O(A_f(z)^{-3})$$
are  uniformly bounded and  decay like $|z|^{-3}$ at infinity. Thus the dominated convergence theorem gives 
$$ \lim_{n \to \infty}  \int_{\rho \leq |z|\leq j_n} \left({\bf B}_{p_n}(\D \tilde f) - {\bf B}_{p_n}(A_f)\right) \dd m =
\int_{|z| \geq \rho} \Big({\bf B}_p(\D  \tilde f) - {\bf B}_p(A_f)\Big)\dd m.
$$

Finally, we claim  that when $n\to \infty$, the third term in \eqref{preliBp7} vanishes. Indeed, from the explicit expression \eqref{Burk}, or from the general properties  
of positively $p$-homogeneous and rank-one convex functionals  \cite{BKK}, 
we have the Lipschitz estimate
$$\left| {\bf B}_{p_n}(A_f+ \D (\eta_{j_n} \psi) ) - {\bf B}_{p_n}(A)\right|\leq C(K) | \D (\eta_{j_n} \psi)| \left(|A_f|+| \D (\eta_{j_n} \psi)|\right)^{p-1}.$$
 Thus we can use the estimate \eqref{decay13} which gives
\begin{align*}
& \left|\int_{{j_n}\leq |z|\leq {j_n}+1} \left({\bf B}_{p_n}(\D \tilde f_{j_n}) - {\bf B}_{p_n}(A_f)\right) \dd m(z) \right| \\
& \qquad 
 \leq C(K)\int_{{j_n}\leq |z|\le {j_n}+1} | \D (\eta_{j_n} \psi)| \left(|A_f|+| \D (\eta_{j_n} \psi)|\right)^{p-1}\dd m(z) 
\\
& \qquad  \leq C(K, A_f) {j_n}^{-1} \to 0,
\end{align*}
since $|\D(\eta_{j_n}\psi)|\leq C {j_n}^{-2}$ in $\{{j_n}\leq |z|\leq {j_n}+1\}$ and the area of this annulus 
 is $ (2{j_n}+1)\pi$.
 Combining now the above estimates and letting $n\to\infty$, we conclude  that
\begin{align*}
\int_{\mb C} \left({\bf B}_p(\D \tilde f) - {\bf B}_p(A_f) \right) \dd m(z) \geq 0,
\end{align*} 
as claimed.
\end{proof}



\begin{proof}[Proof of Theorem \ref{thm:Bparea}]
Using the definition of $\tilde f$ from Lemma \ref{lemma:extension} and applying Lemma \ref{lemma:globalqcBp}, we obtain:
$$\int_{\DD} \Big( {\bf B}_p(\D f ) - {\bf B}_p(A_f) \Big)\,  \dd m(z)  
\geq  -  \int_{\C\setminus \DD} \Big( {\bf B}_p(\D (h \circ A_f)) - {\bf B}_p(A_f) \Big)\, \dd m(z).$$
Our goal is to give a lower bound on the right-hand side of this estimate.  We saw already in the previous lemma that
${\bf B}_p(\D (h \circ A_f)) - {\bf B}_p(A_f) = {\bf B}_p(A_f)  \Bigl(\bigl| h' \circ A_f \bigr|^p  \, - 1 \Bigr) $. Hence changing the variables twice, we have:
 \begin{align*}
& \int_{\C\setminus \DD} \Bigl( {\bf B}_p(\D (h \circ A_f)) - {\bf B}_p(A_f)\Bigr) \, \dd m(z) \\
& \qquad =  \frac{{\bf B}_p(A_f)}{\det(A_f)} \int_{A_f(\C\setminus \DD)} \Bigl(  \bigl| h'(z) \bigr|^p  \, - 1 \Bigr)  \, \dd m(z)\\
& \qquad =   \frac{{\bf B}_p(A_f)}{\det(A_f)} \int_{R(\C\setminus \DD)} \Bigl( \bigl| (f \circ R^{-1})'(z) \bigr|^p  \, - 1 \Bigr)  \, \dd m(z) \\
& \qquad =   \frac{{\bf B}_p(A_f)}{\det(A_f)} \int_{R(\C\setminus \DD)} \Bigl(  \bigl| f'( R^{-1}(z))\bigr|^{p} \, \bigl| (R^{-1})' (z)\bigr|^p  \, - 1 \Bigr)  \, \dd m(z) \\
& \qquad =   \frac{{\bf B}_p(A_f)}{\det(A_f)} \int_{\C\setminus \DD} \Bigl(  \bigl| f'(z)\bigr|^{p} \, \bigl| R' (z)\bigr|^{-p}  \, - 1 \Bigr)  \, 
|R'(z)|^2 \, \dd m(z). 
\end{align*}
Finally,  in the notation of \eqref{eq:expansionarea},  $f = R + \phi$ in $\C\setminus \DD$. In particular,  
\begin{align*}
& \int_{\DD}  \Big({\bf B}_p(\D f(z) ) - {\bf B}_p(A_f)\Big) \, \dd m(z)  \\
& \qquad \geq  - \frac{{\bf B}_p(A_f)}{\det(A_f)} \int_{\C\setminus \DD} \Big(| R'(z) + \phi'(z)|^p |R'(z)|^{2-p} - |R'(z)|^2\Big) \, \dd m(z).
\end{align*}

This suggests us to define
$$ H_p(z,w) \equiv  | R'(z) + w|^p |R'(z)|^{2-p} - |R'(z)|^2
$$
to estimate the integrand on the right-hand side.
We claim that	
\begin{equation} \label{eq:areaclaim}
 \int_{\C\setminus \DD}  H_p\bigl(z,\phi'(z) \bigr)  \, \dd m(z)\geq \frac{p}{2} \int_{\C\setminus \DD}  |\phi'(z)|^2 \,\dd m(z),
\end{equation}
To this end, we need a further lemma.
\begin{lemma} \label{lemma:convexity}
Let $ G_p(w) = |1+ w|^p - 1$. Then for any $2 < p < \infty$,
$$ G_p(w)  \geq  p \Re w + \frac{p}{2} |w|^2, \qquad w \in \C,
$$
where the coefficient in front of $|w|^2$ is the largest possible.
\end{lemma}
\begin{proof} Letting $\zeta = w + 1$,  the task is to show that 
$$ |\zeta|^p \geq 1 + p \Re(\zeta-1) + \frac{p}{2} |\zeta -1|^2 = 1-\frac{p}{2} + \frac{p}{2}|\zeta|^2,
$$
which is easy to verify.  The derivatives of the left-hand side and the right-hand side  both evaluate to $p$ at $\zeta=1$, hence the constant $\frac p 2$ is optimal.
\end{proof}
Note that
$$H_p(z,w) = |R'(z)|^2 G_p\bigl(w/R'(z)\bigr),$$
and thus Lemma \ref{lemma:convexity} yields
\begin{align}
\begin{split}
\label{eq:auxarea1}
H_p\bigl(z,\phi'(z)\bigr) & \geq p  |R'(z)|^2\Re\bigl(\phi'(z) / R'(z) \bigr) + \frac{p}{2}  |\phi'(z) |^2 
\\ & = p \Re\left(\phi'(z) \, \overline{R'(z)} \right) + \frac{p}{2}  |\phi'(z) |^2.  
\end{split}
\end{align}
To use this inequality note that 
as $\phi'(z) =- \sum_{j=2}^\infty \frac{j \,b_j}{z^{j+1}}$ and ${R'(z) = 1 -b_1/z^2}$,  by integration in polar coordinates it follows that 
$$ \int_{\C \setminus \mb D}  \phi'(z)\overline{R'(z)} \,\dd m(z) = 0.
$$ 
Thus we obtain the required claim \eqref{eq:areaclaim} upon integrating the estimate \eqref{eq:auxarea1}.  We have thus shown that
 $$\int_{\mb D}\Big( {\bf B}_p(\D f ) - {\bf B}_p(A_f)\Big) \, \dd m(z)  \geq - \frac{p}{2} \frac{{\bf B}_p(A_f)}{\det(A_f)}  \int_{\C\setminus \DD}  |\phi'(z)|^2 \, \dd m(z),$$
which completes the proof.
\end{proof}

\begin{proof}[Proof of Theorem \ref{thm:Bpareaintro}]
In the notation of \eqref{eq:expansionarea},  $\phi'(z) =- \sum_{j=2}^\infty \frac{j \,b_j}{z^{j+1}}$ where the powers  $z^{-(j+1)}$ are orthogonal in $L^2(\C\setminus \DD)$. Thus
$$\int_{\C \setminus \DD} |\phi'(z)|^2 \, \dd m(z) = \sum_{j=2} j^2 |b_j|^2 \Big\| z^{-j-1} \Big\|^2_{L^2(\C \setminus \DD)} = \pi \sum_{j=2} j |b_j|^2.
$$
Since ${\bf B}_2(A_f) = - \det(A_f)$, this combined with \eqref{eq:Bparea} proves Theorem  \ref{thm:Bpareaintro}.
\end{proof}

\section{The sharp higher integrability of the Jacobian}
\label{sec:LlogL}

As usual, in  this section $\Omega\subset \R^2$ denotes an arbitrary bounded domain such that $\mathscr L^2(\p \Omega)=0$. By letting $\,p\to 2\,$ in Theorems  \ref{thm:soft} and \ref{main}, we will obtain new quasiconvex functionals defined on
$$\R^{2\times 2}_+\equiv  \R^{2\times 2}\cap\{\det>0\}.$$

To be precise, any functional ${\bf E}\colon \R^{2\times2}_+\to \R$ can be naturally extended to $\R^{2\times 2}$ by setting
\begin{equation}\label{eq:extenddetfunc}
{\bf E}(A)=+\infty, \quad \text{ whenever } \det A\leq 0.
\end{equation}
We will always identify ${\bf E}$ with this extension.
 It is important to note that the set $\R^{2\times 2}\backslash \R^{2\times 2}_+=\{\det \leq 0\}$ is $\WW^{1,2}$-quasiconvex, recall Definition \ref{qcset}, and hence the extension of ${\bf E}$ is trivially $\WW^{1,2}$-quasiconvex at matrices with non-positive determinant. It is for this reason that there is no ambiguity in talking about quasiconvexity of functionals defined on $\R^{2\times 2}_+$.

Let us first consider the functional
 \begin{equation*}\label{F2}
 { \mathscr  F}(A) \equiv   \, |A|^2\, -  \,\left(1\, + \,\log |A|^2 \,\right )\, \det(A),  
 \end{equation*}
which is rank-one convex in $\R^{2\times 2}$, cf.\ \cite[Corollary 5.1]{Tade} or \eqref{eq:derivativeBp} below, but not polyconvex: indeed, for $t>0$ we have $\mathscr F(t\, \Id)=-2t^2 \log t$, which decreases faster than any quadratic function \cite[Corollary 5.9]{Dacorogna2007}. 

Here we succeed in proving that $\mathscr F$ is quasiconvex in $\R^{2\times 2}_+$:

\begin{corollary}\label{cor:LlogL} Let  $A \in \R^{2 \times 2}_+$ and $f \in A + \WW^{1,2}_0(\Omega)$ be a homeomorphism. Then
 \begin{equation}\label{LlogL}
 { \mathscr  F}(A) \leq   \fint _\Omega \, { \mathscr  F}(\D f(z)) \, \dd m(z).  
\end{equation}
\end{corollary}

\begin{proof} 
Suppose first that $f$ is a smooth diffeomorphism on $\overline{\Omega}$, with boundary values $f|_{ \partial \Omega} = A$.  Since ${\bf B}_2(A) = - \det(A)$ is a null Lagrangian and
\begin{equation}
\label{eq:derivativeBp}
  \frac{p}{p-2} \left[ {\bf B}_p(A) - {\bf B}_2(A)\right]  \, = \,  |A|^2 - \det(A)(1+ \log(|A|^2) + {\mathcal O}(p-2),
\end{equation}
taking the limit $p \to 2$ we see from Theorem \ref{thm:soft} that \eqref{LlogL} holds for such a map $f$. 

In  the case of a general ${\WW}^{1,2}$-homeomorphism $f$ with boundary values $A$,  we  apply \cite[Theorem 1.1]{IKO} to obtain a sequence of  diffeomorphisms $f_j$ such that $f_j
\to f$ in $\WW^{1,2}(\Omega)$ and $f_j=A$ on $\p\Omega$. We now argue similarly as in the proof of \cite[Corollary 1.7]{AIPS12}, see also \cite{IwaniecVerde}.  Applying the quasiconvexity inequality for each $f_j$, we obtain
\begin{align*}
\mathscr F(A) & \leq \limsup_{j\to \infty} \fint_\Omega \mathscr F(\D f_j(z)) \,\dd m(z)\\
& = \lim_{j\to \infty} \fint_\Omega \left(|\D f_j|^2 - J_{f_j}\right)\dd m(z) -  \liminf_{j\to \infty} \fint_\Omega J_{f_j} \log|\D f_j|^2 \,\dd m(z).
\end{align*}
There is no difficulty in passing to the limit in the first term. For the second term, we estimate
\begin{align*}
& \liminf_{j\to \infty} \fint_\Omega J_{f_j} \log|\D f_j|^2\,\dd m\\
&\quad =\liminf_{j\to \infty} \fint_\Omega J_{f_j} \log(1+|\D f_j|^2) \,\dd m - \lim_{j\to \infty}
\fint_\Omega J_{f_j} \log(1+|\D f_j|^{-2})\, \dd m\\
&\quad \geq \fint_\Omega J_{f} \log(1+|\D f|^2) \, \dd m-\fint_\Omega J_f \log(1+|\D f_j|^{-2})\,\dd m\\
& \quad = \fint_\Omega J_{f} \log|\D f|^2 \, \dd m,
\end{align*}
where in the third line we applied Fatou's lemma for the first term, since the integrand is non-negative, and the Dominated Convergence Theorem for the second term, since the integrand is dominated pointwise by $J_{f_j}|\D f_j|^{-2}\leq 1$. The desired inequality follows.
\end{proof}

\begin{remark}
Exactly the same argument shows that, when $\Omega=\mb D$, inequality \eqref{LlogL} holds more generally if $f$ is a \textit{monotone map} in the topological sense of Morrey, i.e.\ if $f^{-1}(w)$ is connected for all $w$. In fact, by a Sobolev version of a classical theorem of Youngs, a map in $A + \WW^{1,2}_0(\mb D)$ is monotone if and only if it is the strong $\WW^{1,2}$-limit of a sequence of smooth diffeomorphisms with the same boundary condition \cite{IO2016}.
\end{remark}

For maps with  identity boundary values the estimate \eqref{LlogL} was obtained in \cite[Corollary 1.7]{AIPS12}.
The fact that the Jacobian determinant possesses higher integrability was first noticed by M\"uller \cite{Muller1990} and then generalized by Coifman, Lions, Meyer and Semmes \cite{CLMS}. From these results one deduces that if $f\in W^{1,n}_\tp{loc}(\R^n)$ is an orientation-preserving mapping then $J_f \log |\D f|^n\in L^1_\tp{loc}(\R^n)$, see \cite{IwaniecGreco}.  Corollary \ref{cor:LlogL} expresses a sharp, global version of this result when $n=2$: indeed, we may rewrite \eqref{LlogL} as the sharp inequality
\begin{equation}
\label{eq:sharpLlogL}
\int_\Omega J_f(z)(1+\log|\D f(z)|^2) \,\dd m(z) \leq \int_\Omega |\D f(z)|^2 \,\dd m(z) -\mathscr F(A)\mathscr L^2(\Omega),
\end{equation}
to hold whenever $f\in A + \WW^{1,2}_0(\Omega)$ is a homeomorphism and $A \in \R^{2 \times 2}_+$.

As mentioned above, the functional $\mathscr F\colon \R^{2\times 2}\to \R$ is rank-one convex and it is conjectured to be quasiconvex as well \cite[(6.5)]{Tade}, in which case \eqref{eq:sharpLlogL} would hold for general maps $f\in A +\WW^{1,2}_0(\Omega)$ which are not necessarily orientation-preserving. For such maps one cannot give meaning to $\int_\Omega J_f \log|\D f|^2\,\dd m(z)$ as a Lebesgue integral; instead, this integral needs to be interpreted distributionally  \cite{IwaniecVerde}.

\section{New sharp bounds via the Shield transformation}\label{sec:shield}

The optimal integral bounds of the previous Section lead to new ones via the {\it Shield transformation}. This is  a general procedure for generating new functionals, by applying a given functional to the inverse map:

\begin{definition}
Given ${\bf E}\colon \R^{2\times 2}_+\to \R$ we define $\widehat {\bf E}\colon \R^{2\times2}_+\to \R$ by
$$\widehat {\bf E}(A)\equiv {\bf E}(A^{-1}) \det A.$$
\end{definition}

In the Elasticity literature, the involution $\,\widehat \cdot \,$ is sometimes referred to as the \textit{Shield transformation}, after Schield's work \cite{Shield}. 
Since $\widehat{1}=\det$, $\widehat{\cdot}$ does not preserve convexity; yet it preserves the usual semi-convexity notions from the vectorial Calculus of Variations.  Indeed, it is easy to verify that polyconvexity and rank-one convexity are preserved by 
$\,\widehat \cdot \,$, see \cite[Theorem 2.6]{Ball1977} for further details.  The case of quasiconvexity, which is the one concerning us here, is more subtle.  Indeed, as in Definition \ref{quasiconvexity}, in the case  of  deformations  a crucial point is that one needs to address also the regularity of the inverse map when discussing quasiconvexity of $\widehat {\bf E}$, here see also Example \ref{ex:radialmapW}.  In fact, the inverse of a planar  $\WW^{1,1}_\tp{loc}$-homeomorphism $f$ is in $\WW^{1,2}_\tp{loc}$ if and only if $K_f\in \LL^1_\tp{loc}$ \cite[Theorem 1.7]{HK}, and hence this is the assumption that we shall make.

\begin{proposition}
\label{prop:shield}
 Let ${\bf E}\colon \R^{2\times 2}_+\to \R$ be a $\WW^{1,2}$-quasiconvex functional,  in the sense that
${\bf E}$ satisfies 
\eqref{eq:quasiconvexity} 
for any $A\in \R^{2\times 2}_+$ and any homeomorphism $f\in A + \WW_0^{1,2}(\Omega)$. 
Then the Shield transformation $\widehat{{\bf E}}\colon \R^{2\times 2}_+\to \R$ satisfies 
$$\widehat{{\bf E}}(A)\leq \fint_\Omega \widehat{{\bf E}}(\D f(z)) \,\dd m(z)$$
for all $A\in \R^{2\times 2}_+$ and all homeomorphisms $f\in A + \WW_0^{1,1}(\Omega)$ with $K_f\in \LL^1(\Omega)$.
\end{proposition}

\begin{proof}
Let $f\in A + \WW^{1,1}_0(\Omega)$ be a homeomorphism with integrable distortion.
Clearly we may assume that $J_f>0$ a.e.\ in $\Omega$, as otherwise there is nothing to prove. 
By \cite[Theorem 2.1]{HKO}, the inverse homeomorphism $g \equiv  f^{-1}$ satisfies\footnote{Strictly speaking, \cite[Theorem 2.1]{HKO} deals with the Euclidean norm instead of the operator norm, but the same argument establishes the desired identity.}
 \begin{equation}\label{eq:aimo}
\int_{A(\Omega )} |\D g(w)|^2 \, \dd m(w)=\int_\Omega K_f(z) \, \dd m(z)
\end{equation}
and thus $g\in  A^{-1} + W^{1,2}_{0}\bigl(A(\Omega)\bigr)$. Since $\WW^{1,2}$-homeomorphisms satisfy Lusin's condition (N), we can apply the change of variables formula to obtain
\begin{align*}
\int_\Omega \widehat {\bf E}(\D f(z))\,\dd m(z) & =\int_{A(\Omega)}{\bf E}(\D g(w))\,\dd m(w)\\& \geq \mathscr L^2(A(\Omega)){\bf E}(A^{-1}) = \mathscr L^2(\Omega) \widehat {\bf E}(A), 
\end{align*}
where in the last equality we used  $\mathscr L^2(A(\Omega))=\det(A) \,\mathscr L^2(\Omega)$.
\end{proof}

\begin{remark}
The same proof as in Proposition \ref{prop:shield} applies more generally to $\WW^{1,n}$-quasiconvex functionals ${\bf E}\colon \R^{n\times n}_+\to \R$. The natural assumptions are that $f\in A + \WW^{1,1}_0(\Omega)$ is a homeomorphism such that $K_f\in \LL^{n-1}(\Omega)$, where $\Omega\subset \R^n$ is a domain.
\end{remark}

\begin{remark}\label{rmk:automatichomeo}
We note that any map $f\in A + \WW^{1,2}_0(\Omega)$ such that $K_f\in \LL^1(\Omega)$ is necessarily a homeomorphism: such maps are automatically continuous and, by the results of Iwaniec and \v Sver\'ak \cite{IwaSve}, they are also open, hence the claim follows from degree theory.
\end{remark}
 \medskip



We then apply the Shield transform to generate a new quasiconvex functional, simply by setting  
\begin{equation}
\mathscr W\equiv \widehat {\mathscr F}+1.
\end{equation}

Using $|A^{-1}| = |A|/\det(A)$, one easily computes
\begin{equation} \label{W}
\mathscr W(A) =  \begin{cases}
\frac{|A|^2}{{\rm det}\, A} - \log \left(  \frac{|A|^2}{{\rm det}\, A} \right) + \log \det(A) \quad {\rm if} \;  \det(A) > 0,\\
+ \infty \hspace{4,9cm} {\rm otherwise}.
\end{cases}
\end{equation}
This functional had already been considered in \cite{AIPS12} and was also studied in the recent works \cite{Voss1,Voss2}.  In \cite{Voss1} 
it was suggested that $\mathscr W$, denoted by $W^+_{\rm magic}$ in that paper,  could be an example  of a functional  which is  rank-one convex but not quasiconvex. Later in \cite{Voss2} numerical evidence was found which instead supports the quasiconvexity of $\mathscr W$.  

As a direct consequence of Corollary \ref{cor:LlogL} and Proposition \ref{prop:shield}, we now obtain that indeed  $\mathscr W$ \textit{is quasiconvex}:

\begin{corollary}
\label{cor:W}  Let  $A \in \R^{2 \times 2}_+$ and
 let $f \in A+\WW^{1,1}_0 (\Omega)$ be a 
 homeomorphism. If $K_f \in \LL^1(\Omega)$ then 
   \begin{equation}\label{Wqconv}
   \mathscr W(A) \leq \fint_{\Omega} \mathscr W\bigl(\D f(z)\bigr)  \, \dd m(z).
  \end{equation}
\end{corollary}
\smallskip

Koskela and Onninen showed in \cite{KO} that for a map $f\in \WW^{1,2}_\tp{loc}(\Omega)$ with $K_f\in \LL^1_\tp{loc}(\Omega)$, we have $\log J_f \in \LL^1_\tp{loc}(\Omega)$, see also Proposition \ref{prop:KOR} below and \cite{KOR} for an analogue in space. We can interpret \eqref{Wqconv} as a sharp, global version of their result: indeed, we may rewrite this estimate as
$$
\int_\Omega -\log J_f(z) \,\dd m(z) \leq \int_\Omega \left(K_f(z) -\log K_f(z)\right) \dd m(z) -\mathscr W(A)\mathscr L^2(\Omega)
$$
for homeomorphisms $f\in A+ \WW_0^{1,1}(\Omega)$ with $K_f\in \LL^1(\Omega)$.

Moreover,  $\mathscr W$ has also some  closed quasiconvexity features, for details see Proposition \ref{thm:closedqc.new} below.  

\subsection{Quasiconvexity beyond integrable distortion}

So far we have been discussing quasiconvexity of  $\mathscr W$ for maps with integrable distortion.  This assumption is natural since, as we saw in \eqref{eq:aimo}, it is \textit{equivalent} to the $\WW^{1,2}$-regularity of the inverse map. Nonetheless, one can ask whether $\mathscr W$ satisfies the quasiconvexity inequality when tested with smooth homeomorphisms with less regular inverses.  Indeed there are smooth homeomorphisms with finite $\mathscr W$-energy for which the inverse is not in $\WW^{1,2}$; this is possible since the distortion and the Jacobian terms in $\mathscr W$ may cancel each other:

\begin{example}\label{ex:radialmapW}
Consider the radial stretching 
$$f(z)=\exp(1-1/r^2) \frac{z}{r}, \qquad r=|z|.$$ It is easy to see that $f$ is a smooth homeomorphism which equals the identity on $\mb S^1$. We compute
\begin{align*}
K_f(z) = \frac{2}{r^2};
\end{align*}
in particular, $K_f \in \tp{weak-}\LL^1(\mb D)$ but $K_f \not\in \LL^1(\mb D)$. 
Nonetheless $f$ has finite $\mathscr W$-energy: indeed,
$$J_f(z)= \frac{2}{r^4} \exp(2-2/r^2) \quad \implies \quad
\log J_f (z) = 2-\frac{2}{r^2} + \log \frac{2}{r^4}$$
and hence the divergent terms $\frac{2}{r^2}$ cancel each other:
$$\mathscr W(\D f(z)) = 2(1-\log r) \quad \implies \quad \int_{\mb D} \mathscr W(\D f(z)) \, \dd m(z) = 3 \pi.$$
\end{example}


In view of the above example, if no further restrictions are made for the homeomorphisms $f$ in Corollary \ref{cor:W} beyond the assumptions that $f \in A + \WW^{1,1}_0(\Omega)$ and that $\mathscr W(\D f) \in \LL^1(\Omega)$, the inequality \eqref{Wqconv} reduces to a question of approximation, i.e.\  whether for any such $f$ there exists a sequence of diffeomorphisms $f_j \in C^\infty(\Omega)$ equal to $A$ on the boundary, with
\[\lim_{j\to\infty} \int_{\Omega} |\mathscr W(\D f(z)) - \mathscr W(\D f_j(z))| \dd m(z) = 0.\]
In the context of Sobolev homeomorphisms, the main tools for constructing a diffeomorphic approximation employ either the harmonic (or $p$-harmonic) extension \cite{IKO} or a shortest curve extension method \cite{HP}. Unfortunately, neither of these tools work in the present setting due to the existence of boundary maps for which the standard extension methods produce the wrong integrability of the distortion, see \cite[Example 1.3]{KosO}. Hence it is clear that to prove the appropriate result for the functional $\mathscr W$, completely new methods of approximation need to be developed first.

On the other hand, for radial maps as in Example \ref{ex:radialmapW} there is no difficulty in proving the quasiconvexity inequality \eqref{Wqconv} directly, using e.g.\ the argument in Lemma \ref{lemma:rcradialstretchings}. 
In addition, as we shall next see, 
 Corollary \ref{cor:W} extends to all monotone maps. Recall that a map is {\it monotone} in $\Omega$ if 
 $$ \langle f(z) - f(w), z - w \rangle \geq 0, \qquad \forall \; z, w \in \Omega.
 $$
 If furthermore, say,  $f\in \WW^{1,1}_\tp{loc}(\Omega)$, then monotonocity is equivalent to 
\begin{equation*}
\label{eq:monotone}
|\p_{\bar z} f |\leq \Re \p_z f\quad  \tp{ a.e.\ in } \Omega,
\end{equation*}
see for instance \cite[\S 3.11]{AIM}.
Note that any orientation-preserving radial stretching is monotone, since for such maps $\p_z f\in \R$.

In the monotone case, the main point of  proof   is to find an approximation of a given monotone map by maps with smaller $\mathscr W$-energy, at least in the region where the map has large distortion.  In order to construct such an approximating sequence we rely on a clever trick due to Chleb\'ik and Kirchheim \cite{Chlebik}.

\begin{proposition}\label{prop:monotone}
Let $A\in \R^{2\times 2}_+$ and let $f\in A+\WW^{1,2}_0(\Omega)$ be monotone with $ \mathscr W(\D f)\in \LL^1(\mb D)$. Then
$$\mathscr W(A) \leq \fint_{\mb D} \mathscr W(\D f(z))\,\dd m(z).$$
\end{proposition}

\begin{proof} By the definition of $\mathscr  W$ as an extended real-valued functional in \eqref{eq:extenddetfunc}, we may assume that $J_f>0$ a.e.\ in $\mb D$. For $\delta\in (0,1)$ let us consider the map
$$f_\delta (z)\equiv  f(z)+ \delta z,$$
where we  compute 
$$|\p_z f_\delta| =  \sqrt{|\p_z f|^2+2 \delta \Re \p_z f + \delta^2}, \qquad |\p_{\bar z} f_\delta|= |\p_{\bar z} f|.$$
Notice that since $\Re(\p_z f) \geq 0$, the expression $|\p_z f_\delta|$ is an increasing function of $\delta$. Furthermore, we obtain
\begin{equation}
\label{eq:derivativefdelta}
|\p_z f|^2 + \delta^2 \leq |\p_z f_\delta|^2 \leq 2(|\p_z f|^2 + \delta^2).
\end{equation}
Let us set $E_\delta \equiv \{z\in \mb D: |\p_z f_\delta| \leq 2|\p_{\bar z} f|\}$. We claim that if $0 < \delta < 1$, then
\begin{equation}
\label{eq:comparison}
\begin{cases}
\mathscr W(\D f_\delta)\leq \mathscr W(\D f) & \tp{ in } E_\delta,\\
\mathscr W(\D f_\delta)\leq \mathscr W(\D f_1) & \tp{ in } \mb D \backslash E_\delta.
\end{cases}
\end{equation}
To see this, note that
$$\mathscr W(\D f)=\frac{|\p_z f|+|\p_{\bar z} f|}{|\p_z f|-|\p_{\bar z} f|}+ 2 \log\left(|\p_z f|-|\p_{\bar z} f|\right)=\omega(|\p_z f|),$$
where
$$\omega(t)\equiv \frac{t+|\p_{\bar z} f|}{t-|\p_{\bar z} f|}+ 2 \log(t-|\p_{\bar z} f|), \qquad  t > |\p_{\bar z} f|.$$
Since $\omega'(t)=2(t-2|\p_{\bar z}f|)/(t-|\p_{\bar z} f|)^2$, we see that $\omega$ is decreasing on $(|\p_{\bar z} f|,2|\p_{\bar z} f|)$ and increasing on $(2|\p_{\bar z} f|,+\infty)$, thus
$$\begin{cases}
\omega(|\p_z f_\delta|) \leq \omega(|\p_z f|)  &\tp{ in } E_\delta,\\
\omega(|\p_z f_\delta|) \leq \omega(|\p_z f_1|)  &\tp{ in } \mb D \backslash E_\delta, \qquad
\end{cases}
$$
which yields \eqref{eq:comparison}.

Since $J_f>0$ a.e.\ in $\mb D$,  using \eqref{eq:derivativefdelta} we can estimate
$$K_{f_\delta}= \frac{|\D f_\delta|^2}{J_{f_\delta}} 
\leq  2 \frac{|\p_z f_\delta|^2 + |\p_{\bar z} f|^2}{|\p_z f_\delta|^2 - |\p_{\bar z} f|^2}
\leq 4 \frac{|\p_z f|^2 + |\p_{\bar z} f|^2 + \delta^2}{J_f + \delta^2}
\leq 4\frac{|\D f|^2+\delta^2}{\delta^2},$$ thus $K_{f_\delta}\in \LL^1(\mb D)$. 
With $A_\delta\equiv A + \delta \,\Id$ we have that $f_\delta\in  A_\delta + \WW^{1,2}_{0}(\mb D)$ is a homeomorphism, cf.\ Remark \ref{rmk:automatichomeo}, and hence, by Corollary \ref{cor:W} and \eqref{eq:comparison},
\begin{align}
\label{eq:auxmonotone}
\begin{split}
\pi \mathscr W(A_\delta)&  \leq \int_\mb D \mathscr W(\D f_\delta(z))\, \dd m(z)\\
& = \int_{E_\delta} \mathscr W(\D f_\delta(z))\, \dd m(z)+\int_{\mb D \setminus E_\delta} \mathscr W(\D f_\delta(z))\, \dd m(z)\\
& \leq \int_{E_\delta} \mathscr W(\D f(z))\, \dd m(z) + \int_{\mb D \setminus E_\delta} \mathscr W(\D f_\delta(z))\, \dd m(z) .
\end{split}
\end{align}
We can now apply the reverse Fatou lemma to the last term in \eqref{eq:auxmonotone} since we have the estimate
$\mathscr W(\D f_\delta(z)) \leq \mathscr W(\D f_1(z))$ for $z \in \mb D \setminus E_\delta$, where $\mathscr W(\D f_1) \in \LL^1(\mb D)$ since $K_{f_1} \in \LL^1(\mb D)$. 
Therefore
\begin{align*}
\limsup_{\delta \to 0} \int_{\mb D \setminus E_\delta} \mathscr W(\D f_\delta(z))\, \dd m(z) &\leq \int_{\mb D} \limsup_{\delta \to 0} \chi_{\mb D \setminus E_\delta}(z) \mathscr W(\D f_\delta(z)) \, dz
\\ &= \int_{\mb D \setminus E_0} \mathscr W(\D f(z)) \, dz
.
\end{align*}
The conclusion follows by sending $\delta \to 0$ in \eqref{eq:auxmonotone}, since $\mathscr W(\D f)\in \LL^1(\mb D)$.
\end{proof}

\section{Improved integral inequalities} \label{sec:inequalities} 
The purpose of this section is to prove further integral inequalities for the functionals $\mathscr F$ and $\mathscr W$, similar to the Burkholder area inequality of Section \ref{sec:areaBp}.
As a first observation, the precise value of the constant on the right-hand side of \eqref{eq:Bparea} allows us to improve the bound in Corollary \ref{cor:LlogL} for principal homeomorphisms of Definition \ref{prmap}.  

\begin{corollary}
Let $f$ be a 
$\WW^{1,2}_\tp{loc}$-principal mapping, 
with the linear asymptotics $A_f(z) = z + b_1 \overline{z}$. 
Then
\begin{equation}\label{Ffunc}
 \int_{\DD} \left( {\mathscr  F}(\D f(z))  -  {\mathscr  F}(A_f) \right) \dd m(z) \geq  \left(1- \frac{{\mathscr  F}(A_f)}{\det(A_f)}\right)\int_{\C \setminus \DD} |\phi'(z)|^2 \dd m(z). 
 \end{equation}
\end{corollary} 
\begin{proof}
Let us begin by assuming that $f$ is a smooth diffeomorphism. The classical area formula, cf.\ \eqref{eq:area} and \eqref{areafmla}, 
can be written in the form
\begin{equation}\label{eq:areaB2}
 \int_{\DD} \left({\bf B}_2(\D f) - {\bf B}_2(A_f) \right) \dd m(z) =  \int_{\C \setminus \DD} |\phi'(z)|^2 \,\dd m(z),
\end{equation}
since ${\bf B}_2=-\det$. By \eqref{eq:Bparea} and \eqref{eq:areaB2}, for $p>2$ sufficiently close to 2  we have that $f$ is a principal $\frac{p}{p-2}$-quasiconformal map and so
\begin{align*}
\begin{split} \label{eq:auxareaF}
0 & \leq \int_{\DD} \left({\bf B}_p(\D f) - {\bf B}_p(A_f) \right) \dd m(z) + \frac{p}{2}   \frac{{\bf B}_p(A_f)}{\det(A_f)} \int_{\C \setminus \DD} |\phi'|^2 \, \dd m(z)\\
& = \int_{\DD} \left({\bf B}_p(\D f) - {\bf B}_2(\D f)\right) \,\dd m(z) - \int_{\DD} ({\bf B}_p(A_f) - {\bf B}_2(A_f))\, \dd m(z)\\
& \qquad + \left(1+ \frac{p}{2}   \frac{{\bf B}_p(A_f)}{\det(A_f)}\right) \int_{\C \setminus \DD} |\phi'(z)|^2 \,\dd m(z).
\end{split}
\end{align*}   
Recalling that 
$$
  \frac{p}{p-2} \left[ {\bf B}_p(A) - {\bf B}_2(A)\right]  \, = \, { \mathscr  F}(A)  + {\mathcal O}(p-2),
$$
multiplying the above estimate by $\frac{p}{p-2}$ and taking the limit $p \searrow 2$, since
$$\frac{p}{p-2}\left(1+ \frac{p}{2}   \frac{{\bf B}_p(A_f)}{\det(A_f)}\right)=
\frac{p}{2 \det(A_f)} \left(\frac{p}{p-2} [{\bf B}_p(A_f)-{\bf B}_2 (A_f)]\right)- \frac p 2,
$$
we obtain the desired inequality. The case of general orientation-preserving $\WW^{1,2}$ maps follows as in the proof of Corollary \ref{cor:LlogL}.
\end{proof}

The reader can easily check that \eqref{Ffunc} remains true under the scalings $f \mapsto t f $. 
Also, writing  $A(z)=a_+z +a_-\bar z$ we compute
 $$ 
 1-\frac{ { \mathscr  F}(A) }{\det(A)}  = 1 - \frac{2 |a_-|}{|a_+|-|a_-|} + 2 \log(|a_+| + |a_-|) \equiv c_p(|a_+|, |a_-|).
 $$
The sign of $c_p$ is not constant in the set $\{|a_+|>|a_-|\}$: for instance, the reader can verify that $c_p(1,|a_-|) > 0$ for $|a_-|$ small while $c_p(1,|a_-|)<0$ when $|a_-|$ is close to 1. 

We next prove Theorem \ref{thm:Wareaintro}, which is a version of the area inequality for $\mathscr W$. This result will be very useful in Section \ref{sec:swlsc} since, in combination with Theorem \ref{thm:intdistort}, it shows that one can test quasiconvexity of $\mathscr W$ with an appropriate class of gradient Young measures. 

\begin{theorem}
\label{thm:Warea}
Let $f\in \WW^{1,1}_\tp{loc}(\C)$ be a homeomorphism with $K_f\in \LL^1(\mb D)$ and suppose that $f$ is conformal outside $\mb D$ with expansion
\begin{equation} \label{princip3} f(z) = z + \frac{b_1}{z} + 
\sum_{j=2}^\infty \frac {b_j}{z^j},  \qquad |z| > 1.
\end{equation}
Letting $A_f\in \R^{2\times 2}$ be given by $A_f(z) = z + b_1 \bar z$, we have
$$\mathscr W(A_f)\leq \fint_{\mb D} \mathscr W(\D f(z)) \, \dd m(z).$$
\end{theorem}

Here recall from  \eqref{asympto} that for any principal homeomorphism 
$\det(A_f) > 0$. In fact, under the normalisation \eqref{princip3} we have $1 \leq \mathscr W(A_f) < +\infty$.

The proof of the above Theorem follows the same broad strategy as for  the Burkholder area inequality in Theorem \ref{thm:Bparea}, but some of the details are quite different.  As before, the first step is to establish a quasiconvexity inequality over the full space, as in  Lemma \ref{lemma:globalqcBp}. 

\begin{lemma}\label{lemma:globalqcW}
Given $f$  as in Theorem \ref{thm:Warea}, define the auxiliary function $\tilde f$ as in Lemma \ref{lemma:extension}. Then we have
$$0 \leq \int_{\C} \left(\mathscr W(\D \tilde f(z)) - \mathscr W(A_f) \right) \dd m(z).$$
\end{lemma}

\begin{proof}
Like in the proof of Lemma \ref{lemma:extension}, define for $|z|>1$ a new function $\psi(z)$ by the identity
\begin{equation} \label{tilde2}
\tilde f(z)=h(A_f(z))=A_f(z) + \mathcal O(A_f(z)^{-2})\equiv A_f(z) + \psi(z).
\end{equation}
Let also $0\leq \eta_j\leq 1$ be a smooth, radially symmetric cutoff such that
$\eta_j(z)= 1$ if $|z|\leq j$,  $\eta_j(z)=0$ if $|z|\geq j+1$ and $|\nabla \eta_j|\leq 2$. 

As in the proof of Lemma \ref{lemma:globalqcBp} we have the basic estimate, that
for any given  $\varepsilon>0$ and  for all $j$ sufficiently large, depending on $\varepsilon$, 
\begin{equation}
\label{eq:decaycutoff}
|\D (\eta_j \psi)(z)| \leq  |\eta_j \D \psi|(z)+ |\psi \otimes \nabla \eta_j|(z) \leq C |z|^{-2} \leq \varepsilon
\end{equation}
whenever $|z|\geq j$. 
 With these we set
$$\tilde f_j =\begin{cases} f & \tp{in } \mb D,\\ A_f+\eta_j \psi & \tp{in } \C \backslash \mb D.\end{cases}$$

Our first task is to show that for $j$ large enough the distortion functions $K(\tilde f_j) \in \LL^1_\tp{loc}(\C)$. Indeed, in the unit disc $K(\tilde f_j) = K(f) \in \LL^1(\DD)$ by assumption, while by \eqref{tilde2} we have 
$K(\tilde f_j) = K(A_f)$ when $1 \leq |z| \leq j$ or $|z| \geq j+1$.  
Hence we only need to cover the annulus $\{ z : j < |z| < j+1 \}$, where by \eqref{eq:decaycutoff}
$$ K(\tilde f_j) = \frac{\bigl| A_f + \D (\eta_j \psi) \bigr|^2}{\, \det\bigl(A_f + \D (\eta_j \psi)\bigr)\,} \leq 2K(A_f ),
$$
when $j = j_\varepsilon$ is large enough.

Thus $K(\tilde f_j(z))  \in \LL^1_\tp{loc}(\C)$, and since $\tilde f_j(z)=A_f$ for $|z|\geq j+1$, we may now  apply Corollary \ref{cor:W}  to conclude that 
\begin{align*}
\int_{\C} \left(\mathscr W(\D \tilde f_j(z)) - \mathscr W(A_f)\right) \dd m(z)\ge 0 
\end{align*}
or, rearranging, 
$$\int_{|z|\leq j} \left(\mathscr W(\D \tilde f) - \mathscr W(A_f)\right) \dd m + \int_{j\leq |z|\leq j+1}\left( \mathscr W(\D \tilde f_j) - \mathscr W(A_f)\right) \dd m\ge 0.$$

Our claim is that, when $j\to \infty$, the second term above vanishes. 
For this we need quantitative estimates. Recalling that 
\begin{equation}
\label{Wdecom}
\mathscr W(\D f) = K_f - \log K_f + \log J_f ,
\end{equation}
we will estimate each term separately.

First, since $$|\det A_f- \det(A_f+\D(\eta_j \psi))| \leq C |\D(\eta_j \psi)| (|A_f|+|\D(\eta_j \psi)|), $$by taking $\varepsilon$ small enough we have 
\begin{equation}
\label{eq:lowerboundjac}
\frac 1 2\det A_f\leq \det(A_f+\D (\eta_j \psi)) \quad \tp{for all } |z|\geq j
\end{equation}

Second, given matrices $A, B \in \R^{2\times 2}_+$, we have the Lipschitz estimate
\begin{align*}
|K_A- K_B| & =
\left|\frac{\det B (|A|^2-|B|^2) + |B|^2 (\det B - \det A)}{\det A \det B}\right|\\
& \leq \frac{3|B|^2(|A|+|B|)}{\det A \det B} \left||A|-|B|\right|.
\end{align*}
Thus, applying this estimate with $A=A_f$ and $B=\D \tilde f_j=A_f+\D(\eta_j \psi)$, by \eqref{eq:decaycutoff} and \eqref{eq:lowerboundjac} we find
\begin{equation}
\label{eq:estimateK}
|K_{A_f} - K_{f_j}| \leq C(A_f) j^{-2}.
\end{equation}
For the second term in \eqref{Wdecom}   we use the estimate $|\log x - \log y|\leq c^{-1} |x-y|$ which holds provided that $0<c\leq x,y$. Thus, using \eqref{eq:estimateK} twice, we get
\begin{equation}
\label{eq:estimatelogK}
|\log K_{A_f} - \log K_{f_j} | \leq C(K_{A_f})  |K_{A_f} - K_{f_j} | \leq  C(A_f) j^{-2}.
\end{equation}
Finally, for the last term in  \eqref{Wdecom}  we have
\begin{align}
\begin{split}
\label{eq:estimatelogJ}
|\log J_{A_f} - \log J_{f_j} | & \leq C(A_f) |J_{A_f} -  J_{f_j}| \\
& \leq C(A_f) |A_f-\D f_j| \leq C(A_f) j^{-2},
\end{split}
\end{align}
where the constant $C(A_f)$ changes in each inequality.
Combining \eqref{eq:estimateK}--\eqref{eq:estimatelogJ}, since  $\mathscr L^2(\{j\leq |z|\leq j+1\}) \leq C j$, we finally obtain
\begin{align*}
 \left|\int_{j\leq |z|\leq j+1} \left(\mathscr W(\D \tilde f_j(z)) - \mathscr W(A_f)\right)\, \dd m(z) \right|   \leq \frac{C(p,A_f)}{j} \to 0,
\end{align*}
as claimed. Therefore, letting $j\to\infty$, we have
\begin{align*}
\int_{\mb C} \left( \mathscr W(\D \tilde f(z)) - \mathscr W(A_f) \right) \dd m(z) \geq 0.
\end{align*}
As in the proof of Lemma \ref{lemma:globalqcBp}, note that this is indeed a Lebesgue integral: as $|z|\to \infty$,  with $h'(z) = 1+ \mathcal O(z^{-3})$ we see from \eqref{tilde2} that 
$$\mathscr W(\D \tilde f) - \mathscr W(A_f) = \log( |h'\circ A_f|^2 )= \mathcal O(z^{-3}),$$
which is integrable.
\end{proof}

Continuing with the proof of Theorem \ref{thm:Warea}, instead of Lemma \ref{lemma:convexity} we now rely on the following well-known fact, whose proof we recall for the convenience of the reader:

%

\begin{lemma} \label{lemma:nullquaddomain} Let ${\mathscr E} \subset \C$ be an ellipse and suppose $u \in L^{1} (\C \setminus {\mathscr E})$ is holomorphic in $\C \setminus {\mathscr E}$.
Then
$$ \int_{\C \setminus {\mathscr E}} u(z) \,\dd m(z) = 0.
$$
\end{lemma}
\begin{proof} We may assume that $\Omega \equiv \C \setminus {\mathscr E} = R(\C \setminus \DD)$, where $R(z)= z + c/z$ and $|c| < 1$.  Since $u$ is Lebesgue-integrable and holomorphic in $\Omega$, we have 
$u(z) =  {\mathcal O}\left(\frac{1}{z^3} \right)$ as $|z| \to \infty$.
Now $\psi(w)\equiv R(1/w)$ is a conformal map from $\DD$ to  $\Omega$, and thus by Stokes' theorem and the Residue Theorem we have
\begin{align*}
 \int_{\Omega} u(z) \, \dd m(z) & = \frac{1}{2i} \int_{\partial \Omega} u(z) \, \overline{z} \, \dd z \\
& =  \frac{1}{2i} \int_{\mb S^1} (u \circ \psi)(w) \, \left(  w + \frac{\; \overline{ c }\, }{w} \,  \right)  \, \left( c - \frac{1}{w^2} \right) \dd w = 0,
\end{align*}
since at origin $u \circ \psi$ has a zero of order at least $3$.
\end{proof}

Lemma \ref{lemma:nullquaddomain} asserts that the complement of an ellipse is a \textit{null quadrature domain} \cite{GS}; such domains have been completely classified in the plane by Sakai \cite{Sakai} and very recently in \cite{Eberle} in higher dimensions.   

We can finally proceed to the proof of the main result.

\begin{proof}[Proof of Theorems \ref{thm:Warea} and \ref{thm:Wareaintro}]
By Lemma \ref{lemma:globalqcW},  recalling the definition of $\tilde f$ from Lemma \ref{lemma:extension}, we have
$$\int_{\mb D} \Big(\mathscr W(\D f(z)) - \mathscr W(A_f) \Big) \dd m(z) \geq -\int_{\C \backslash \DD}\left( \mathscr W(\D \tilde f(z)) - \mathscr W(A_f) \right) \dd m(z).$$
We claim that the right-hand side of this inequality vanishes, and so the conclusion will follow. Indeed,  in $\C \setminus \DD$  we have $\tilde f = h \circ A_f$, where $h$ is holomorphic, and thus 
$$ \mathscr W(\D\tilde f) - \mathscr W(A_f) = (\log J_h) \circ A_f \quad {\rm in} \; \C \setminus \DD.$$
A change of variables then gives 
$$  \int_{\C \setminus \DD}  (\log J_h) \circ A_f(z)  \, \dd m(z) = \frac{1}{\det(A_f)} \int_{\C  \setminus A_f (\DD)} \log J_h(w)\, \dd m(w)
$$
Here  $A_f(\DD)$ is an ellipse and  $h$ is a conformal map in $\C  \setminus A_f (\DD)$, while Lemma \ref{lemma:extension} gives  the decay  $h'(z) = 1 +  {\mathcal O}\left(\frac{1}{z^3} \right)$ as $|z| \to \infty$. It follows that  $\log h'(z)$ is analytic in  $\C  \setminus A_f (\DD)$ with the decay 
$$ \log h'(z) = {\mathcal O}\left(\frac{1}{z^3} \right) \quad  \tp{ as }  |z| \to \infty.$$
By Lemma \ref{lemma:nullquaddomain}, the integrals of $ \log h'(z) $ and of $\log J_h = 2 \Re \log h'(z) $ over $\C \setminus A_f(\DD)$ vanish, which completes
the proof for  Theorem \ref{thm:Wareaintro}, as well as for Theorem \ref{thm:Warea}, which is a reformulation of it.
\end{proof}

\section{The additive volumetric-isochoric split}\label{sec:advolum}

In this section we consider general functionals satisfying the so-called \textit{additive volumetric-isochoric split}, that is, we consider functionals defined on $\R^{2\times 2}_+\equiv  \{A\in \R^{2\times 2}: \det A>0\}$ which have the form
\begin{equation}
\label{eq:split}
{\bf E}(A)={\bf G}(\det A)+{\bf H}(K_A), \qquad K_A\equiv  \frac{|A|^2}{\det A},
\end{equation}
where  ${\bf G}\colon (0,+\infty)\to \R$ and ${\bf H}\colon [1,+\infty)\to \R$ are given functions. The term ${\bf G}$ corresponds to the \textit{volumetric} part of ${\bf E}$, while ${\bf H}$ represents the \textit{isochoric} part of ${\bf E}$; note also that ${\bf H}$ is invariant under the left- and right-actions of the conformal group $Q_2(1)$. Of course, any functional as in \eqref{eq:split} extends naturally to an functional ${\bf E}\colon \R^{2\times 2}\to \R\cup \{+\infty\}$ by setting ${\bf E}(A)=+\infty$ if $\det A\leq 0$,  cf. ~the discussion in Section \ref{sec:LlogL}. We also have that $\mathscr W$, as defined in \eqref{W},  can be written  in the form \eqref{eq:split} by taking
$${\bf G}(t)=\log(t),\quad {\bf H}(t)=t-\log t.$$

In addition to rank-one convexity of ${\bf E}$, we will assume that the isochoric part ${\bf H}$ is \textit{convex}:

\begin{theorem}\label{thm:decomposition}
Let ${\bf E}\colon\R^{2\times 2}_+\to \R$ be a rank-one convex functional of the form 
\eqref{eq:split},
where  ${\bf G}\colon (0,\infty)\to \R$ and ${\bf H}\colon [1,+\infty)\to \R$ is convex. Then there is a polyconvex functional ${\bf F}\colon \R^{2\times 2}_+\to \R$ and a constant $c\geq 0$ such that
$${\bf E}={\bf F}+c \mathscr W.$$ 
\end{theorem}
We recall that a functional ${\bf F}\colon \R^{2\times 2}\to \eR$ is said to be \textit{polyconvex} if there is a convex function $\tilde{\bf F}\colon \R^5\to \eR$ such that ${\bf F}(A)=\tilde {\bf F}(A, \det(A))$, see also \cite{Dacorogna2007} for further details. Since the determinant is a null Lagrangian, Jensen's inequality easily implies that polyconvex functionals are quasiconvex.

Theorem \ref{thm:decomposition} was proved implicitly in \cite{Voss1} and in this section we  give a short, direct proof. Combining Corollary \ref{cor:W} and Theorem \ref{thm:decomposition}, we obtain:

\begin{corollary}
Any functional ${\bf E}\colon \R^{2\times 2}\to \R$ as in Theorem \ref{thm:decomposition} is quasiconvex: if $A\in \R^{2\times 2}_+$ and if $f\in A+\WW^{1,1}_0(\Omega,\R^2)$ is a homeomorphism such that $K_f\in \LL^1(\Omega)$, then
$${\bf E}(A)\leq \fint_\Omega^* {\bf E}(\D f(z)) \,\dd m(z).$$
\end{corollary}

The proof of Theorem \ref{thm:decomposition} relies on the classical Baker--Ericksen inequality. Given $A\in \R^{2\times 2}$, we write $\lambda(A)\equiv (\lambda_1(A),\lambda_2(A))$ for the vector of \textit{singular values} of $A$, which is are the eigenvalues of the positive-definite matrix $\sqrt{A^\tp{T} A}$. The Baker--Ericksen inequality read as follows:

\begin{lemma}\label{lemma:BE}
Let ${\bf E}\colon \R^{2\times 2}_+\to \R$ be an isotropic rank-one convex functional: thus there is a symmetric function $\Phi\colon (0,\infty)^2 \to \R$ such that
$${\bf E}(A)=\Phi(\lambda_1(A), \lambda_2(A)).$$ If $\Phi$ is $C^1$ and $\lambda_1\neq \lambda_2$ then
$$\frac{\lambda_1 \p_1 \Phi(\lambda) - \lambda_2 \p_2 \Phi(\lambda)}{\lambda_1-\lambda_2}\geq 0$$
for all $\lambda=(\lambda_1,\lambda_2)\in \R^2$ such that $\lambda_1,\lambda_2>0$.
\end{lemma}

Lemma \ref{lemma:BE} is well-known and the interested reader can find a short proof for instance in \cite[Proposition 3.2]{GK}.
We will also require the following result:

\begin{lemma}\label{trivlemma}
Let ${\bf E}\colon \R^{2\times 2}_+\to \R$ be a rank-one convex functional with the representation \eqref{eq:split}.
\begin{enumerate}
\item\label{it:h=0} If ${\bf H}=0$ then ${\bf G}\colon (0,\infty)\to \R$ is convex and ${\bf E}$ is polyconvex.
\item\label{it:g=0} If ${\bf G}=0$ then ${\bf H}\colon [1,\infty)\to \R$ is convex and non-decreasing, and ${\bf E}$ is polyconvex.
\end{enumerate}
\end{lemma}

The first claim in Lemma \ref{trivlemma} is classical, see e.g.\ \cite[Theorem 5.46]{Dacorogna2007}. The result in Lemma \ref{trivlemma}\eqref{it:g=0} is not difficult to obtain, see e.g.\ \cite{Neff2017}. Here we present a short proof for the sake of completeness. The crucial point is the easily-checked fact that $A\mapsto K_A$ is a polyconvex functional. We refer the reader to \cite{IM} for a systematic study of polyconvexity properties of distortion functions in higher dimensions.

\begin{proof}[Proof of Lemma \ref{trivlemma}(\ref{it:g=0})]
Since ${\bf E}$ is rank-one convex, for $\lambda_1\geq 1$ we have that 
$\lambda_1\mapsto {\bf E}(\tp{diag}(\lambda_1,1))={\bf H}(\lambda_1)$ is convex. To prove the monotonicity, fix $1\leq s <t$ and let $\theta\in (0,1)$ be such that $\theta t + (1-\theta) t^{-1} =s$. Thus, by rank-one convexity,
\begin{align*}
{\bf H}(s) & ={\bf E}(\tp{diag}(s,1))\\ &\leq \theta {\bf E}(\tp{diag}(t,1))+(1-\theta){\bf E}(\tp{diag}(t^{-1},1))\\ 
& =\theta {\bf H}(t) + (1-\theta){\bf H}(t)= {\bf H}(t)
\end{align*}
and hence ${\bf H}$ has the claimed properties.
Since ${\bf H}$ is non-decreasing and convex, and $K_A$ is polyconvex, it follows that ${\bf E}(A)={\bf H}(K_A)$ is polyconvex as well.
\end{proof}

\begin{proof}[Proof of Theorem \ref{thm:decomposition}]
Since polyconvexity and rank-one convexity are preserved under pointwise limits, there is no loss of generality in assuming that both ${\bf G}\colon (0,\infty)\to \R$ and ${\bf H}\colon (1,\infty)\to \R$ are smooth. Note, however, that we do not assume that ${\bf H}$ is smooth up to $t=1$.

We consider arbitrary $x>y>0$. Since ${\bf E}$ is rank-one convex, a simple calculation yields
$$0\leq x^2 \p_{xx} {\bf E}(\tp{diag}(x,y))=(x y )^2 {\bf G}''(x y ) + \Big(\frac x y \Big)^2 {\bf H}''(x/y).$$
By changing variables $t=x y , s = x/y$, we deduce the inequality
$$\inf_{t>0} t^2 {\bf G}''(t)+\inf_{s>1} s^2 {\bf H}''(s) \equiv G_0+H_0\geq 0.$$	
Similarly, with $\Phi(x,y)={\bf E}(\tp{diag}(x,y))$ as in Lemma \ref{lemma:BE}, we calculate
$$\frac{ x \p_x \Phi(x,y)-y\p_y \Phi(x,y)}{x-y} = \frac{2 x}{y} \frac{{\bf H}'(x/y)}{x-y}$$
and thus the Baker--Ericksen inequality implies the condition
$${\bf H}'(t)\geq 0 \quad \tp{ for } t>1.$$

By assumption $H_0\geq 0$.  Suppose that $G_0\geq 0$ as well; in this case, both ${\bf H}$ and ${\bf G}$ are convex and Lemma \ref{trivlemma} shows that ${\bf E}$, being the sum of two polyconvex functionals, is itself polyconvex, so we may take $c=0$. Hence we now assume that $G_0\leq 0$ and we take $c\equiv -G_0$.

 We claim that
$F\equiv {\bf E}-c \mathscr W$ is polyconvex. In fact, $F$ can be written as 
\begin{align*}
{\bf F}(A)& =[{\bf G}(\det A)- c \log(\det A)]+ [{\bf H}(K_A)- c (K_A-\log K_A)]\\ 
& \equiv \widetilde {\bf G}(\det A)+\widetilde {\bf H}(K_A)
\end{align*}
and we claim that both terms are polyconvex functionals. This will follow from Lemma \ref{trivlemma}. That $\widetilde {\bf G}$ is convex follows from the definition of $c$:
$$\widetilde {\bf G}''(t)={\bf G}''(t)+c/t^2 \geq 0.$$
Again from the definition of $c$, we have
$$\widetilde {\bf H}''(t)={\bf H}''(t)-c/t^2\geq (H_0-c)/t^2\geq 0,$$
so $\widetilde {\bf H}$ is convex. Suppose now that $\widetilde {\bf H}$ is not non-decreasing, so in particular there is $t_0>1$ such that $\widetilde {\bf H}'(t_0)<0$. For $t>1$, since ${\bf H}'(t)\geq 0$, 
$$\widetilde {\bf H}'(t)={\bf H}'(t)+c(1-1/t)\geq c(1-1/t).$$
The right-hand side vanishes in the limit $t\to 1$; so, by choosing $t$ sufficiently close to $1$, we may suppose that $t<t_0$ and that $\widetilde {\bf H}'(t)\geq \widetilde {\bf H}'(t)/2>\widetilde {\bf H}'(t_0)$.
This contradicts the fact that $\widetilde {\bf H}'$ is non-decreasing in $(1,\infty)$, since $\widetilde {\bf H}$ is convex in the same interval. 
\end{proof}

\section{Sequential weak lower semicontinuity and minimizers}
\label{sec:swlsc}

In this last section we apply the previous results and methods  to establish existence of minimisers for the Burkholder functionals as well as for a quite large class of functionals directly related to Nonlinear Elasticity. All functionals  considered here are non-polyconvex.
 
\subsection{Existence of minimizers for the Burkholder energy}
As usual, throughout this section $\Omega\subset\C$ denotes a bounded domain.
It is well known that for functionals with standard growth properties, quasiconvexity is equivalent to sequential weak lower semicontinuity \cite{AcerbiFusco,Morrey}. However, the equivalence need not hold for $\overline{\R}$-valued functionals, cf.\ the discussion in Section \ref{sec:prelims}.  Therefore, in finding minimizers for the Burkholder functional, we start by showing that under natural assumptions the functional is sequentially weakly lower semicontinuous.

\begin{proposition}\label{prop:lscBp}
Let $K\geq 1$ and fix $2\leq p \leq \frac{2K}{K-1}$.  Given a sequence $(f_j)\subset \WW^{1,p}(\Omega)$ of $K$-quasiregular maps such that $f_j\weak f$ in $\WW^{1,p}(\Omega)$ and $({\bf B}_p(\D f_j))$ is equiintegrable,  we have
$$\liminf_{j\to \infty} \int_\Omega {\bf B}_p(\D f_j(z))\, \dd m(z) \geq \int_\Omega {\bf B}_p(\D f(z))\,\dd m(z).$$
\end{proposition}

\begin{proof}
Theorem \ref{thm:fundYM} shows that up to a subsequence, which we do not relabel,  $(\D f_j)$ generates a $\WW^{1,p}$-gradient Young measure $(\nu_z)_{z\in \Omega}$. Since by hypothesis $({\bf B}_p(\D f_j))$ is equiintegrable, we have
\begin{align*}
\lim_{j\to \infty} \int_\Omega {\bf B}_p(\D f_j) \, \dd m(z) 
& = \int_\Omega \int_{\R^{2\times 2}} {\bf B}_p(A) \,\dd \nu_z(A) \, \dd m(z).
\end{align*}
Moreover, here 
$\tp{supp}\,\nu_z \subset Q_2(K)$ a.e., as seen  from  
\begin{equation}
\label{eq:FK}
0 = \int_\Omega {\bf F}_K(\D f_j) \, \dd m(z) \to \int_\Omega \int_{\R^{2\times 2}} {\bf F}_K(A) \, \dd \nu_z(A) \, \dd m(z),
\end{equation} 
with ${\bf F}_K\equiv \min\{0,K \det(\cdot) - |\cdot|^2\}$.
In addition, ${\bf B}_{K,p}={\bf B}_p$ on the $K$-quasiconformal cone $Q_2(K)$, and hence
$$\lim_{j\to \infty} \int_\Omega {\bf B}_p(\D f_j) \, \dd m(z) 
 = \int_\Omega \int_{\R^{2\times 2}} {\bf B}_{K,p}(A) \,\dd \nu_z(A) \, \dd m(z).
$$ 
Continuing the calculation, by Proposition \ref{prop:homogenization} we have  $\nu_z\in \mathscr M^p_\tp{qc}$ for a.e.\ $z$ and thus, by Theorem \ref{main},  we find
 $$\int_{\R^{2\times 2}} {\bf B}_{K,p}(A) \, \dd \nu_z(A) \geq {\bf B}_{K,p}(\langle \nu_z, \Id \rangle) \quad \tp{ for a.e.\ } z\in \Omega.$$
 Finally via \eqref{limit} this yields
 $$\lim_{j\to \infty} \int_\Omega {\bf B}_p(\D f_j) \, \dd m(z) 
 \geq  \int_\Omega {\bf B}_p(\langle \nu_z, \Id \rangle) \, \dd m(z) 
 =  \int_\Omega {\bf B}_p(\D f) \, \dd m(z),$$
 as wished.
\end{proof}

Proposition \ref{prop:lscBp} gives us the  tools to prove  existence of minimizers for the ${\bf B}_p$-energy in suitable Dirichlet classes, as stated in Corollary \ref{cor:Bpminims}:

\begin{proof}[Proof of Corollary \ref{cor:Bpminims}]
The main point is that each $K$-quasiregular map $f \in g + \WW^{1,p}_0(\Omega,\C)$ admits the uniform bound
\begin{equation}\label{eq:sobouniformf}||f||_{\WW^{1,q}(\Omega)} \leq C(\Omega,q,K,g)\end{equation}
for all $p \leq q < \frac{2K}{K-1}$.
Indeed, we extend $f$ as a $K$-quasiregular mapping of the whole plane by setting $f \equiv g$ in $\C \setminus \Omega$, and then we use the higher integrability of quasiregular mappings \cite[Corollary 13.2.5]{AIM}.

The corollary now follows by the Direct Method of the Calculus of Variations. Indeed, take a sequence $(f_j)\subset g + \WW^{1,p}_0(\Omega)$ of $K$-quasiregular maps such that
$$\int_\Omega {\bf B}_p(\D f_j) \, \dd m \to 
\inf\left\{\int_\Omega {\bf B}_p(\D h) \, \dd m: h\in g + \WW^{1,p}_0(\Omega) \tp{ is } K\tp{-quasiregular}\right\}.$$
By \eqref{eq:sobouniformf} 
we may assume that $f_j \weak f$ in $\WW^{1,p}(\Omega)$ for some map $f$ which is then automatically $K$-quasiregular. Moreover, \eqref{eq:sobouniformf} also shows that $|\D f_j|^p$ is equiintegrable, hence Proposition \ref{prop:lscBp} applies and gives
$$\liminf_{j\to \infty} \int_\Omega {\bf B}_p(\D f_j(z)) \, \dd m(z) \geq \int_\Omega {\bf B}_p(\D f(z)) \, \dd m(z).$$
It follows that $f$ is a minimizer, as claimed.
\end{proof}

\subsection{Sequential weak lower-semicontinuity of $\mathscr W$}

Our next goal is to prove Theorem \ref{thm:lscW}, 
which is an analogue of Proposition \ref{prop:lscBp} for the functional $\mathscr W$ introduced in \eqref{W}.


Before proceeding with the main part of the proof, we need the following result, which is essentially proved in \cite{KOR}:

\begin{proposition}
\label{prop:KOR}
Let $g\in \WW^{1,1}_\tp{loc}(\C)$ be a homeomorphism with $K_g\in \LL^q_\tp{loc}$ for $q\geq 1$.  If $f\in g + \WW^{1,1}_0(\Omega)$ is a homeomorphism such that $K_f\in \LL^q(\Omega),$ then 
$$\int_\Omega \log^q\left(e+\frac{1}{J_f(z)}\right) \, \dd m(z) \leq C(q,g,\Omega) \left(1+\int_\Omega K_f(z)^q \, \dd m(z)\right).$$
\end{proposition}

\begin{proof}
In \cite{KOR} the authors only prove a local estimate; however  setting $f(x) =g(x)$ on $\C\setminus \Omega$ defines a global mapping and reduces the required bound to the estimate in  \cite{KOR}. 
\end{proof}

As a quick consequence of Proposition \ref{prop:KOR} we have 

\begin{lemma}\label{lemma:equiintegrability} 
Let $(f_j)\subset \WW^{1,1}(\Omega)$ be a sequence of homeomorphisms that $\sup_j \|K_{f_j}\|_{\LL^q(\Omega)}<\infty$ for some $q>1$. Suppose that either
\begin{enumerate}
\item\label{it:trace} $f_j=g$ on $\p \Omega$ for a homeomorphism $g\in \WW^{1,1}_\tp{loc}(\C)$ with $K_g \in \LL^q_\tp{loc}$,  or
\item\label{it:principal} $\Omega=\mb D$ and $f_j$ are principal maps.
\end{enumerate}
Then 
$(\mathscr W(\D f_j))$ is equiintegrable.
\end{lemma}

\begin{proof} 
Since $K_{f_j}\geq 1$ we have $K_{f_j} - \log K_{f_j} \leq 2 K_{f_j}$ which is equiintegrable by assumption. To deal with the Jacobian term in $\mathscr W$, we use the pointwise estimate
\begin{equation}
\label{eq:elementaryestimatelog}
|\log(J_{f_j})| \leq  \log\left(e+\frac{1}{J_{f_j}}\right) + 2(J_{f_j})^{\frac 1 2}.
\end{equation}
The first term is clearly equiintegrable by Proposition \ref{prop:KOR}. 
Given a measurable set $U\subset \Omega$, we have
$$\int_U (J_{f_j})^{\frac 1 2}\, \dd m(z) \leq \mathscr L^2(U)^{\frac 1 2}  \left(\int_{\Omega}  J_{f_j} \, \dd m(z)\right)^{\frac 1 2}  \leq C \mathscr L^2(U)^{\frac 1 2},$$
where $C=\mathscr L^2(g(\Omega))$ in case (\ref{it:trace}) and $C=\mathscr L^2(\mb D(0,2))$ in case (\ref{it:principal}), by the area formula.
Thus the second term is also equiintegrable.
\end{proof}

\begin{remark}
\label{rmk:explicitKO}
A small variant of the above proof, combined with Proposition \ref{prop:KOR}, gives the estimate
$$\int_\Omega |\log(J_f(z))|^q  \, \dd m(z) \leq C(g,\Omega)\left(1+\int_\Omega K_f(z)^q\, \dd m(z)\right)$$
for homeomorphisms $f\in g + \WW^{1,1}_0(\Omega)$ such that $K_f\in \LL^q(\Omega)$, where $q\geq 1$.
\end{remark}

For the proof of Theorem \ref{thm:lscW} we also need some closed-quasiconvexity features for the $\mathscr W$-functional.

\begin{proposition}\label{thm:closedqc.new}
Let $\nu \in \mathscr M^{2}_\tp{qc}(\R^{2\times 2}_+)$ be a gradient Young measure  
generated by a bounded sequence $(\psi_j)\subset \WW^{1,{2}}(\mb D)$ of homeomorphisms such that 
$$\|K_{\psi_j}\|_{\LL^{q}(\mb D)}\leq C,$$ 
for some $q > 1$. Then 
$$ \mathscr W(\langle \nu, \Id \rangle) \leq \int_{\R^{2\times 2}} \mathscr W(A) \,\dd \nu(A).
$$ 
\end{proposition}

\begin{proof}  
To start with, momentarily assume that in conformal coordinates
\begin{equation}
\label{eq:normalizationnuz0}
\langle \nu, \Id \rangle = A; \quad A(z) = z + a \overline{z},  \quad {\rm with} \; \; |a|<1.
\end{equation}
Applying Theorem \ref{thm:intdistort}, we find a sequence of maps $f_j\colon \C \to \C$ such that:
\begin{enumerate}
\item $f_j$ are principal maps;
\item $(f_j)\subset \WW^{1,2}_\tp{loc}(\C)$ is bounded and $(f_j|_{\mb D(0,r)})$ generates $\nu$ for all $r<1$;
\item $\psi_j = h_j\circ f_j$ for some conformal maps $h_j\colon f_j(\mb D) \to \psi_j(\DD)$.
\end{enumerate}
In particular, we have $K_{f_j}=K_{\psi_j}$ a.e.\ in $\mb D$ and so by Lemma \ref{lemma:equiintegrability} the sequence $(\mathscr W(\D f_j))$ is again equiintegrable over $\mb D$. Thus
\begin{align*}
\int_{\M} \mathscr W(A) \,\dd \nu(A)
&=\lim_{j\to \infty} \fint_{\mb D(0,R)} \mathscr W(\D f_j(z))\,\dd m(z)\\
& \to\lim_{j\to \infty} \fint_{\mb D} \mathscr W(\D f_j(z))\,\dd m(z) \ge \lim_{j\to \infty} \mathscr W(A_{f_j}) = \mathscr W(\langle \nu, \Id \rangle),
\end{align*}
as $r\nearrow 1$, where the  inequality follows from Theorem \ref{thm:Warea} and the last identity from \eqref{eq:asympt} and Remark \ref{affine}. 

To deal with the general case where we do not have \eqref{eq:normalizationnuz0}, note that since $\nu \in \mathscr M^{2}_\tp{qc}(\R^{2\times 2}_+)$
one still has $\det(\langle \nu, \Id \rangle) > 0$. Thus we simply replace $\nu$ by a suitable normalized measure.
Namely  if $\langle \nu, \Id \rangle = A \in \R^{2\times 2}_+$, 
choose $t>0$ and $Q\in \tp{SO}(2)$ such that
$$ t Q A (z) = z + a \bar z, \qquad |a| < 1. 
$$
Here notice that $\mathscr W$ satisfies
$$\mathscr W(QA)=\mathscr W(A) \qquad \tp{and} \qquad \mathscr W(tA)=\mathscr W(A)+\log t^2$$
for all $Q\in \tp{SO}(2)$ and $t>0$. Hence, if in the notation of Lemma \ref{lemma:invYM} one defines
$\mu \equiv \left(\nu_{(Q,\Id)}\right)_t$, then $\mu$ satisfies \eqref{eq:normalizationnuz0} so that 
\begin{align*}
\langle \nu, \mathscr W\rangle +\log t^2
= \langle \nu, \mathscr W(t\cdot)\rangle 
= \langle \mu, \mathscr W \rangle 
\ge \mathscr  W(\langle  \mu,\Id \rangle) = \mathscr W(\langle  \nu, \Id \rangle )+ \log t^2.
\end{align*}
This completes the proof. 
\end{proof} 

We are now ready for the main result of this Subsection.

\begin{proof}[Proof of Theorem \ref{thm:lscW}]  Given a homeomorphism $g\in \WW^{1,2}_\tp{loc}(\C)$ and a sequence $(f_j)\subset g + \WW^{1,2}_0(\Omega)$ such that 
$f_j\weak f$ in $\WW^{1,2}(\Omega)$ and  $\|K_{f_j}\|_{\LL^{q}(\Omega)} \leq C < \infty$ for some $q>1$, we are to show that 
\begin{equation}
\label{eq:Wwlsc}
\liminf_{j\to \infty} \int_\Omega {\mathscr W}(\D f_j(z)) \, \dd m(z) \geq \int_\Omega {\mathscr W}(\D f(z))\, \dd m(z).
\end{equation}

For this we follow a similar strategy as in Proposition \ref{prop:lscBp}.  First note that by Remark \ref{rmk:automatichomeo} the sequence $f_j$ consists of homeomorphisms.
Also, by passing to subsequences we may assume that $f_j \weak f$ in $\WW^{1,2}(\Omega)$ and, by Theorem \ref{thm:fundYM},  that $f_j$ generates the $\WW^{1,2}$-gradient Young measure $(\nu_z)_{z\in \Omega}$. 

Since $(\mathscr W(\D f_j))$ is equiintegrable by Lemma \ref{lemma:equiintegrability}, we have
\begin{equation}
\label{eq:repW}
\lim_{j\to \infty}\int_\Omega \mathscr W(\D f_j)\,\dd m(z)=
\int_\Omega \int_{\R^{2\times 2}} \mathscr W(A) \,\dd \nu_z(A) \, \dd m(z).
\end{equation}
Moreover, for a.e.\ $z\in \Omega$ 
\begin{equation}
\label{eq:conditionsbarycenter}
\det \langle  \nu_z, \Id \rangle  >0\quad \tp{and} \quad \nu_{z}\in \Meas^2_\tp{qc},
\end{equation}
where the former claim follows by arguing similarly to \eqref{eq:FK}, but with $F_K$ replaced with $\min\{0,\det\}$, and the latter claim by Proposition \ref{prop:homogenization}. 

Let us then fix a point $z_0\in \Omega$ for which \eqref{eq:conditionsbarycenter} holds. The measure $\nu_{z_0}$ is generated by taking a diagonal subsequence $(\psi_j)$ of $\psi_{j,\lambda}(z)=\lambda^{-1} f_j(z_0+\lambda z)$, where $j\to \infty$ and $\lambda\to 0$, cf. ~Remark \ref{rmk:homogenization}. In particular,  $\psi_j\colon \mb D \to \C$ is a sequence of homeomorphisms such that $\sup_j\|K_{\psi_j}\|_{\LL^q(\mb D)} <\infty$. 

We can now apply Proposition \ref{thm:closedqc.new} which says that
\begin{equation}
\label{Wqconv2}
\int_{\M} \mathscr W(A) \,\dd \nu_{z_0}(A)
 \ge \mathscr W(\langle  \nu_{z_0}, \Id \rangle ).
\end{equation}
On the other hand, by \eqref{limit} we have $\langle  \nu_{z}, \Id \rangle = \D f(z) $ for a.e.\ $z \in \Omega$. Combining this  with \eqref{eq:repW} and \eqref{Wqconv2} completes the proof.
\end{proof} 

It remains an interesting open question whether \eqref{eq:Wwlsc} still holds at $q =1$, i.e.\ under the natural assumption $\|K_{f_j}\|_{\LL^{1}(\Omega)} \leq C < \infty$.

\subsection{Existence of minimizers in Nonlinear Elasticity}\label{sec:elasticity}
As a last theme, let us collect the  previous results to prove existence of minimisers for a quite large family of functionals in Nonlinear Elasticity. We again emphasize that these are non-polyconvex, see Example \ref{nonpoly}.

First, as is well-known, polyconvex functionals are lower semicontinuous, see e.g.~\cite{DacorognaMarcellini}. Thus  combining Theorems \ref{thm:lscW} and \ref{thm:decomposition}   gives:  

\begin{corollary}
\label{cor:lscsplit}
Let $g\in \WW^{1,1}_\tp{loc}(\C)$ be a homeomorphism with $K_g\in \LL^q_\tp{loc}(\C)$ for some $q>1$ and  let ${\bf E}\colon \R^{2\times 2}_+\to \R$ be a functional as in Theorem \ref{thm:decomposition},
$${\bf E}(A) = {\bf G}(\det A)+ {\bf H}(K_A),
$$
 which we assume to be rank-one convex.  

If $f_j \weak f$ in $g + \WW^{1,2}_0(\Omega)$ and $\sup_j \|K_{f_j}\|_{\LL^q(\Omega)}<\infty$  then
$$\liminf_{j\to \infty} \int_\Omega {\bf E}(\D f_j(z)) \, \dd m(z) \geq \int_\Omega {\bf E}(\D f(z)) \, \dd m(z).$$
\end{corollary}

In conclusion, to promote the lower semicontinuity to  the existence of minimizers requires now
 some form of coercivity, and this takes us to the following examples.

\begin{corollary}\label{thm:miniselasticity}
Suppose  $q > q_0 \geq 1$ with $p\geq 2$, and let
$${\bf E}(A)\equiv {\bf G}(\det A)+ {\bf H}(K_A) + |A|^p,$$
where we assume that for some $C>0$,
\begin{enumerate}
\item $A\mapsto {\bf G}(\det A)+{\bf H}(K_A)$ is rank-one convex;
\smallskip

\item $|{\bf G}(t)| \leq C (1+ |\log(t)|^{q_0} )$;
\smallskip

\item ${\bf H}$ is convex and ${\bf H}(t)\geq  t^q/C - C $. 
\end{enumerate}
\smallskip

\noindent Then for any homeomorphism $g\in \WW^{1,1}_\tp{loc}(\C)$ with $K_g\in \LL^q_\tp{loc}(\C)$ there is a minimizer $f\in \WW^{1,p}(\Omega)$ of the problem
$$\inf\left\{\int_\Omega {\bf E}(\D h(z))\, \dd m(z) : h \in g + \WW^{1,p}_0(\Omega)\right\}.$$
In addition, $f$ is a homeomorphism such that $f^{-1}\in \WW^{1,2}(g(\Omega))$.
\end{corollary}

\begin{proof} 
By our hypothesis we have the lower bound
\begin{align*}
\int_\Omega {\bf E}(\D f(z)) \, \dd m(z) & \geq 
\int_\Omega |\D f(z)|^p \, \dd m(z) +  \frac{1}{C} \int_\Omega K_f(z)^q \, \dd m(z)\\
& \qquad  - C \int_\Omega |\log(J_f(z))|^{ q_0} \, \dd m(z)  - 2C|\Omega| 
\end{align*}
for any orientation-preserving map $f\in \WW^{1,p}_g(\Omega)$. 

Next, 
for any $\varepsilon > 0$ and $q > q_0 \geq 1$, we have  $C x^{q_0} \leq \varepsilon x^q + M$ where  the constant $M = M(\varepsilon,C,q_0/q) < \infty$. 
With Remark \ref{rmk:explicitKO} this gives us  
$$C \int_\Omega |\log J_f(z)|^{ q_0} \, \dd m(z) \leq \frac{1}{2C} \int_\Omega K_f(z)^q \,\dd m(z)+ C_1(M, g,\Omega).$$
That is, we have the coercivity
\begin{equation*}
\label{eq:coercivity}
\int_\Omega {\bf E}(\D f) \, \dd m(z) \geq 
 \int_\Omega |\D f|^p \, \dd m(z) +  \frac{1}{2C} \int_\Omega K_f(z)^q \,\dd m(z) - C_2(M, g,\Omega).
\end{equation*}
The result is now a consequence of the Direct Method: 
we take a minimizing sequence $(f_j)\subset g + \WW^{1,2}_0(\Omega)$ 
and by the last estimate we  have 
$$\sup_j\|\D f_j\|_{\LL^p(\Omega)} <\infty, \qquad  \sup_j\|K_{f_j} \|_{\LL^q(\Omega)} <\infty.$$ Since $p\geq 2$, up to a subsequence we have $f_j\weak f$ in $\WW^{1,2}(\Omega)$ and hence Corollary \ref{cor:lscsplit} yields 
$$\liminf_{j\to \infty} \int_\Omega {\bf E} ( \D f_j(z)) \, \dd m(z) \geq \int_\Omega {\bf E}(\D f(z)) \, \dd m(z),$$
thus $f$ is a minimizer.  Since $f\in g + \WW^{1,p}_0(\Omega)$ is a map of integrable distortion, we have $f^{-1}\in \WW^{1,2}(g(\Omega))$.
\end{proof}

The main point in the proof of Corollary \ref{thm:miniselasticity} is that, once the condition $\int_\Omega {\bf E}(\D f) \,\dd m(z)<\infty$ imposes bounds on the $\LL^p$ norms of both $\D f$ and $K_f$, the existence of minimizers follows from the sequential lower semicontinuity result of Corollary \ref{cor:lscsplit}, which in turn is essentially a consequence of Theorem \ref{thm:lscW}. In particular, it is not difficult to write more general versions of Corollary \ref{thm:miniselasticity}: for instance,  to ${\bf E}$ as in Corollary \ref{thm:miniselasticity} one can add a polyconvex term ${\bf P}\geq 0$, as well as a quasiconvex term ${\bf Q}$ satisfying $0\leq {\bf Q}\leq c(1+|\cdot|^p)$ for some $c \geq 0$, and still obtain existence of minimizers. Instead of pursuing the maximal degree of generality, we give here a simple, concrete example,  where we add a neo-Hookean term to $\mathscr W$, and leave more complicated examples to the interested reader:

\begin{example} \label{nonpoly} For $c\geq 0$, consider the functional ${\bf E}_c\colon \R^{2\times 2}_+\to \R$,
$${\bf E}_c(A) = \mathscr W(A) + c \left(\frac{1}{\det A} + |A|^2\right)^2.$$
Then the ${\bf E}$-energy admits minimizers, as in Corollary \ref{thm:miniselasticity}.
Moreover,  ${\bf E}_c$ is bounded from below and it satisfies \eqref{eq:blowup}. If $c$ is small enough then ${\bf E}_c$ is non-polyconvex.
\end{example}

\begin{proof}[Proof of the above claims]
The functional ${\bf E}_c$  satisfies the assumptions of Corollary \ref{thm:miniselasticity}, hence one obtains a minimizer for ${\bf E}_c$ in $\WW^{1,2}$.

It is not difficult to check that ${\bf E}_c$ is bounded from below; for instance,  
$$\mathscr W(A)+\frac{c}{(\det A)^2}\geq \frac 3 2 + \log(\sqrt{2 c}).$$
The blow-up condition \eqref{eq:blowup} is easy to verify.

We then show that ${\bf E}_c$ is non-polyconvex for small enough $c$. First, consider points
$$A_1 = \tp{diag}(3/10,3/10),\quad A_2=\tp{diag}(2/3, 8), \quad A_3 = \tp{diag}(8,2/3).$$
These points satisfy the so-called minors relations, that is,
\begin{gather*}
\frac{100}{121} A_1 + \frac{21}{242} A_2 + \frac{21}{242} A_3 = \Id, \\
\frac{100}{121} \det A_1 + \frac{21}{242} \det A_2 + \frac{21}{242} \det A_3 = 1,
\end{gather*}
yet we have
$$\frac{100}{121} \mathscr W(A_1) + \frac{21}{242} \mathscr W(A_2 )+ \frac{21}{242} \mathscr W(A_3) \approx 0.78 < 1=\mathscr W(\Id),$$
which in particular shows that $\mathscr W$ is non-polyconvex.  Clearly we still have
$$\frac{100}{121} \mathscr {\bf E}_c(A_1) + \frac{21}{242} {\bf E}_c(A_2 )+ \frac{21}{242} {\bf E}_c(A_3)  < {\bf E}_c(\Id),$$
provided that $c$ is chosen sufficiently small.
\end{proof}

\begin{remark}\label{weehat} In the same token, the non-negative functional 
$$\widetilde{\,  \mathscr  W \, } (A) =  \frac{|A|^2}{\det A} - \log \left( \frac{|A|^2}{\det A}\right) + |\log \det A|$$ 
from \eqref{modiW} is quasiconvex but not polyconvex. 
Indeed, the functional is the sum of $ \mathscr W$ and a polyconvex functional,  thus  quasiconvex by Theorem \ref{thm:morreysplit}. Moreover, 
$\widetilde{\,  \mathscr  W \, } (A) =  \mathscr W(A)$ when $\det(A) > 1$. 

Hence we can use the above minors relations and multiply the $A_j$'s by a number $t >1$ so that each has determinant $> 1$.
As $\mathscr W(t A) = \mathscr W(A) + \log(t^2)$ with $ \mathscr W(t \Id) = 1 + \log(t^2)$, the new minors relations with the multiplied matrices
show that $\widetilde{\,  \mathscr W \,}$ is non-polyconvex. 

Also, similarly  as in Example  \ref{nonpoly} considering, say,  $\widetilde{\,  \mathscr  W \, }(A) + c(|A|^2 + K_A^2)$
gives a non-negative and non-polyconvex functional which admits  minimizers.
\end{remark}

Having established the existence of minimizers for the above class of functionals it is natural to inquire about their regularity properties, but
this appears to be a very difficult problem.  Away from perturbative regimes,  almost nothing is known concerning regularity of minimizers in
nonlinear elasticity, even for polyconvex functionals, but see \cite{Bauman,IKO2} for some interesting results.

\bibliographystyle{amsplain}

\end{document}